\documentclass[12pt]{amsart}
\usepackage[margin=1in]{geometry}
\usepackage[english]{babel}
\usepackage[utf8]{inputenc}
\usepackage{subcaption}
\usepackage{tabularx}
\usepackage{amsmath}
\usepackage{amssymb}
\usepackage{amsfonts}
\usepackage{amsthm}
\usepackage{comment}
\usepackage{mathrsfs}
\usepackage[all]{xy}
\usepackage{capt-of}
\usepackage[pdftex]{graphicx}
\usepackage{color}
\usepackage[noadjust]{cite}
\usepackage{url}
\usepackage{indent first}
\usepackage[labelfont=bf,labelsep=period,justification=raggedright]{caption}
\usepackage[english]{babel}
\usepackage[utf8]{inputenc}
\usepackage{hyperref}
\usepackage[colorinlistoftodos]{todonotes}
\usepackage{tkz-fct}
\usepackage{tikz}
\usetikzlibrary{calc}
\usepackage{array}
\usepackage{booktabs}
\usepackage{enumerate}
\newcolumntype{L}[1]{>{\raggedright\let\newline\\\arraybackslash\hspace{0pt}}m{#1}}
\newcolumntype{C}[1]{>{\centering\let\newline\\\arraybackslash\hspace{0pt}}m{#1}}
\newcolumntype{R}[1]{>{\raggedleft\let\newline\\\arraybackslash\hspace{0pt}}m{#1}}
\DeclareMathOperator{\Char}{char}

\DeclareMathOperator{\disc}{disc}

\DeclareMathOperator{\MaxSpec}{MaxSpec}

\DeclareMathOperator{\ord}{ord}

\DeclareMathOperator{\Per}{Per}

\DeclareMathOperator{\Res}{res}
\DeclareMathOperator{\res}{res}

\DeclareMathOperator{\Spec}{Spec}

\newcommand{\isom}{\cong}
\newcommand{\eps}{\varepsilon}

\renewcommand{\AA}{\mathbb{A}}
\newcommand{\CC}{\mathbb{C}}
\newcommand{\DD}{\mathbb{D}}

\newcommand{\FF}{\mathbb{F}}

\newcommand{\NN}{\mathbb{N}}

\newcommand{\PP}{\mathbb{P}}
\newcommand{\QQ}{\mathbb{Q}}
\newcommand{\RR}{\mathbb{R}}
\newcommand{\ZZ}{\mathbb{Z}}

\newcommand{\mcF}{\mathcal{F}}

\newcommand{\mcJ}{\mathcal{J}}
\newcommand{\mcM}{\mathcal{M}}

\newcommand{\mcO}{\mathcal{O}}

\newcommand{\mcR}{\mathcal{R}}

\newcommand{\mfX}{\mathfrak{X}}

\newcommand{\mfm}{\mathfrak{m}}

\newcommand{\mfp}{\mathfrak{p}}

\newcommand{\mfu}{\mathfrak{u}}
\newcommand{\mfv}{\mathfrak{v}}
\newcommand{\mfw}{\mathfrak{w}}

\newcommand{\mc}[1]{\mathcal{#1}}
\newcommand{\ol}[1]{\overline{#1}}
\newcommand{\mf}[1]{\mathfrak{#1}}

\theoremstyle{plain}
\newtheorem{thm}{Theorem}
\newtheorem{thmx}{Theorem}

\newtheorem{lemma}[thm]{Lemma}
\newtheorem{cor}[thm]{Corollary}

\newtheorem{prop}[thm]{Proposition}

\newtheorem{qn}[thm]{Question}

\theoremstyle{definition}
\newtheorem{defn}[thm]{Definition}

 \newtheorem{conjecture}[thm]{Conjecture}

\theoremstyle{remark}
\newtheorem{rem}[thm]{Remark}
\newtheorem{ex}[thm]{Example}

\numberwithin{equation}{section}
\numberwithin{thm}{section}

\begin{document}

\title{Reduction of dynatomic
 curves}
\author[Doyle]{John R. Doyle}
\address{John R. Doyle: University of Rochester, Rochester, NY 14627, USA}
\email{john.doyle@rochester.edu}
\author[Krieger]{Holly Krieger}
\address{Holly Krieger:
University of Cambridge,
Cambridge, CB3 0WB, England}
\email{hkrieger@dpmms.cam.ac.uk}
\author[Obus]{Andrew Obus}
\address{Andrew Obus: University of Virginia, Charlottesville, VA 22904, USA}
\email{obus@virginia.edu}
\author[Pries]{Rachel Pries}
\address{Rachel Pries: Colorado State University, Fort Collins, CO 80523, USA}
\email{pries@math.colostate.edu}
\author[Rubinstein-Salzedo]{Simon Rubinstein-Salzedo}
\address{Simon Rubinstein-Salzedo: Euler Circle, Palo Alto, CA 94306, USA}
\email{simon@eulercircle.com}
\author[West]{Lloyd West}
\address{Lloyd West: University of Virginia, Charlottesville, VA 22904, USA}
\email{lww8k@virginia.edu}
\date{\today}

\begin{abstract} 
The dynatomic modular curves $Y_1(n)$ parametrize polynomial maps together with a point of period $n$.
It is known that $Y_1(n)$ is smooth and irreducible in characteristic 0 for families of polynomial maps of the form $f_c(z) = z^m +c$ where $m\geq 2$. 
In the present paper, we build on the work of Morton to
partially characterize the primes $p$ for which the reduction modulo $p$ of $Y_1(n)$ remains smooth and/or irreducible. 
As an application, we give new examples of good reduction of $Y_1(n)$ for several primes dividing the ramification discriminant when $n=7,8,11$.
The proofs involve arithmetic and complex dynamics, reduction theory for curves, ramification theory, and the combinatorics of the Mandelbrot set.
\end{abstract}



\maketitle

\section{Introduction}

We study the reduction modulo $p$ of the dynatomic modular curves, which parametrize polynomial maps together with a point of period $n$. 

For a self-morphism $f:\PP^1_K\to \PP^1_K$, defined over a field $K$, denote the $n$-fold iterate by $f^n$. A \emph{$K$-point of period $n$} for $f$ is a point $P\in \PP^1(K)$ such that $f^{n}(P) =P$; the set of such is denoted $\Per_n(f)(K)$. 

Consider a polynomial $f(x,c)\in K[x,c]$. This defines a one dimensional family $\mcF=\{f_c\}$ of self-morphisms of the projective line, $f_c:\PP^1_K \to \PP^1_K; \, x \mapsto f_c(x)=f(x,c)$. The most famous of these 
is the family of quadratic polynomial maps $f_c(x)=x^2+c$. Given such a family $\mcF$, there is an algebraic curve $Y^\mathrm{dyn}_{1,\mcF}(n)$, defined over $K$, whose $K$-points correspond to pairs $(f_c, P)$ consisting of a map $f_c$ with $c\in K$ together with a point $P\in \Per^*_n(f_c)(K)$ of (formal) period $n$, as defined in Definition \ref{defn:formal_period}. 

There is a branched cover \begin{eqnarray*}
\pi_1 : \, Y^\mathrm{dyn}_{1,\mcF}(n)&\to &\AA^1\\  (f_c,P)&\mapsto & c.
\end{eqnarray*} The curves $Y^\mathrm{dyn}_{1,\mcF}(n)$ are analogous to the classical modular curves $Y^\mathrm{ell}_1(N)$ in the case of torsion points on elliptic curves; accordingly they are called \emph{dynatomic modular curves}.  For ease of notation, we drop the superscript ``$\mathrm{dyn}$''  
and omit the reference to the family $\mcF$ where this causes no ambiguity. The curve classifying \emph{orbits} of (formal) period $n$ is denoted $Y_0(n)$; it is a $\ZZ/n\ZZ$-quotient of $Y_1(n)$, giving a tower of covers $Y_1(n)\to Y_0(n)\to \AA^1$. 

The curves $Y_1(n)$ are smooth and irreducible in characteristic 0 for families of the form $f_c(x) = x^m +c$ where $m\geq 2$ (\cite{bousch:thesis}, \cite{Morton96}). The chief purpose of the present paper is to work toward a characterization of the primes $p$ for which the reduction modulo $p$ of $Y_1(n)$ remains smooth and/or irreducible.

One motivation for this project is the conjecture of Morton and Silverman on uniform bounds for the number of periodic points (\cite{MortonSilverman94}): in \S\ref{section:UBC} we explain how the function field case of the conjecture could be deduced from sufficient information about the reduction of $Y_1(n)$. 

In \cite{Morton96}, Morton defines a discriminant $D_n$ (see Definition \ref{Ddeltan}), whose list of prime factors contains all the primes of bad reduction for $Y_1(n)$; our computational data indicates, however, that $Y_1(n)$ has good reduction at many of the primes dividing $D_n$. For example, for the family $f_c(x) = x^2 +c$, the discriminant $D_5$ has 7 odd prime factors (3, 5, 11, 31, 3701, 4217 and 86131), but only two of them (5 and 3701) are primes of bad reduction for $Y_1(5)$.

Here is a brief summary of the paper.  A more complete summary of the results is contained in \S \ref{Summary}.
The necessary dynamical background is contained in \S \ref{SDynatomicCurves}. 
In \S \ref{Sprelims}, we give some useful lemmas for detecting good and bad reduction.  In \S \ref{Sgood}, we reprove Morton's necessary criterion for bad primes, and we also give results relating the reduction of $Y_1(n)$ to that of $Y_0(n)$. In \S \ref{Sbad}, 
for the family $f_c(x) = x^m +c$, we prove that $Y_1(n)$ has bad reduction at any prime that does not divide $m$ and either divides $n$ or divides $D_n$ exactly once. When $n=5$, the primes $5$ and $3701$ in the above example fall into precisely these two cases. 

In \S \ref{Sirreducibility}, 
we consider the quadratic polynomial family $f_c(x)=x^2+c$ and a prime $p$ that divides $D_n$ exactly once. 
Although the reduction of $Y_0(n)$ may be singular in this case, we prove that it is nonetheless geometrically irreducible. 
This result is obtained from a study of the connectedness of the ``monodromy graph'' associated to $Y_0(n)$ in characteristic 0, which we defer to \S \ref{Smonodromy}; this is a result of independent interest in complex dynamics because of its connection to combinatorics of the Mandelbrot set.  
For applications to the uniform boundedness conjecture, geometrically irreducible reduction is as good as good reduction.

In \S \ref{Sc0c2}, we provide a conceptual explanation why $Y_1(n)$ can have good reduction modulo certain primes that divide $D_n$ and use this to give new examples of good reduction of $Y_1(n)$ for several primes dividing the ramification discriminant when $n=7,8,11$.

Data is included in Appendices \ref{Sdata} and \ref{Tmono}.

\section*{Acknowledgments}

\thanks{
We would like to thank the American Institute of Mathematics for hosting the workshop {{\it Galois theory of orbits in arithmetic dynamics}} where we started this collaboration. We would also like to thank Joseph Gunther, Joseph Silverman and Michael Stoll for helpful comments. Krieger's research was partially supported by NSF grant DMS-1303770. Obus's research was partially supported by NSF grants DMS-1265290 and DMS-1602054. Pries's research was partially supported by NSF grant DMS-1502227.  
}

\section{Organization and summary of main results} \label{Summary}

In this section, for simplicity,
we assume that 
$f_c(x) = x^m +c$. In \S \ref{SDynatomicCurves} we recall the definition and key properties of the dynatomic curves $Y_0(n)$ and $Y_1(n)$. For now, note that $Y_1(n)$ has a natural model that is smooth in characteristic $0$; see Definition \ref{Dquotient} for the case of $Y_0(n)$. The term \emph{good reduction modulo $p$} means that the corresponding curve is smooth in characteristic $p$. 

Following Morton \cite{Morton96}, there is a natural choice of polynomial $\delta_n(1,c)$ whose roots are the branch points of the cover $\pi_1: \, Y_1(n)\to \AA^1$. This polynomial has a factorization over $\ZZ$ of the form $\delta_n(1,c) = \prod_{\substack{d\mid n }} \Delta_{n,d}(c)$ such that:
for $d\mid n$ and $d \neq n$, the roots of $\Delta_{n,d}(c)$ are parabolic parameters
of formal period $n$ and exact period $d$, i.e. parameters for which an orbit of primitive period $n$ degenerates to an orbit of primitive period $d$; and
the roots of $\Delta_{n,n}(c)$ are primitive parabolic parameters, i.e. parameters for which two orbits of exact period $n$ collide.
See Definition~\ref{defn:factors} for details.
Writing $D_{n,d} := \disc \Delta_{n,d}(c)$ and $R_{n,d,e} := \res (\Delta_{n,d}(c),\Delta_{n,e}(c) )$, then the discriminant $D_n := \disc\, \delta_n(1,c)$
has a factorization 
\begin{equation*}
D_n= \prod_{\substack{d\mid n }} D_{n,d}\cdot \prod_{\substack{d\mid n, e\mid n \\ d < e}} R_{n,d,e}^2.
\end{equation*}

\subsection{Bad reduction}

We prove several results characterizing primes of bad reduction.
First, if $n > 3$, $p \mid n$, and $p \nmid m$, 
we prove that $Y_1(n)$ has bad reduction at $p$ in Proposition \ref{Ppdividesnbadreduction}.
Here is the main result about bad reduction:

\begin{thmx}\label{THM:v(disc)=1_implies_bad} (Theorem \ref{Tbad})
Let $p \nmid m$ be prime. 
If $\ord_p D_{n,n} = 1$, then the curve $Y_0(n)$ has bad reduction.
If $\ord_p D_n = 1$, then the curve $Y_1(n)$ has bad reduction.
\end{thmx}

\begin{rem}
In computational examples with $n\leq 6$ and $f_c(x) = x^2 + c$ the only odd primes of bad reduction 
for $Y_0(n)$ are precisely those accounted for in Theorem \ref{THM:v(disc)=1_implies_bad}.  Further computations with $n=7$ and $n=8$ confirm this observation for many small prime factors of $D_{n,n}$; namely for the primes appearing in Tables~\ref{table:data1-7}--\ref{table:data8}.  This explains the bad reduction of $Y_0(n)$ at ``mysterious large primes,'' such as $p = 3701$ for $n = 5$ and $p = 8029187$ for $n = 6$, though it does \emph{not} give a straightforward recipe for computing these primes! 
\end{rem}

Despite the fact that $Y_0(n)$ has singular reduction at primes that exactly divide the discriminant $D_{n,n}$, we prove the following theorem.

\begin{thmx}\label{THM:divides_once_then_irreducible} (Corollary \ref{CX0irred})
If $f_c(x) = x^2 + c$, and if an odd prime $p$ divides $D_{n,n}$ exactly once, then $Y_0(n)$ is geometrically irreducible modulo $p$.
\end{thmx}
This theorem is deduced -- using the theory of the \'etale fundamental group in mixed characteristic -- from a result on the connectedness of the ``monodromy graph'' in characteristic 0. The result on the monodromy graph is proved in \S \ref{Smonodromy} using the combinatorics of the Mandelbrot set. 

In fact, if one can show that all odd primes of bad reduction of $Y_0(n)$ are explained by Theorem \ref{THM:v(disc)=1_implies_bad}, then one can use Theorem \ref{THM:divides_once_then_irreducible} to obtain uniform bounds on periodic points of $x^2 + c$ over function fields (see Propositions \ref{prop:UBCfromIrreducibility} and \ref{Puniformboundedness}).  We caution the reader that we have very limited evidence that Theorem \ref{THM:v(disc)=1_implies_bad} explains all of these primes, although it does for $n \leq 8$.

\subsection{Good reduction}

If $p \nmid m$ is a prime of bad reduction for $Y_1(n)$, Morton proved that $p \mid D_n$ \cite[Theorem 15]{Morton96}. 
We reprove and strengthen this result using general techniques relating reduction of branched covers to the geometry of the branch locus.

\begin{thmx}\label{THM:good_away_from_disc} (Proposition \ref{prop:Pbadimpliesdisc})
The primes $p\nmid m$ of bad reduction for the curve $Y_1(n)$ divide $D_n$. 
The primes $p\nmid m$ of bad reduction for $Y_0(n)$ divide $D_{n,n}$.
\end{thmx}

\begin{rem}
We show in Propositions \ref{PRnne} and \ref{P12} that the primes of bad reduction for $Y_1(n)$ in fact divide a smaller integer $D^{\prime}_n$ with $D^{\prime}_n\mid D_n$.
\end{rem}
 
\subsection{Comparison of $Y_1(n)$ and $Y_0(n)$}

We prove the following result connecting the geometry of $Y_1(n)$ with that of $Y_0(n)$. 

\begin{thmx}\label{THM:enough_to_know_X0_good}
(Propositions \ref{PX1goodX0good} and \ref{PX1irredX0irred})
Suppose $p \nmid nm(m-1)$.
\begin{enumerate}
\item If $n$ is prime, then $Y_1(n)$ has good reduction modulo $p$ if and only if 
$Y_0(n)$ has good reduction modulo $p$.
\item The curve $Y_1(n)$ is geometrically irreducible modulo $p$
if and only if $Y_0(n)$ is geometrically irreducible modulo $p$.
\end{enumerate}
\end{thmx}

\begin{rem}
In Theorem \ref{THM:enough_to_know_X0_good}:
\begin{enumerate}
\item Part (1) is false when $n$ is composite;
see, for example, $n=6$ and $p=67$ in
Remark \ref{rem:false_when_composite}. 

\item Part (2) 
is a strengthening of \cite[Proposition 17]{Morton96}.

\item Combining part (2) with Theorem \ref{THM:good_away_from_disc} shows that the reduction of $Y_1(n)$ modulo $p$ is geometrically irreducible modulo $p$ if $p\nmid nm(m-1) D_{n,n}$.
\end{enumerate}
\end{rem}

\subsection{New applications for good reduction}

In \S \ref{Sc0c2}, we restrict to the case $f_c(x)=x^2+c$ and analyze the reduction of $Y_0(n)$ 
above the points $c=0$ and $c=-2$. These values of $c$ are special because they yield maps that come from algebraic groups, namely power and Chebyshev maps, respectively.
We develop a geometric strategy to show that, for particular $n$ and $p$:
\begin{enumerate} \item $Y_0(n)$
does not obtain a singularity above $c=0$ or $c=-2$ when reduced modulo $p$,
\item and the power of $p$ dividing $D_{n,n}$ is fully explained by branch points colliding with $c=0$ or $c=-2$ modulo $p$. 
\end{enumerate}

This allows us to verify certain primes of good reduction computed in \cite{MortonX4}, \cite{FPS}, \cite{Stoll} when $n=4,5,6$.
Furthermore, in Theorem \ref{corol:Examples},
we prove that 
$Y_0(n)$ has good reduction modulo $p$, even 
though $p$ divides $D_{n,n}$ in the 
following new cases:
 
$Y_0(7)$ for $p=3,43,127$,

$Y_0(8)$ for $p=3,5,17,257$, and

$Y_0(11)$ for $p=3,23,89,683$.

\section{Dynatomic modular curves}\label{SDynatomicCurves}

In this section, we review facts about the moduli spaces for polynomial maps with periodic points.  Let $K$ be a field.

\subsection{Periodic points, dynatomic polynomials and dynatomic curves}

We begin by recalling the following standard definitions from dynamics (\cite[Chapter 4]{SilvermanBook}):

\begin{defn} Let $f: \PP^1 \to \PP^1 $ be a rational map defined over a field $K$. 
\begin{enumerate}
\item Say $P\in \PP^1(K)$ is a \emph{point of period $n$} for $f$ if $f^n(P) =P$. Denote the set of all $K$-points of period $n$ for $f$ by $\Per_n(f)(K)$. Also write $\Per(f)(K)=\bigcup_{n=1}^\infty \Per_n(f)(K)$; a point $P\in\Per(f)(K)$ is said to be a \emph{periodic point}.
\item For $P\in \Per_n(f)(K)$  the \emph{$n$-th multiplier} is defined to be $\lambda_n(P)=(f^n)^\prime(P)$. 
\item Say $P\in \PP^1(K)$ is a  \emph{point of primitive period $n$} for $f$ if $f^n(P) =P$ and $f^m(P)\neq P$ for $1\leq m<n$. 
\end{enumerate}
\end{defn}

\begin{rem}
By the chain rule, any two periodic points in the same orbit have the same multiplier, so it also makes sense to define the multiplier $\lambda_n(\mcO)$, where $\mcO$ is an {\it orbit} of periodic points. If $K=\CC$, one calls an $n$-cycle $\mcO$ {\it attracting}, {\it indifferent}, or {\it repelling} according to whether $|\lambda_n(\mcO)|$ is less than 1, equal to 1, or greater than 1, respectively.
\end{rem}

We now turn to the construction of the dynatomic curves $Y_1(n)$. Let $\mcF$ be a family of polynomials of the form $f_c(x)=f(x,c)\in \ZZ[x,c]$. Then, for any positive integer $n$, let $\Psi_{\mcF,n}(x,c) = f_c^n(x)-x \in \ZZ[x,c]$. A point $(z,b)$ on the variety defined by $\Psi_{\mcF,n}(x,c)$ in $\AA^2$ corresponds to a polynomial $f_b(x)$ together with a point $z$ of period $n$ for $f_b$. However, this variety is reducible, since $\Psi_{\mcF,d}(x,c)$ is a factor of $\Psi_{\mcF,n}(x,c)$ whenever $d\mid n$. 

Therefore one defines
\[
\Phi_{\mcF,n}(x,c)=\prod_{d\mid n} \Psi_{\mcF,d}(x,c)^{\mu(n/d)}.
\] 
By M\"obius inversion, there is a factorization 
\[
\Psi_{\mcF,n}(x,c)=\prod_{d\mid n} \Phi_{\mcF,d}(x,c).
\] 
It can be shown that $\Phi_{\mcF,n}(x,c)$ is a polynomial (see \cite[Theorem 4.5]{SilvermanBook}); it is called the \emph{$n$-th dynatomic polynomial for the family $\mcF$}.

\begin{defn} The \emph{dynatomic modular curve} is the subscheme ${Y}_{\mcF,1}(n)$ of $\AA^2_\ZZ$  defined by 
\[
\Phi_{\mcF,n}(x,c) = 0.
\]
\end{defn}

We shall generally omit the reference to the family $\mcF$ where unambiguous, writing $Y_1(n)$ for $Y_{\mcF,1}(n)$. We also use the shorter term ``dynatomic curve'' synonymously with ``dynatomic modular curve''. Given a scheme $S$, write $Y_1(n)_S$ for ${Y}_1(n)\times_\ZZ S$.  If $S = \Spec R$, we write $Y_1(n)_R$ for short. 

\begin{defn}
The \emph{reduction modulo $p$} of $Y_1(n)$ is the fiber $Y_{1}(n)_{\FF_p}$ of $Y_{1}(n)$ over the prime $p$ of $\ZZ$.
\end{defn}

In general, a point $(w,b)\in {Y}_1(n)(K)$ corresponds to a polynomial $f_b$ and a point $w\in K$ of primitive period $n$ for $f_b$. However, there are finitely many points where orbits of period $n$ degenerate to orbits of lower primitive period. One therefore adopts the following definition:

\begin{defn}\label{defn:formal_period} Let $f: \PP^1 \to \PP^1 $ be a polynomial map defined over $K$. Say $P\in \PP^1(K)$ is a   \emph{point of formal period $n$} for $f$ if $P$ is a point of primitive period $d$ where either
\begin{itemize}
\item $n=d$;
\item $n=ds$ and $\lambda_d(P)$ is a primitive $s$-th root of unity; or
\item $n=dsp^e$ and $\lambda_d(P)$ is a primitive $s$-th root of unity, $e\geq1$, and $p=\Char K$.
\end{itemize}
Denote the set of all $K$-points of formal period $n$ for $f$ by $\Per^*_n(f)(K)$.
\end{defn}

\begin{prop}[{\cite[Theorem 4.5]{SilvermanBook}}]\label{prop:moduli}
For a field $K$, the $K$-points of $
Y_1(n)_K$ are in bijection with the set 
\[
\left\{  (f_c, P)\; : \; c\in K, \; P\in\Per^*_n(f_c)(K)\right\}
\]
\end{prop}

\subsection{The moduli of periodic orbits}

There is a morphism $\pi_1: {Y}_1(n)\to\AA^1$ defined by projection onto the second coordinate
$c$.  The cover $\pi_1:\; {Y}_{1}(n)\to \AA^1$ has a cyclic subgroup of automorphisms $\langle \sigma \rangle \isom \ZZ /n \ZZ$, where $\sigma$ acts by $(x,\, c)\mapsto (f_c(x), \, c)$. Hence, $\langle \sigma \rangle$ acts on $Y_1(n)_S$ for any scheme $S$. The quotient under this action is a moduli space for pairs $(f_c, \mcO)$ consisting of a polynomial $f_c$ together with an orbit $\mcO$ of formal period $n$. Note that $\mcO$ may be defined over $K$ even if it contains points that are not. 

\begin{defn}\label{Dquotient}
For any ring $R$, let $Y_0(n)_R= Y_1(n)_R/\langle \sigma \rangle $.
\end{defn}

There are morphisms $\varphi$ and $\pi_0$ as in the diagram:
\[\xymatrix{& Y_1(n)\ar[dl]_{\ZZ/n\ZZ}^{\varphi}\ar[dd]^{\pi_1} \\ Y_0(n)\ar[dr]_{\pi_0} \\ & \AA^1}\]

\subsection{Completions of the dynatomic curves}

Let $K$ be a field. It is sometimes convenient in later proofs to work with a completion, $X_1(n)_K$, of $Y_1(n)_K$, which we now define. The completion that we use is the one that is regular at infinity.

A family of polynomials $\mcF$ is \emph{homogenous} if it is of the form $\mcF = \{ f(x,c)\}$, such that $f(x,u^r)\in \ZZ [x,u]$ is a homogenous polynomial of degree $m$ in $u$ and $x$, for some positive integer $r$, and such that $f(x,0)=x^m$. 
 
\begin{defn}\label{defn:standardModelX1} 

Let $\mcF$ be a homogenous family of polynomials. Let $K$ be a field such that $\disc f(x,1)\neq 0$ in $K$. Let $V$ be the singular locus of $Y_i(n)_K$, and let $W$ be the image of $V$ under $\pi_i:Y_i(n)_K\to \AA^1_K$ for $i \in \{0,1\}$. Embed $\AA^1_K$ in $\PP^1_K$ and let $Z_i(n)_K$ be the normalization of $\PP^1\setminus W$ in the total ring of fractions of $Y_i(n)_K$. Then we define $X_i(n)_K$ to be the result of gluing $Z_i(n)_K$ to $Y_i(n)_K$ along $Y_i(n)_K\setminus V$. 
\end{defn}

\begin{rem}
By \cite[Proposition 10]{Morton96}, under the hypotheses of the definition, the points at infinity in $X_1(n)_{\overline{K}} \setminus Y_1(n)_{\overline{K}}$ correspond to the periodic sequences $\{\zeta_j:\; j\geq 0\}$ that are of exact period $n$, where each $\zeta_j$ is a root of $f(x,1)=0$ in $\overline{K}$. 
\end{rem}

It is occasionally convenient to have a model of the curves $X_i(n)$ over a mixed characteristic dvr. We define one as follows:

\begin{defn}
Let $R$ be a characteristic $(0, p)$ discrete valuation ring with field of fractions $K$ and residue field $k$. Suppose that $\mcF = \{ f(x,c)\}$ is a homogenous family of polynomials with $p\nmid \disc f(x,1)$ and such that $Y_i(n)$ is smooth and irreducible in characteristic zero. Then we define $\mfX_i(n)_R$ to be the normalization of $\PP^1_R$ in the function field extension $K(Y_i(n)_K)/K(c)$. The special and generic fibers of $\mfX_i(n)_R$ are written $\mfX_i(n)_k$ and $\mfX_i(n)_K$, respectively.  
\end{defn}

\begin{rem}
Since $\mfX_i(n)$ is normal, it is smooth in codimension $1$.  In particular, the generic fiber $\mfX_i(n)_K$ is the smooth curve $X_i(n)_K$.
The special fiber $\mfX_0(n)_k$ is a quotient of $\mfX_1(n)_k$ by $\ZZ/n\ZZ$; see Remark~\ref{RY0nR}.
\end{rem}

\begin{rem}
Under the assumptions of the definition, the affine part of $\mfX_1(n)_R$ is none other than the scheme $Y_1(n)_R$, and thus its special fiber is just $Y_1(n)_k$; see Proposition \ref{Psamemodel}. 
\end{rem} 

\subsection{Parabolic parameters}\label{Sparabolic}

The branched cover of curves $\pi_1: \, Y_1(n)_K \to \AA_K^1$ has degree $\nu(n)= \sum_{d|n} m^d \mu(n/d)$, where $m$ is the degree in $x$ of $f_c(x)$. For generic $c$, the map $f_c$ has $\nu(n)$ points of primitive period $n$ that fall into $r(n)= \nu(n)/n$ orbits; these constitute the fiber of $\pi$ above $c$. For finitely many $c$,
the fiber of $\pi_1$ above $c$ consists of fewer than $\nu(n)$ points; such parameters are called \emph{parabolic parameters}. The points of multiplicity greater than $1$ in the fiber of $\pi_1$ above a parabolic parameter $c$ are called \emph{parabolic points} and the orbit of such a point under $f_c$ is called a \emph{parabolic orbit}. In characteristic zero, there is at most one parabolic orbit in any fiber of $\pi_1$; see Proposition~\ref{prop:branchedCover}(3). 

\begin{rem}\label{rem:multiplier1} For example, if $f_c$ has an orbit $\mcO$ whose primitive period $d$ is strictly less than its formal period $n$, then 
$\mcO$ is a parabolic orbit. In this case, $\lambda_d(\mcO)$ is a primitive $n/d$-th root of unity (assuming $\Char K \nmid n$) and $\lambda_n(\mcO) = \lambda_d(\mcO)^{n/d} = 1$. 
\end{rem}

\begin{defn}\label{Ddeltan} For the generic polynomial $f_c$, let $\alpha_1,\dots,\alpha_{r}$  be representatives for the $r= \nu(n)/n$ orbits of formal period $n$, i.e. $\alpha_i$ is a choice of point in the $i$-th orbit. Define
\[
\delta_n(x,c) = \prod_{i=1}^{r} (x-\lambda_n(\alpha_i)) \ {\rm and} \ D_n = \disc (\delta_n(1,c))
.\]
By \cite[Theorem A]{MortonVivaldi}, $\delta_n(x,c)\in \ZZ[x,c]$, and so $D_n \in {\mathbb Z}$.
\end{defn}
  
Recall the following definitions from \cite{MortonVivaldi}:

\begin{defn}\label{defn:factors} Let $C_m(x)$ denote the $m$-th cyclotomic polynomial and $\Res$ the resultant
with respect to $x$.
For $d\mid n$ and $d \neq n$, define
\[
\Delta_{n,d}(c) = \Res_x(C_{n/d}(x), \, \delta_d(x,c)).
\]
By Remark \ref{rem:multiplier1}, $\delta_n(1,c)$ is divisible by $\Delta_{n,d}(c)$. Define $\Delta_{n,n}(c)$ by 
\begin{equation}\label{factorizationOfDelta1}
\delta_n(1,c) = \Delta_{n,n}(c) \prod_{\substack{d\mid n \\ d \neq n}} \Delta_{n,d}(c).
\end{equation}
There is therefore a factorization
\begin{equation}
D_n= \prod_{\substack{d\mid n }} \disc(\Delta_{n,d})\cdot \prod_{\substack{d\mid n, e\mid n \\ d < e}} \Res(\Delta_{n,d},\Delta_{n,e})^2.
\end{equation}
Let $D_{n,d} = \disc(\Delta_{n,d})$ and $R_{n,d,e}=\Res(\Delta_{n,d},\Delta_{n,e})$.

\end{defn}

With this notation, we can state the following key result:

\begin{thm}[{\cite[Theorem A]{MortonVivaldi}}] \label{discriminantOfPhi} 
The $\Delta_{n,d}$ lie in $\ZZ[c]$ for all $d \mid n$, and satisfy the following identities:
\[
\Res_x ( \Phi_n(x,c), \Phi_d(x,c)) = \Delta_{n,d}(c)^{d}\quad\quad\quad \text{for $d\mid n$ and $d \neq n$} 
\]
and
\begin{equation}\label{eqn:discPhi}
\disc_x(\Phi_n(x,c)) = \pm  \Delta_{n,n}(c)^n \prod_{\substack{d\mid n \\ d \neq n}} \Delta_{n,d}(c)^{n-d}.
\end{equation}
Moreover, if $\alpha_1,\dots,\alpha_{r}$ are the representatives for the $r= \nu(n)/n$ orbits of formal period $n$ for $f_c$ and if 
\[
L_i(x) = \prod_{q=0}^{n-1} (x-f_c^q(\alpha_i)),
\]
then, for some unit $\eta$ in $\ZZ[c,\alpha_1,\dots,\alpha_{r}]$,
\begin{eqnarray}\label{eqn:meaningDelta_nn}
\Delta_{n,n}(c) = \eta \prod_{i\neq j} L_j(\alpha_i).
\end{eqnarray}

\end{thm}

Theorem~\ref{discriminantOfPhi} shows that the parabolic parameters are precisely the roots of $\delta_n(1,c)=0$. Moreover the factorization (\ref{eqn:discPhi})  can be interpreted as follows. For $d\mid n$ and $d \neq n$, the roots of $\Delta_{n,d}(c)$ are parabolic parameters for which an orbit of primitive period $n$ degenerates to an orbit of primitive period $d$; in this case, $\lambda_d(x)$ is a primitive $s$-th root of unity for any point $z$ in the parabolic orbit, where $s= n/(dp^e)$. For $d=n$, the identity (\ref{eqn:meaningDelta_nn}) shows that the roots of  $\Delta_{n,n}(c)$ are parameters for which two orbits of formal period $n$ collide. 

\begin{defn}\label{defn:satellite/primitive} Roots of $\Delta_{n,d}(c)$, for $d\neq n$, are called \emph{satellite parabolic parameters}.  Roots of $\Delta_{n,n}(c)$ are called \emph{primitive parabolic parameters}.
\end{defn}

\subsection{Smoothness and irreducibility}
The next hypothesis 
was introduced in \cite{Morton96}. 

\begin{defn} \label{defn:hypH}
Let $f(x,c)\in \ZZ[x,c]$ be a polynomial of degree $\alpha$ in $x$. 
The family $f_c(x)=f(x,c)$ satisfies \emph{hypothesis (H)} if:
\begin{enumerate}
\item $f(x,0)=x^{\alpha}$ and $f(x,u^m)$ is a homogeneous polynomial in $\ZZ[x,u]$ for some $m$;
\item $\gcd(\disc f(x,1), {\alpha}^n-1) =1$;
\item $\delta_n(1,c)$ has no multiple roots. 
\end{enumerate}
\end{defn}

Note that the family $f_c(x) = x^m + c$ satisfies parts (1) and (2) of hypothesis (H) for any $m \geq 2$, since $\disc(x^m+1)$ is a power of $m$ up to $\pm 1$.

\begin{thm}[{\cite[Theorem B]{Morton96}}] \label{thm:nonsing}
Let $f(x,c)$ be a family satisfying hypothesis (H). 
Suppose $D_n$ and $\disc f(x,1)$ are nonzero in $K$. Then $Y_1(n)_K$ is geometrically irreducible.  
\end{thm}

\begin{rem}
In \cite{MortonGalois} (Theorem B), Morton is able to prove Theorem \ref{thm:nonsing} above with part (ii) of hypothesis (H) replaced by ``$f(x,1)$ has distinct roots''.  
\end{rem}

\begin{thm}[{\cite[Chapter 3, Theorem 1, $m=2$ case]{bousch:thesis},\cite[Theorem 4.1, $m \geq 2$ case]{lau/schleicher:1994}}]\label{BBLG2} 
For the families of the form $f(x,c) = x^m+c$, the polynomial $\delta_n(1,c)$ has no multiple roots in characteristic 0. In particular, by Theorem \ref{thm:nonsing},  $Y_1(n)_\QQ$ is geometrically irreducible.
\end{thm}

\begin{thm}[{\cite[Expos\'e XIV, $m=2$ case]{DouadyHubbard}, \cite[Theorem 1.1, $m\geq 2$ case]{GaoOu}}]\label{Tmultibrotsmooth}
For the families of the form $f(x,c) = x^m+c$, the curve $Y_1(n)_\CC$ is smooth.
\end{thm}

\begin{rem}
The proofs of Theorems \ref{BBLG2} and 
\ref{Tmultibrotsmooth} rely on complex analysis and properties of the Mandelbrot set. 
Another proof is in \cite[Theorems 1.1--1.2]{BuffLei14}. 
\end{rem}

\subsection{Dynatomic curves as branched covers}

Consider the tower of covers 

\[\xymatrix{& X_1(n)_K\ar[dl]_{\ZZ/n\ZZ}^{\varphi}\ar[dd]_{\pi_1}^{\deg = \nu(n)} \\ X_0(n)_K\ar[dr]^{\pi_0}_{\deg = \nu(n)/n} \\ & \PP^1_K.}\]

\begin{prop}\ \label{prop:branchedCover}
Let $\mcF$ be a family of polynomials satisfying \hyperref[defn:hypH]{Hypothesis (H)}. Let $K$ be an algebraically closed field over which $Y_1(n)_K$ is irreducible and $D_n \neq 0$.
\begin{enumerate}

\item There are $\nu(n)/m$ distinct points of $X_1(n)_K$ in the fiber of $\pi_1 = \pi_0 \circ \varphi$ above $\infty$ in $\PP^1$, each with ramification index $m$ under $\pi_1$. Moreover, these points are all defined over the splitting field of $f(x,1)$. Note: this holds without part (3) of \hyperref[defn:hypH]{Hypothesis (H)}.

\item The cover $\pi_1: Y_1(n)_K\to \AA^1_K$ 
(resp.\ $\pi_0: Y_0(n)_K\to \AA^1_K$)
is branched over the subset of $\AA^1_K$ consisting of the roots of $\delta_n(1,c)$
(resp.\ $\Delta_{n,n}(c)$).

\item
If $c$ is a root of $\Delta_{n,d}(c)$ for $d \leq n$, then there are $d$ ramification points in
$\pi_1^{-1}(c)$,
which form a single orbit for the action of $\ZZ/n\ZZ$. 
At each such point $z$:
\begin{enumerate}
\item 
If $d=n$, $\varphi$ is unramified at $z$ 
and $\pi_1$ has ramification index  2 at $z$. 
\item If $d<n$, $\varphi$ has ramification index $n/d$ at $z$ and $\varphi(z)$ is unramified under $\pi_0$. 
\end{enumerate}
\end{enumerate}
\end{prop}
\begin{proof}
See \cite[Proposition 9 and 10]{Morton96}.
\end{proof}

\subsection{Mandelbrot set}
In this section, we consider the well-studied family $f_c(x)=x^2+c$;  we refer the interested reader to \cite{Branner89}, for example, for a survey. For $K=\CC$, the base curve $\AA^1_\CC$
in the cover $\pi_1$ 
is known as the ``parameter space.'' The Mandelbrot set $\mcM$ is the subset of parameters $c$ for which the orbit of the critical point $0$ under iteration of $f_c$ is bounded. The Mandelbrot set is connected (\cite[Corollary 8.3]{DouadyHubbard}). A parameter $c\in \mcM$ for which $0$ converges to an attracting cycle of $f_c$ is called \emph{hyperbolic}. Each hyperbolic map $f_c$ has exactly one attracting cycle. On a given connected component of the set of hyperbolic parameters, the period of the attracting cycle is a constant, say $n$; such a component is called a \emph{hyperbolic component of period $n$}.  The various hyperbolic components form the distinctive ``bulbs'' of the Mandelbrot set. The following is a well-known result in complex dynamics:

\begin{prop}\label{prop:uniqueRoot}
Let $H$ be a hyperbolic component of period $n$. Define a map 
\[
\lambda_n: H \to\mathbb{D}
\] which takes a parameter $c\in H$ to the $n$-th multiplier of the unique attracting cycle of $f_c$.  Then $\lambda_n$ is a conformal isomorphism and it extends to a homeomorphism $\bar{\lambda}_n:\overline{H} \to \overline{\mathbb{D}}$.
\end{prop}

\begin{defn}
For a hyperbolic component $H$ of period $n$, the unique point $c_0 \in H$ with $\lambda_n = 1$ is called the \emph{root of the hyperbolic component $H$}.
\end{defn}

\begin{rem}\label{hyperbolicRootsAndRootsOf1}
The roots of hyperbolic components are precisely the parabolic parameters, i.e. the roots of $\delta_n(1,c)$.  Satellite  parabolic parameters, i.e. the roots of $\Delta_{n,d}(c)$ with $d\neq n$, are precisely the parameters $c$ such that $c$ lies on the boundary of some hyperbolic component $H$ of period $d$ and also is a root of a hyperbolic component of period $n$. Such parameters on the boundary of a given component $H$ biject under $\lambda_d$ with the primitive roots of unity in $\overline{\mathbb{D}}$. 
Primitive parabolic parameters, i.e. roots of $\Delta_{n,n}(c)$, occur at the cusps of the cardioid shaped hyperbolic components. 
\end{rem}

We discuss more properties of the Mandelbrot set in \textsection \ref{Smonodromy}.

\section{Remarks on uniform boundedness and gonality}\label{section:UBC} 

We state a version of the uniform boundedness conjecture of Morton and Silverman for one dimensional families of polynomials (\cite{MortonSilverman94}; see \cite{PoonenUB} for a generalized form). 

\begin{conjecture}[Uniform boundedness for one dimensional families of polynomials]\label{conj:UniformBoundedness} 
Fix positive integers $N$ and $d\geq 2$.
Let $k$ be a field.

\begin{enumerate}
\item 
Let $\mcF/\QQ$ be a one dimensional family of polynomials of degree $d$. There exists $B = B_{N,\mcF} >0$ such that 
\[
|\Per(f)(L)|<B
\]
for all number fields $L$ such that $[L:\QQ]\leq N $ and for all rational maps $f\in \mcF(L)$.

\item 
Let $\mcF/k$ be a one dimensional family of polynomials of degree $d$. There exists $B = B_{N,k,\mcF} >0$ such that 
\[
|\Per(f)(L)|<B 
\]
for all function fields $L=k(C) $, where $C/k$ is a curve with $ \gamma_k(C)\leq N$, and for all non-isotrivial maps $f\in \mcF(L)$. Here $\gamma_k(C)$ denotes the $k$-gonality of the curve $C/k$. 
\end{enumerate}
\end{conjecture}

In this section, we consider strategies for proving the function field version (2). For the quadratic family $f_c(x)= x^2+c$ over $\QQ$, the rational points of $X_0(n)$ have been computed for $n\leq 6$ (\cite{MortonX4}, \cite{FPS}, and, conditional on BSD \cite{Stoll}); but even in this simplest case, little is known about the uniform statement. 
This difficulty is due in part to lack of understanding about primes of bad reduction (see \cite[\S 5]{Stoll}); 
for the function field case this is the main difficulty, as we now explain. 
\begin{rem}\label{rem:strategy}
Let $\mcF/k$ be a 1-dimensional family of polynomial maps of degree $d$. Denote by $M(L,\mcF)$ the supremum of periods of any $L$-rational periodic point for a non-isotrivial map $f\in \mcF(L)$. Then a bound on $M(L,\mcF)$ gives a bound on $\Per(f)(L)$, since for each period $n$ there are at most $d^n$ points of period $n$ for $f$. Denote by $N(L,\mcF)$ the maximal period of any $L$-rational periodic \emph{orbit} for a non-isotrivial map $f\in \mcF(L)$. Then $M(L,\mcF)\leq N(L,\mcF)$. An orbit of exact period $n$ for a non-isotrivial map $f\in \mcF(L)$ corresponds to a non-constant point of $Y_0(n)(L)$. Therefore, to bound $|\Per(f)(L)|$, it suffices to find an $n_0$ such that $Y_0(n)(L)$ consists only of constant points (i.e., points defined over $k$) for all $n>n_0$.   
\end{rem}

\begin{prop}\label{prop:UBCfromIrreducibility}
Let $k$ be a field of characteristic $p$ with $p\nmid m$. Let $\mcF/k$ be a family satisfying parts (1)--(2) of \hyperref[defn:hypH]{hypothesis (H)} (Definition~\ref{defn:hypH}). Then part (2) of Conjecture \ref{conj:UniformBoundedness} holds for $\mcF/k$ if there exists $n_1$ such that reduction mod $p$ of $X_0(n)$ is geometrically irreducible for all $n> n_1$. 
\end{prop}

First recall the following two lemmas on gonality. Let $X^{ns}$ denote the non-singular points of a curve $X$.

\begin{lemma}\label{lem:pointsForGonality}
If $X/\FF_q$ is a geometrically irreducible curve, then $\gamma_{\FF_q}(X) \geq |X^{ns}(\FF_q)|/(q+1)$.
\end{lemma}
\begin{proof}
If $\varphi: X \dashrightarrow \PP^1$ is a nonconstant rational map of degree $d$, then we have a morphism $\varphi: X^{ns} \to \PP^1$ and each of the $q+1$ points of $\PP^1(\FF_q)$ can have at most $d$ preimages under $\varphi$, so $d(q+1) \geq |X^{ns}(\FF_q)|$.
\end{proof}

\begin{lemma}\label{lem:gonalityUnderFiniteGenExt}
Consider an extension $k/\FF$ of a perfect field $\FF$. Suppose $X_n$ is a sequence of curves defined over $\FF$ with $\gamma_\FF(X_n)\to \infty$ as $n\to \infty$ and $X_n(\FF)\neq \varnothing$ for all $n$. Then $\gamma_k(X_n)\to \infty$ as $n\to \infty$. 
\end{lemma}
\begin{proof}
Under an algebraic extension of fields $L/\FF$ we have $ \gamma_L(X_n) \geq \sqrt{\gamma_\FF(X_n)}$ (\cite[Theorem 2.5]{PoonenGonality}). In particular, $\gamma_{\overline{\FF}}(X_n) \geq \sqrt{\gamma_\FF(X_n)}$.  On the other hand, for any extension $M/\overline{\FF}$, we have $\gamma_{\overline{\FF}}(X_n) = {\gamma_M(X_n)}$ (by specialization; see  \cite[Proposition 1(ii)]{PoonenGonality}).  Since $k$ is contained in such an extension $M/\overline{\FF}$, we have $\gamma_k(X_n)\geq \gamma_M(X_n) = \gamma_{\overline{\FF}}(X_n) \geq \sqrt{\gamma_\FF(X_n)}$, whence the claim follows. 
\end{proof}

\begin{proof}[Proof of Proposition \ref{prop:UBCfromIrreducibility}]
Let $L$ be a function field, $L=k(C) $, where $C/k$ is a curve with $ \gamma_k(C)\leq N$. A nonconstant point of $X_0(n)(L)$ corresponds to a dominant rational map $C\dashrightarrow X_0(n)_k$. If $X_0(n)_k$ is geometrically irreducible, the existence of such a map implies that $\gamma_k(C)\geq \gamma_k(X_0(n))$ (see \cite[Proposition 1.1]{PoonenGonality}).  Therefore, by Remark \ref{rem:strategy}, it is enough to find $n_0$ such that $\gamma_k(X_0(n))> N$ for all  $n>n_0$.  That is, we need to show $ \gamma_k(X_0(n))\to \infty$ with $n$. 

Suppose $X_0(n)$ has irreducible reduction mod $p$ for all $n>n_1$. Let $\FF$ be the splitting field of $f(x,1)$ for $x^m + c$ over $\FF_p$. By Proposition \ref{prop:branchedCover}(1) and and Proposition \ref{Psamemodel} (note that we do not use the results of this section below, so there is no circularity in this forward reference), there are at least $\nu(n)/mn$ smooth points at infinity in $X_0(n)_{\FF_{p}}(\FF)$. So, by Lemma \ref{lem:pointsForGonality}, $\gamma_{\FF_p}(X_0(n)_{\FF_p}) \geq \gamma_{\FF}(X_0(n)_{\FF}) \geq \nu(n)/(mn(p^m+1))$ for all $n>n_1$. Since $\nu(n)/n\to \infty$ as $n\to \infty$, we have $\gamma_{\FF_p}(X_1(n)_{k})\to \infty$. Now by Lemma~\ref{lem:gonalityUnderFiniteGenExt}, $ \gamma_k(X_0(n))\to \infty$.
\end{proof}

For the function field part of Conjecture \ref{conj:UniformBoundedness} in characteristic 0, it suffices to find, for each $n$, one small prime of geometrically irreducible reduction, as in the next result.

\begin{prop}\label{Puniformboundedness}
Let $k$ be a field of characteristic zero.  Let $\mcF/k$ be a family satisfying parts (1) and (2) of \hyperref[defn:hypH]{hypothesis (H)}.
Then
Part (2) of Conjecture \ref{conj:UniformBoundedness} holds for $\mcF/k$ if there exists $n_0$ and a sequence $\{p_n\}_{n>n_0}$ of primes such that:
(1) $\nu(n)/np_n\to \infty$ as $n\to \infty$;  (2) $p_n\nmid m$ and 
(3) for $n>n_0$, the curve $X_0(n)$ has geometrically irreducible reduction mod $p_n$.
\end{prop}
\begin{lemma}\label{lem:gonalityOfReduction}
Let $X$ be a smooth curve over a number field $K$, and let $\mathfrak{p}$ be a prime of $K$ with residue field $\FF_q$, such that the reduction $X^\prime$ of $X$ modulo $\mathfrak{p}$ is geometrically irreducible. Then $\gamma_K(X)\geq \gamma_{\FF_q}(X^\prime)$. 
\end{lemma}
\begin{proof}
This is a slight generalization of \cite[Proposition 5]{Xarles}. Suppose $f: X\to Y$ is a morphism of degree $d$ where $Y=\PP^1_K$. To prove the lemma, it is enough to show that this implies the existence of a morphism $f^\prime: X^\prime \to \PP^1_{\FF_q}$ of degree $d^\prime \leq d$. Applying \cite[Proposition 4.14]{LiuLorenzini}), there are models $\mathcal{X}$ and $\mathcal{Y}$ of $X$ and $Y$ over $\mathcal{O}_\mathfrak{p}$ together with a rational map $\phi:\mathcal{X}\to \mathcal{Y}$ that restricts to a quasifinite map on $X^\prime$. Let $Z$ be the image of $X^\prime$ under $\phi$.
Consider $f':\phi|_{X'}: X' \to Z$. Then $f'$ is finite, since $X^\prime$ is proper. The curve $Z$ smooth, since it is both geometrically irreducible, because $X^\prime$ is, and of arithmetic genus zero, because it is in the special fiber of $\mathcal{Y}$. 
Moreover, since $Z$ is defined over a finite field, it has a rational point and is therefore isomorphic to $\PP^1$. 
\end{proof}
\begin{proof}[Proof of Proposition \ref{Puniformboundedness}]
 By the same reasoning as in the second paragraph of the proof of Proposition \ref{prop:UBCfromIrreducibility}, the existence of a sequence of primes $p_n$ satisfying (1) and (2) implies that $\gamma_{\FF_{p_n}}(X_0(n)_{\FF_{p_n}})\to \infty$ as $n \to\infty$. By Lemma \ref{lem:gonalityOfReduction}, this implies $\gamma_{\QQ}(X_0(n)_{\QQ})\to \infty$. By Lemma \ref{lem:gonalityUnderFiniteGenExt}, $\gamma_{K}(X_0(n)_{K})\to \infty$ for any field $K$ of characteristic 0. By the argument of the first paragraph of the proof of Proposition \ref{prop:UBCfromIrreducibility}, this is sufficient.   
\end{proof}

\section{Preliminaries for good/bad reduction}\label{Sprelims}

We use a test of Fulton that detects singularities by comparing the cardinality of fibers of branched covers of curves with the degree of the local ramification divisor. If $A$ is a Dedekind domain, whose maximal ideals all have the same residue field, and $I \subseteq A$ is a nonzero ideal, then 
the \emph{degree} of $I$ is $\sum_{x \in \MaxSpec A} \ord_x I$; this is generally applied to the \emph{discriminant} ideal of an extension of $A$.  

\begin{thm}[{\cite[Theorem 2.3]{Fulton69}}]
\label{Lfultontest}
Let $X = \Spec R$, where $R$ is a discrete valuation ring with algebraically closed residue field $k$. 
Let $x \in X$ be the closed point. 
Let $f: Y = \Spec S \to \Spec R$
be a finite, flat, generically separable morphism of degree $\alpha$.   Let $\delta$ be the degree of the discriminant of $S/R$. Then $\delta \geq \alpha - |f^{-1}(x)|$, with equality holding if and only if $Y$ is smooth and $f$ is tamely ramified.
\end{thm}

\begin{proof}
That $f$ is generically separable ensures that it is a \emph{covering} in the sense of Fulton.  The proposition then follows from \cite[Theorem 2.3]{Fulton69}.
\end{proof}

\begin{cor}\label{Cnoncollisioncount}
In the situation of Theorem \ref{Lfultontest}, there are at least $\alpha - 2\delta$ points in $f^{-1}(x)$ of multiplicity one.
\end{cor}

\begin{proof}
By Theorem \ref{Lfultontest}, $|f^{-1}(x)| \geq \alpha - \delta$ and $\sum_{y \in f^{-1}(x)} m_y = \alpha$, where $m_y$ is the multiplicity of $y$.  If $\beta$ is the number of $m_y$ equal to $1$, then 
$$\sum_{y \in f^{-1}(x)} m_y \geq \beta + 2(\alpha - \delta - \beta) = 
2\alpha - 2\delta - \beta,$$ which is greater than $\alpha$ if $\beta < \alpha - 2 \delta$.
\end{proof}

For the rest of \S\ref{Sprelims}, let $R$ be a complete discrete valuation ring with fraction field $K$
and algebraically closed residue field $k$.  Let $\ol{K}$ be the algebraic closure of $K$.  Let $A = R[\![T]\!]$, and let $B$ be a normal, finite, flat $A$-algebra of degree $\alpha$. Write $A_{\ol{K}}$, $B_{\ol{K}}$, $A_k$, $B_k$ respectively for $A \otimes_R {\ol{K}}$, $B \otimes_R {\ol{K}}$, $A \otimes_R k$, and $B \otimes_R k$.  Assume $B_k$ is reduced and $B_k/A_k$ is generically \'etale.  Write $d_s$ (resp.\ $d_{\eta}$) for the degree of the discriminant of $B_k/A_k$ (resp.\ $B_{\ol{K}}/A_{\ol{K}}$).

\begin{prop}[{\cite[I, 3.4]{GreenMatignon98}}]\label{Pfultontest}
With notation as above, $d_s = d_{\eta}$.
\end{prop}

\begin{proof}
It suffices to consider the case that $B$ is local, by localizing $B$ at its (finitely many) maximal ideals and summing the equalities. Then the result follows from
\cite[I, 3.4]{GreenMatignon98}.  
\end{proof}

The following corollaries are important criteria for good and bad reduction, respectively.
See \cite[\S 6]{Raynaud} for similar material in the Galois case.

\begin{cor}\label{Cgoodredtest}
Suppose $B$ is local and $\Spec(B_{\ol{K}}) \to \Spec(A_{\ol{K}})$ is ramified above at most one closed point $z$, for which the ramification is tame.  Then $\Spec(B_{\ol{K}}) \to \Spec(A_{\ol{K}})$ is totally ramified (of index $\alpha$) above $z$, char$(k)$ does not divide $\alpha$, and $B_k$ is normal; in particular, $\Spec B_k$ is smooth. \end{cor}

\begin{proof} Let $e_1, \ldots, e_r$ be the ramification indices of the $r$ points above $z$ (or above some arbitrary $\ol{K}$-point, if there is no ramified point). Then $\alpha = \sum_{i=1}^r e_i$, and $d_{\eta} = \sum_{i=1}^r (e_i - 1)$.  By Theorem \ref{Lfultontest} and Proposition \ref{Pfultontest}, $d_{\eta} = d_s \geq \alpha - 1$.  This is only possible if $r = 1$ and $e_1 = \alpha$, so that $d_{\eta} = d_s = \alpha-1$. By Theorem \ref{Lfultontest}, $B_k$ is normal and ${\rm char}(k) \nmid \alpha$.
\end{proof}

\begin{cor}\label{Cbadredtest}
For any element $a$ in the maximal ideal of $R$, let $d_a$ be the degree of the discriminant of the extension $S_a/R$ resulting from setting $T = a$ in $B/A$.  
\begin{enumerate}
\item If $d_a < d_{\eta}$, then either $B_k$ is not normal or $B_k/A_k$ is wildly ramified.
\item In addition, if there exists a closed point of $\Spec A_{\ol{K}}$ whose preimage under the map $\pi: \Spec B_{\ol{K}} \to \Spec A_{\ol{K}}$ consists of $\alpha - d_a$ points, then $B_k$ is not normal. 
\end{enumerate}
\end{cor}

\begin{proof}
Let $\nu$ be the number of maximal ideals in $B$ (which is also the number of maximal ideals in $B_k$ and in $S_a$, since $A$ is local).  Suppose $B_k$ is normal and $S_a/R$ is tamely ramified. Theorem \ref{Lfultontest} applied to $B_k/A_k$ shows that $d_s = \alpha - \nu$.  On the other hand, Theorem \ref{Lfultontest} applied to $S_a/R$ shows that $d_a \geq \alpha - \nu = d_s$. Since $d_{\eta} = d_s$ by Proposition \ref{Pfultontest}, this contradicts $d_a < d_{\eta}$. This proves (1).

Under the additional assumptions of (2), $\nu \leq \alpha - d_a$, and thus $\nu = \alpha - d_a$. By Theorem \ref{Lfultontest}, the extension $S_a/R$ is tamely ramified and $S_a$ is normal.  The ramification indices of $S_a/R$ are the multiplicities of the points on the fiber of $\Spec B \to \Spec A$ above the maximal ideal of $A$, which are also equal to the ramification indices of $B_k/A_k$. By Corollary \ref{Cgoodredtest}, these are not divisible by char$(k)$. Thus is it not possible for $B_k/A_k$ to be wildly ramified with $B_k$ normal.  We conclude that $B_k$ is not normal, proving (2). 
\end{proof}

\section{Good reduction}\label{Sgood}
We assume throughout this section that the family 
$\mc{F} = \{f_c\} = \{f(x,c)\}$ satisfies Morton's hypothesis (H) from Definition~\ref{defn:hypH} for $c = u^m$. If $f_c(x) = x^m + c$, recall that this property is satisfied. In this section, let $Y_1(n)$ denote $Y_{1, \mcF}(n)$ (and similarly for $Y_0(n)$).

Under hypothesis (H), we prove
that the dynatomic curves which are smooth in characteristic zero have good reduction modulo all primes not dividing $D_n = \disc \delta_n(1, c)$.  We also show a similar statement about irreducibility that was already proven in \cite[Theorem~15]{Morton96}, but the proof given here avoids working explicitly with factorizations of polynomials. 

If $R$ is a ring, recall that $Y_1(n)_R$ is the closed subscheme of $\AA^2_R$ cut out by the dynatomic polynomial $\Phi_n \in \ZZ[x,c]$ and that $Y_0(n)_R:=Y_1(n)_R/\langle \sigma \rangle$, where $\sigma(x, c) := (f_c(x), c)$.  

Throughout this section, let $k$ be a field of characteristic $p$. We assume that $R$ is a characteristic zero complete discrete valuation ring with field of fractions $K$ and residue field $k$.  The generic and special fibers of $Y_0(n)_R$ are $Y_0(n)_K$ and $Y_0(n)_k$, respectively. Let $\ell$ be the splitting field of $f(x,1)$ over $k$.  Note that $\ell/k$ is separable if $p \nmid \disc f(x,1)$.

\begin{rem}\label{Rxmplusc}
If $f(x,c) = x^m + c$, then the following Propositions \ref{Pnormal}, \ref{Pgensep}, and \ref{prop:Pbadimpliesdisc} apply to all primes not dividing $m$, since $\disc f(x,1)$ is a power of $m$. 
\end{rem}

\begin{prop}\label{Pnormal}
Let $p$ be a prime not dividing $\disc f(x, 1)$.   
Then $Y_i(n)_R$ has reduced special fiber.  
If $Y_i(n)_K$ is smooth, then $Y_i(n)_R$ is isomorphic to the normalization of $\AA^1_R$ in the function field extension $K(Y_i(n)_K)/K(c)$ coming from the generic fiber of $\pi_i: Y_i(n)_R \to \AA^1_R$.
\end{prop}

\begin{proof}
Write $c_k$ for the reduction of $c$ to $\AA^1_{k}$.
The fiber product $Y_1(n)_k \times_{\AA^1_{k}} \Spec \ell ((1/c_k))$ is reduced by \cite[Proof of Proposition 10]{Morton96}, which means that $Y_1(n)_k$ is reduced; (in \cite{Morton96}, it is assumed that $Y_1(n)_k$ is irreducible, but this is not used).   Since $Y_0(n)_k$ is a quotient of $Y_1(n)_k$, it is reduced as well.

To show the second statement about $Y_i(n)_R$, it suffices to show that $Y_i(n)_K$ smooth implies $Y_i(n)_R$ is normal. Since $Y_i(n)_k$ is reduced, this follows from \cite[Lemma 4.1.18]{LiuAG}.
\end{proof}

\begin{prop}\label{Pgensep}
Let $p$ be a prime not dividing $m \disc f(x,1)$, where $m$ is the parameter in hypothesis (H) for $f$. The natural projection
$\pi_{1, k}: Y_1(n)_k \to \AA^1_k$ given by $(x, c) \mapsto c$ is generically separable of degree $\nu(n)$ (the degree of $Y_1(n)_k \to \AA^1_k$).
\end{prop}

\begin{proof}
By \cite[Proposition 10]{Morton96}, taking the base change of $\pi_{1, k}$ by the map $\Spec \ell((1/c)) \to \AA^1_k$ yields a separable map $\pi$ of degree $\nu(n)$; indeed, each point has degree $m$ over $\Spec \ell((1/c))$.  Then $\pi_{1, k}$ is generically separable since $\ell((1/c))$ is separable over $k(c)$.
\end{proof}

Recall that $\mathfrak{X}_i(n)_R$ is the normalization of $\PP^1_R$ in $K(Y_i(n)_K)$ for $i \in \{0,1\}$, assuming $Y_i(n)_K$ is smooth.

\begin{prop}\label{Psamemodel}
Let $p$ be a prime not dividing $m \disc f(x,1)$, let $i \in \{0,1\}$, and assume that $Y_i(n)_K$ is smooth.
\begin{enumerate}
\item The curve $Y_i(n)_k$ is isomorphic to the preimage
of $\AA^1_k$ in the special fiber $\mathfrak{X}_i(n)_k \to \PP^1_k$ of $\mathfrak{X}_i(n)_R \to \PP^1_R$. 
\item If $k$ is perfect, then $\mathfrak{X}_i(n)_k \cong X_i(n)_k$.  In particular, $\mathfrak{X}_i(n)_k$ is smooth above $\infty$.
\end{enumerate}
\end{prop}

\begin{proof}
By Proposition \ref{Pnormal}, whose hypotheses are satisfied by Remark \ref{Rxmplusc}, $Y_i(n)_k$ is the special fiber of the normalization of $\AA^1_R$ in $K(Y_K)$.  The definition of $\mathfrak{X}_i(n)_R$ implies part (1). 

To prove part (2), it suffices by part (1) and the definition of $X_i(n)_k$ (Definition \ref{defn:standardModelX1}) to show that $\mf{X}_i(n)_k$ is smooth above $\infty$. 
By Corollary \ref{Cgoodredtest}, it suffices to show that no branch point of $\mf{X}_i(n)_K = X_i(n)_K \to \PP^1_K$ specializes to $\infty$ other than $\infty$.  But any such finite branch point is a root of $\delta_n(1, c)$ (Proposition \ref{prop:branchedCover}(2)).  These roots are integral over $R$ by \cite[top of p.\ 331]{Morton96}, so they do not specialize to $\infty$. 
\end{proof}

\begin{prop} \label{prop:Pbadimpliesdisc} 
Let $p$ be a prime not dividing $m\disc f(x,1)$.
\begin{enumerate}
\item If $p \nmid D_n$,
then $Y_1(n)_k$ is irreducible.  If, in addition, $Y_1(n)_K$ is smooth, then so is $Y_1(n)_k$ (cf.\ {\cite[Theorem 15]{Morton96}}).
\item The same holds when $Y_1(n)_k$ is replaced by $Y_0(n)_k$ and $D_n$ is replaced by $D_{n,n}$.
\end{enumerate}
\end{prop}

\begin{proof}
We may assume $k$ algebraically closed.  
For (1), first assume that $Y_1(n)_K$ is smooth.  By Proposition \ref{Pnormal}, $Y_1(n)_R$ is a normal scheme, and
the special fiber $Y_1(n)_k$ of $Y_1(n)_R$ is reduced.  By Proposition \ref{Pgensep}, the special fiber of the map $Y_1(n)_R \to \PP^1_R$ is generically separable.  

Now, let $x$ be a closed point of $\AA^1_R$.  Then $\widehat{\mc{O}}_{\AA^1_R, x} \cong R[\![T]\!]$.  Since $p \nmid D_n$, Proposition \ref{prop:branchedCover}(2) shows that \emph{at most one} branch point of $Y_1(n)_K \to \AA^1_K$ specializes to $x$. Let $y \in Y_1(n)_k$ be a closed point above $x$. The extension $\widehat{\mc{O}}_{Y_1(n)_R, y} / \widehat{\mc{O}}_{\PP^1_R, x}$ is of the form considered in \S\ref{Sprelims}. By Corollary \ref{Cgoodredtest}, $y$ is a smooth point of $Y_1(n)_R$.  Since $y$ is an arbitrary closed point and $k$ is algebraically closed and thus perfect, $Y_1(n)_k$ is smooth. Since $Y_1(n)_K$ is connected by Theorem \ref{thm:nonsing}, so is the generic fiber $\mf{X}_1(n)_K$ of $\mf{X}_1(n)_R$.  By Zariski's Connectedness Theorem, $\mf{X}_1(n)_k$ is also connected.  Identify $Y_1(n)_k$ with an open subscheme of $\mf{X}_1(n)_k$ using Proposition \ref{Psamemodel}(1).  Since all points of $\mf{X}_1(n)_k \setminus Y_1(n)_k$ are smooth by Proposition \ref{Psamemodel}(2), $Y_1(n)_k$ is connected.  Since $Y_1(n)_k$ is smooth and connected, it is irreducible.

If $Y_1(n)_K$ is not smooth; it is still irreducible by Theorem \ref{thm:nonsing}.  If $\widetilde{Y}_K$ is the normalization of $Y_1(n)_K$, then the induced map $\widetilde{Y}_K \to \AA^1_K$ is a branched cover with branch locus contained inside that of $Y_1(n)_K \to \AA^1_K$.  If $\widetilde{Y}_R$ is the normalization of $\AA^1_R$ in $K(\widetilde{Y}_K)$, and $\widetilde{Y}_k$ is its special fiber, then the same proof as in the smooth case shows that $\widetilde{Y}_k$ is smooth and irreducible.  Since $Y_1(n)_k$ is birational to $\widetilde{Y}_k$, this finishes the proof of (1). 

The same proof works for (2), using Proposition \ref{prop:branchedCover}(3) in place of Proposition \ref{prop:branchedCover}(2).
\end{proof}

\begin{rem}\label{RY0nR}
If $Y_1(n)_K$ is smooth, then $Y_1(n)_R$ and $Y_0(n)_R$ are  normal.
Since the $\sigma$-action is ``vertical" and $Y_1(n)_k \to \AA^1_k$ is separable, the induced $\sigma$-action on $Y_1(n)_k$ is \emph{faithful}.  
\end{rem}

Good reduction of $Y_1(n)$ can sometimes be characterized in terms of the non-collision of parabolic orbits, as in the following lemma.

\begin{lemma}\label{Lramcollision}
Suppose $p \nmid m \disc f(x,1)$
and $Y_1(n)$ is smooth in characteristic zero. 
If $Y_0(n)$ has good reduction modulo $p$, and no branch points of $Y_1(n) \to Y_0(n)$ collide modulo $p$, then $Y_1(n)$ has good reduction modulo $p$.
\end{lemma}

\begin{proof}
Let $y$ be a closed point on the special fiber of $Y_1(n)_R$, and $x$ be its image in $Y_0(n)_R$.  By Propositions \ref{Pnormal} and \ref{Pgensep}, the extension $\widehat{\mc{O}}_{Y_1(n)_R, y}/\widehat{\mc{O}}_{Y_0(n), x}$ is of the form considered in \S\ref{Sprelims}. By Corollary \ref{Cgoodredtest}, $y$ is smooth.  The lemma follows.
\end{proof}

Recall the factorization of the discriminant $D_n = \mathrm{disc}(\delta_{n}(1,c))$ given by equation \ref{factorizationOfDelta1}:

\begin{equation}\label{Efactorization}
D_n= \prod_{d\mid n} D_{n,d}\cdot \prod_{\substack{d < e \\ d\mid n, e \mid n}} R_{n,d,e}^2,
\end{equation}
where $D_{n,d} = \mathrm{disc}(\Delta_{n,d})$ and $R_{n,d,e}=\mathrm{res}(\Delta_{n,d},\Delta_{n,e})$. We now give refined criteria for good reduction in terms of these factors of $D_n$. 

\begin{prop}\label{PRnne}
Let $p \nmid m \disc f(x,1)$
and suppose $Y_1(n)$ is smooth in characteristic zero.  Suppose that $p$ divides $\prod_{e|n, e < n} R_{n, e, n}$ but does not divide $D_n/\prod_{e|n, e < n} R_{n, e, n}^2$.  Then $Y_1(n)$ has good reduction modulo $p$.  
\end{prop}

\begin{proof}
Since $p \nmid D_{n,n}$, Proposition \ref{prop:Pbadimpliesdisc}(2) shows that $Y_0(n)$ has good reduction modulo $p$.  The finite branch points of $Y_1(n) \to Y_0(n)$ lie above the roots of $F(c) := \prod_{\substack{d | n \\ d < n}} \Delta_{n,d}(c)$ in $\AA^1$.  By assumption, 
$$p \nmid \disc(F(c)) = \prod_{d < n} D_{n,d} \prod_{\substack{d|n, e|n \\ d < e < n }} R_{n, d, e}^2.$$  So the images of these branch points in $\AA^1$ do not collide modulo $p$. Thus the branch points themselves do not collide on $Y_0(n)_k$.  The result follows from Lemma \ref{Lramcollision}. 
\end{proof}

\begin{rem} \label{Rsimplify}
Proposition \ref{PRnne} states that the set of primes where $Y_1(n)$ has bad reduction
does not contain any prime $p$ which divides $D_n$ only because it divides $\prod_{e < n} R_{n, e, n}$.
In Tables~\ref{table:data1-7}--\ref{table:data8} for $f(x, c) = x^2 + c$, this means bad reduction cannot occur if $p$ appears in the list of $(e,n)$ rows for $e < n$ but not outside this list. 
For example, $Y_1(6)$ has good reduction modulo $p=79, 211, 68700493$ and 
$Y_1(8)$ has good reduction modulo $p=53, 593, 12073, 248117$.
\end{rem}

\subsection{The case of $f(x, c) = x^m + c$}\label{Sxmplusc}
In \S\ref{Sxmplusc}, we obtain additional results on good reduction when $f(x, c) = x^m + c$.  By Theorem \ref{Tmultibrotsmooth}, $X_1(n)$ is smooth for such $f$ in characteristic zero.  

\begin{lemma}\label{Ltotram} 
In characteristic zero, for $n > 2$ and $f(x) = x^m + c$, the totally ramified points of $Y_1(n) \to Y_0(n)$ are exactly the points
\begin{equation}\label{Elabelram10}
z = (x_0, c_0) := \left(\frac{\zeta_{(m-1)n}}{m^{1/(m-1)}}, x_0 - x_0^m\right),
\end{equation}
as $\zeta_{(m-1)n}$ ranges through the ($m-1$)-st roots of the primitive $n$-th roots of unity. Furthermore, over any $\ZZ[\zeta_{(m-1)n}, m^{1/(m-1)}]$-algebra, we have  
\begin{equation}\label{Etotram}
\frac{\partial \Phi_n}{\partial c}(z) \prod_{\substack{d | n \\ d \neq 1, n}} \Phi_d(z) = -\frac{(m-1)m^{1/(m-1)}n}{\zeta_{n}(\zeta_n - 1)^2}
\end{equation}
for $z$ as in \eqref{Elabelram10}.  Here we take $\zeta_n = \zeta_{(m-1)n}^{m-1}$. 
\end{lemma}

\begin{proof}
The points $z$ in \eqref{Elabelram10} are exactly those points in $\AA^2$ such that $x_0$ is a fixed point of $x \mapsto x^m + c_0$ with multiplier $\lambda_1(x_0)$ equal to a primitive $n$th root of unity. This proves the first statement. 
Next, fix $\zeta_{(m-1)n}$.  
By \cite[Lemma 3.7]{BuffLei14} (using $m = 1$ and $s = n$), $$\frac{\partial \Phi_n}{\partial c}(z) \prod_{\substack{d | n \\ d \neq 1, n}} \Phi_d(z) = \frac{n \, \partial \rho /\partial c}{\rho(\rho - 1)}(c_0)$$ over $\CC$, where $\rho$ is the multiplier function on $Y_1(1)$, thought of as a function of $c$ (note that \cite{BuffLei14} assumes $f(x,c) = x^2 + c$, but the proof of this statement is the same for $x^m + c$ --- see also \cite[Lemma 2.6]{GaoOu}).  Now, $\rho(c_0) = \zeta_n$.  Since $\rho(x,c) = mx^{m-1}$,
\begin{equation}\label{Etotram2}
\frac{\partial \rho}{\partial c} = m(m-1)x^{m-2}\frac{\partial x}{\partial c}.
\end{equation}
The variables $x$ and $c$ are related by $x^m + c = x$, so $\partial x / \partial c$ at $c = c_0$ equals $1/(1 - mx_0^{m-1})$, or $1/(1 - \zeta_n)$. Plugging in $x_0$ to (\ref{Etotram2}) and simplifying proves (\ref{Etotram}) over $\CC$. Since $\Phi_n$ is defined over $\ZZ$, the formula holds as well for all $\ZZ[\zeta_{(m-1)n}]$-algebras.  
\end{proof}

\begin{prop}\label{PX1goodX0good}
Assume $f(x, c) = x^m + c$.  If $p$ and $n$ are distinct primes with $p \nmid m(m-1)$, then $Y_1(n)$ has good reduction modulo $p$ if and only if $Y_0(n)$ does. 
\end{prop}

\begin{proof}
Since $Y_1(n)$ is a finite cover of $Y_0(n)$, it has bad reduction if $Y_0(n)$ does --- indeed, its minimal semistable model has nodes at the preimages of the nodes of the minimal semistable model of $Y_0(n)$.

Suppose $Y_0(n)$ has good reduction.  By Lemma \ref{Lramcollision}, it suffices to show that the branch points of $\phi: Y_1(n) \to Y_0(n)$ do not collide on the special fiber.  Since $\phi$ is cyclic of prime order, this is the same as checking that the ramification points do not collide on the special fiber.  These points are given in Lemma \ref{Ltotram}. By the assumptions on $p$, none of them collide on the special fiber. 
\end{proof}

\begin{ex}
Assume $f(x,c) = x^2 + c$.  Let $n=5$ and $p=11$.
Then $Y_0(5)$ has good reduction modulo $11$ since $11 \nmid D_{5,5}$.  So $Y_1(5)$ has good reduction modulo $11$ by Proposition
\ref{PX1goodX0good}
(even though $11 \mid D_{1,5}$ and $11 \mid D_{1,1}$.) 
\end{ex}

\begin{rem}\label{rem:false_when_composite} Proposition \ref{PX1goodX0good} is false when $n$ is composite. For example, take the family $f_c(z) = z^2 +c$ with $n=6$ and $p=67$. The prime $67$ does not divide $D_{6,6}$, so $Y_0(6)$ has good reduction at 67 by Proposition~\ref{prop:Pbadimpliesdisc}(2). On the other hand, $67$ divides $D_{n}$ exactly once (it divides $D_{6,3}$); so Theorem~\ref{Tbad} below shows that $Y_1(6)$ has bad reduction at $67$. Nonetheless, $Y_1(6)$ is still \emph{irreducible} modulo 67, as Proposition \ref{PX1irredX0irred} below shows.
\end{rem}

\begin{prop}\label{PX1irredX0irred}
Assume $f(x,c) = x^m + c$.  
If $p \nmid (m-1)mn$, then $Y_1(n)$ is geometrically irreducible modulo $p$ if and only if $Y_0(n)$ is. 
\end{prop}

\begin{proof}
Since $Y_1(n)$ is a finite cover of $Y_0(n)$, it is geometrically reducible modulo $p$ if $Y_0(n)$ is.  Now, suppose $Y_0(n)$ is geometrically irreducible modulo $p$. Then the  $\ZZ/n\ZZ$-action on $Y_1(n)$ in characteristic $p$ permutes the irreducible components transitively.  To show that $Y_1(n)$ is geometrically irreducible modulo $p$, it suffices to show that $Y_1(n) \to Y_0(n)$ has a smooth, totally ramified point in characteristic $p$, since such a point can lie on only one irreducible component but must simultaneously lie on all irreducible components.  Since $p \nmid (m-1)mn$, the right hand side of (\ref{Etotram}) is a $p$-adic unit.  Lemma \ref{Ltotram} then implies that $\partial \Phi_n/\partial c$ (at the point $z$) does not vanish modulo $p$, which shows that $z$ is the desired point.
\end{proof}

\subsection{Quadratic maps} In this section, we restrict to the quadratic maps $f_c(x)=x^2+c$. In this case, we can rule out prime factors of $D_n$ that only divide the factors $D_{n,1}$ and $D_{n,2}$.

\begin{lemma}\label{lemma:distinctHyperbolicComponents} Let $d < n$ be positive integers with $d \mid n$. Let $p$ be a prime not dividing $2n/d$. Suppose that there are distinct roots $c_1$ and $c_2$ of $\Delta_{n,d}$ and ramification points $(x_1, c_1)$ and $(x_2, c_2)$  of $Y_1(n)\to \AA^1$ that collide on the special fiber $Y_1(n)_{\FF_p}$. Then $c_1$ and $c_2$ lie on the boundaries of distinct period $d$ hyperbolic components of the Mandelbrot set.
\end{lemma}

\begin{proof}
By Proposition \ref{prop:branchedCover}(4), for $i = 1,2$, the orbit of $x_i$ is a parabolic orbit for $f_{c_i}$; hence each $c_i$ lies on the boundary of a unique period $d$ hyperbolic component $H_i$. The $d$-th multipliers $\lambda_i = \lambda_{d}(x_i)$ are primitive $(n/d)$-th roots of unity (by the comment after Proposition \ref{discriminantOfPhi}).  Both $x_i$ and $c_i$ are $p$-adic integers at all odd primes $p$ by \cite[Corollary 3.5]{MortonVivaldi}. Since $x_1 \equiv x_2$ and $c_1 \equiv c_2$ (mod $p$), it follows that $\lambda_1 \equiv \lambda_2$ (mod $p$). If $c_1\neq c_2$ and $H_1=H_2$, then the multipliers $\lambda_i$ would be \emph{distinct} primitive $(n/d)$-th roots of unity (by Remark \ref{hyperbolicRootsAndRootsOf1}), hence distinct mod $p$ for $p\nmid n/d$.  Thus $H_1 \not = H_2$.
\end{proof}

\begin{prop}\label{P12}
Suppose $p \nmid 2n$ and $n > 2$.  Let $S = \{1\}$ if $n$ is odd and $S = \{1,2\}$ if $n$ is even.  If $p$ divides $\prod_{d \in S} D_{n,d}$ but not $D_n/\prod_{d \in S} D_{n,d}$, then $Y_1(n)$ has good reduction at $p$.
\end{prop}

\begin{proof}
Since $p\nmid D_{n,n}$, the curve $Y_0(n)$ has good reduction modulo $p$ by Proposition \ref{prop:Pbadimpliesdisc}(2). By Lemma \ref{Lramcollision}, $Y_1(n)$  has good reduction modulo $p$ unless there are two ramification points $(x_1, c_1)$ and $(x_2, c_2)$ of $Y_1(n)\to Y_0(n)$ that collide on the special fiber. In this case, $c_1$ and $c_2$ are roots of $\delta_n(1,c)$ that collide modulo $p$.  By the hypotheses, $c_1$ and $c_2$ are roots of $\Delta_{n, d}$ for some $d \in S$. By Lemma \ref{lemma:distinctHyperbolicComponents}, $c_1$ and $c_2$ lie on the boundary of distinct hyperbolic components of period $d$. But there is only one hyperbolic component of period $d$, since $d \in \{1, 2\}$.     
\end{proof}

\begin{rem}
Proposition \ref{P12} states that the set of primes of bad reduction for $Y_1(n)$ 
does not contain any prime $p$ which divides $D_n$ only because it divides $D_{n,1}$ (resp.\ $D_{n,1}D_{n,2}$) when $n$ is odd (resp.\ $n$ is even).  In Tables \ref{table:data1-7}--\ref{table:data8}, this means bad reduction cannot occur if $p$ appears in the $(1,1)$ (resp.\ and/or $(2,2)$) row but not in other rows.
For example, $Y_1(7)$ has good reduction modulo $p=29$.
\end{rem}

\section{Bad reduction}\label{Sbad}

As in \S\ref{Sgood}, we assume in this section that the family $\mc{F} = \{f_c\} = \{f(x,c)\}$ satisfies Morton's hypothesis (H) from Definition~\ref{defn:hypH} for $c = u^m$.  By $Y_1(n)$ we mean $Y_{1,\mc{F}}(n)$ (and the same for $Y_0(n)$). We further assume that $Y_1(n)$ is \emph{smooth} in characteristic zero, so that results about bad reduction are not vacuous: this holds when $f(x,c) = x^m +c$, by Theorem \ref{Tmultibrotsmooth}. In this section we give two criteria for the dynatomic curve to have bad reduction at a prime. 

\begin{prop}\label{Ppdividesnbadreduction}
Assume $f_c(x) = x^m + c$.  If $p$ is a prime not dividing $m$ but dividing $n > 3$, then $Y_1(n)$ has bad reduction modulo $p$.
\end{prop}

\begin{proof}
We show that the reduction of (any) point $z = (x_0, c_0)$ from Lemma \ref{Ltotram} is a singular point of $Y_1(n)_{\ol{\FF}_p}$.  Since $z$ is a ramification 
point of $Y_1(n) \to \AA^1$, it follows that $\partial \Phi_n / \partial 
x = 0$ at $z$, even in characteristic zero.  To compute $\partial \Phi_n/\partial c$ at $z$, we use (\ref{Etotram}) 
and compute $\prod_{\substack{d | n \\ d \neq 1, n}} \Phi_d(z)$.  
Writing $\Psi_n$ for $f_c^n(x) - x$, then
$$\prod_{\substack{d | n \\ d \neq 1, n}} \Phi_d = \frac{\Psi_n}{\Phi_1 \Phi_n} = \frac{\Psi_n}{\Psi_1}\prod_{d | n} \Psi_d^{-\mu(n/d)}.$$ The right hand side has an equal number of factors with positive and negative exponents, and all factors (other than the $\Psi_n$'s, which cancel out) have simple zeroes at $x_0$, so the product is equal to the product of the $x$-derivatives of the non-$\Psi_n$ terms at $x_0$. Since $\partial \Psi_d / \partial x = \zeta_n^d - 1$ at $z$,  the right hand side equals
$$\frac{1}{\zeta_n - 1}\prod_{\substack {d|n \\ d \neq n}} (\zeta_n^d - 1)^{-\mu(n/d)}.$$  Since $\zeta_n^d - 1$ is a $p$-adic unit whenever $n/d$ is not a power of $p$, the expression simplifies to $(\zeta_n^{n/p} - 1)/(\zeta_n - 1)$ up to $p$-adic units.  Now, using (\ref{Etotram}), there is an equality, up to $p$-adic units,
$$\frac{\partial \Phi_n}{\partial c}(z) = \frac{(m-1)m^{1/(m-1)}n}{(\zeta_n^{n/p}-1)(\zeta_n - 1)}.$$  If $p | n$ and $n > 3$, then the right hand side has positive $p$-valuation, and thus reduces to $0$ in characteristic $p$.  So the reduction of $z$ modulo $p$ is not smooth. 
\end{proof}

Our second criterion for bad reduction is in terms of the valuation of the discriminant:

\begin{thm}\label{Tbad}
Suppose $p$ is a prime not dividing $m\disc f(x,1)$.  If $p$ divides the order of $D_{n,n}$ (resp.\ $D_n$) exactly once, then $Y_0(n)$ (resp.\ $Y_1(n)$) has bad reduction modulo $p$. 
\end{thm}

\begin{rem}
For $n \leq 6$ and $f(x,c) = x^2 + c$, we verified that the \emph{only} odd primes $p$ of bad reduction for $Y_0(n)$ are those with $v_p(D_{n,n}) = 1$. 
\end{rem}

\begin{rem}
If $p$ divides $D_{n,n}$ (resp.\ $D_n$), 
then the branch points of $Y_0(n) \to \PP^1$ (resp.\ $Y_1(n) \to \PP^1$) collide modulo $p$.  The proof of Theorem \ref{Tbad} shows that when $p$ divides $D_{n,n}$ or $D_n$ exactly once, then exactly two \emph{ramification} points must also collide modulo $p$, which is what causes the bad reduction.   
\end{rem}

We start with two lemmas from algebraic number theory.

\begin{lemma}\label{Lalgnum}
Suppose the leading coefficient of $\phi \in \ZZ_p[x]$ is a unit. If $v_p(\disc(\phi)) = 1$, then there exists $a \in \ZZ_p$ such that $v_p(\phi(a)) = 1$.  Furthermore, there exist exactly two roots $r_1$ and $r_2$ of $\phi$ whose residues in $\ol{\FF}_p$ are the same as that of $a$, and these are the only two roots that have the same residue in $\ol{\FF}_p$.
\end{lemma}

\begin{proof}
Adjoining a root $x$ of $\phi$ gives an integrally closed extension of $\ZZ_p$, as otherwise  $p^{2r}$ would divide $\disc \phi$, where $p^r$ is the index of the ring $\ZZ_p[x]$ in its integral closure.  Suppose the prime $(p)$ splits as $\mf{p}_1^{e_1} \cdots \mf{p}_s^{e_s}$ in $\ZZ_p[x]$.  Since $p\mid \disc(\phi)$ exactly once, $e_i = 1$ for all but one $i$, for which $e_i = 2$ and $\mf{p}_i$ has residue field of cardinality $p$.  Assume without loss of generality that $i = 1$.  Then the reduction $\ol{\phi}$ of $\phi \pmod{p}$ has irreducible factorization $\ol{\phi}_1^2 \ol{\phi}_2 \cdots \ol{\phi}_s$, where $\ol{\phi}_1$ has degree $1$.  By Hensel's lemma, this lifts to a splitting $\phi = \beta \phi_2 \cdots \phi_s$, where $\beta \in \ZZ_p[x]$ is a lift of $\ol{\phi}_1^2$.  In particular, $\phi$ has exactly two roots $r_1$ and $r_2$ having the same residue in $\ol{\FF}_p$, and these roots are conjugate over $\ZZ_p$ since they are the roots of $\beta$. 
By the definition of the discriminant, $v_p(r_1 - r_2) = 1/2$.  Thus $r_1,r_2= a \pm b\sqrt{p}$, where $b$ is integral over  $\ZZ_p$ with valuation $0$ and $a = (r_1 + r_2)/2 \in \ZZ_p$ is the element we seek.
\end{proof}

\noindent\textbf{Notation.} Write $\QQ_p^{ur}$ for the completion of the maximal unramified extension of $\QQ_p$, and $\ZZ_p^{ur}$ for its valuation ring.

\begin{lemma}\label{LDeltannval}
Let $a \in \ZZ_p^{ur}$, and let $S_a/\ZZ_p^{ur}$ be the extension given by restricting 
$Y_0(n)_{\ZZ_p^{ur}} \to \AA^1$ above $c = a$.
Then the degree of the discriminant of $S_a/\ZZ_p^{ur}$ is $v_p(\Delta_{n,n}(a))$.
\end{lemma}

\begin{proof}
Write $R = \ZZ_p^{ur}$, $k= \ol{\FF}_p$, and $K = \QQ_p^{ur}$ for simplicity.
By \cite[Proposition 9d]{Morton96}, $\Delta_{n,n}(c)$ is the discriminant of the field extension for $Y_0(n) \to \AA^1$ over any field of characteristic zero; (the proposition is only stated for algebraically closed fields, but it holds over all fields since it is a polynomial identity and everything is defined over $\QQ$ by Theorem \ref{discriminantOfPhi}).

Identifying $a$ with a point in $\AA^1_R$ and $\ol{a}$ with its specialization to $\AA^1_k$, the map
$Y_0(n)_R \times_{\AA^1_R} \Spec \widehat{\mc{O}}_{\AA^1_R, \ol{a}} \to \Spec \widehat{\mc{O}}_{\AA^1_R, \ol{a}}$ is finite and free, and thus has discriminant given by a \emph{principal} ideal $(\delta)$ of $\widehat{\mc{O}}_{\AA^1_R, \ol{a}} \cong 
R[\![T]\!]$, where $T$ is a coordinate vanishing at $a$, e.g., $T = c-a$. Since the discriminant respects the unramified base change 
$\Spec \widehat{\mc{O}}_{\AA^1_R, \ol{a}} \to \AA^1_R$, 
it follows that $(\delta) = (\Delta_{n,n})$, where $\Delta_{n,n}$ is considered as a function of $T$.  Then the discriminant ideal of $S_a/\ZZ_p^{ur}$ is just the result of setting $T = 0$ in $(\Delta_{n,n})$, that is, $(\Delta_{n,n}(a))$.  The lemma follows.
\end{proof}

\begin{proof}[Proof of Theorem \ref{Tbad}] 
Let $R = \ZZ_p^{ur}$ and $K = \QQ_p^{ur}$.  We show that the special fiber of $Y_0(n)_R$ (resp.\ $Y_1(n)_R$) is singular when $p$ divides $D_{n,n}$ (resp.\ $D_n$) exactly once.  First note that, by Proposition \ref{Pnormal}, $Y_i(n)_R$ is the normalization of $\AA^1_R$ in $K(Y_i(n)_K)$ for $i \in \{0,1\}$.  Also recall that the polynomials $\Delta_{n,d}(c)$ are all defined over $\ZZ$ (Theorem \ref{discriminantOfPhi}) and have leading coefficients that are units in $\ZZ_p$ (\cite[top of p.\ 331]{Morton96}). 

If $v_p(D_n) = 1$, then by (\ref{Efactorization}) $p$ does not divide any $R_{n,d,e}$ with $d \neq e$ and there exists a unique $d \mid n$ such that $p$ divides $D_{n,d}$.  Then Lemma \ref{Lalgnum} guarantees an $a \in \ZZ_p$ (in particular, in $\ZZ_p^{ur}$) for which $v_p(\Delta_{n,d}(a)) = 1$, and for which exactly two roots $r_1$ and $r_2$ of $\delta_n(1,c)$ have the same reduction as $a$ in an appropriate extension of $\ZZ_p^{ur}$.  Likewise, if $v_p(D_{n,n}) = 1$, then Lemma \ref{Lalgnum} guarantees $a \in \ZZ_p^{ur}$ for which $v_p(\Delta_{n,n}(a)) = 1$ and two roots $r_1$ and $r_2$ of $\Delta_{n,n}(c)$ have the same reduction as $a$.
In either case, let $\ol{a}$ be the specialization of $a$ to $\AA^1_{\ol{\FF}_p}$.

Let $\mc{Z} = \Spec \widehat{\mc{O}}_{\PP^1_R, \ol{a}} \cong \Spec R[\![T]\!]$, and let $$\mc{Y}_i := Y_i(n)_R \times_{\AA^1_R} \mc{Z} \text{ for } i \in \{0, 1\}.$$  Now, $\mc{Y}_0$ and $\mc{Y}_1$ are flat over $\mc{Z}$, and the maps $\mc{Y}_0 \to \mc{Z}$ and $\mc{Y}_1 \to \mc{Z}$ correspond to extensions $B_0/R[\![T]\!]$ and $B_1/R[\![T]\!]$ as in \S\ref{Sprelims}.  Write $B_{i, \ol{K}}$, $R[\![T]\!]_{\ol{K}}$ for $B_i \otimes_R \ol{K}$, $R[\![T]\!] \otimes_R \ol{K}$, respectively (with $i \in \{0,1\}$).

Suppose $v_p(D_{n,n}) = 1$.  The only two (geometric) branch points on the generic fiber of $\mc{Y}_0 \to \mc{Z}$ come from $r_1$ and $r_2$. The map $Y_0(n) \to \PP^1$ has simple ramification above each of them (Proposition \ref{prop:branchedCover}(4a)), which means that the degree of the discriminant of $B_{0,\ol{K}}/R[\![T]\!]_{\ol{K}}$ is $2$.  By Lemma \ref{LDeltannval}, the degree of the discriminant of $\mc{Y}_0 \to \mc{Z}$ restricted above $c = a$ is $v_p(\Delta_{n,n}(a)) = 1$. The cardinality of a fiber of $\mc{Y}_0 \times_R \ol{K} \to \mc{Z} \times_R \ol{K}$ above a branch point is $\alpha-1$, where $\alpha$ is the degree of $\mc{Y}_0 \to \mc{Z}$. Thus the assumptions of Corollary \ref{Cbadredtest}(2) hold, and $\mc{Y}_0$ has a singular point.  Thus $Y_0(n)$ has bad reduction modulo $p$.

Now, suppose $v_p(D_n) = v_p(D_{n,d}) = 1$ for some $d < n$.  
The (geometric) branch points of the generic fiber of $\mc{Y} \to \mc{Z}$ are again $r_1$ and $r_2$, but this time there are $d$ ramification points above each of them, each with ramification index $n/d$ (Proposition \ref{prop:branchedCover}(4b)).  So the degree of the discriminant of $B_{1, \ol{K}}/R[\![T]\!]_{\ol{K}}$ is $2d(n/d - 1) = 2(n-d)$. On the other hand, $$v_p(\disc_x(\Phi_n(x, a))) = (n-d)v_p(\Delta_{n,d})= n-d$$ by Theorem \ref{discriminantOfPhi}.  Also, the cardinality of a fiber of $\mc{Y}_1 \times_R \ol{K} \to \mc{Z} \times_R \ol{K}$ above a branch point is $\alpha - (n-d)$, where $\alpha$ is the degree of $\mc{Y}_1 \to \mc{Z}$.  The assumptions of Corollary \ref{Cbadredtest}(2) hold, and $\mc{Y}_1$ has a singular point.  Thus $Y_1(n)$ has bad reduction modulo $p$.
\end{proof}

\begin{rem}
In the proof above, it is not really important to work over $\ZZ_p^{ur}$ rather than $\ZZ_p$.  We only use a ring with algebraically closed residue field in order to quote results from \S\ref{Sprelims}. 
\end{rem}

\section{Irreducibility of the dynatomic curve at primes exactly dividing the discriminant}\label{Sirreducibility}

We assume throughout this section that $f_c(x) = x^2 + c$.  Our goal is to show that the reduction of $Y_0(n)$ modulo $p$ is geometrically irreducible whenever $v_p(D_{n,n}) = 1$.  This will depend on a result about the so-called \emph{monodromy graph} (Definition \ref{Dmonodromygraph}), whose proof we defer until the next section.

First, we recall a few definitions related to \'{e}tale fundamental groups. If $S$ is a scheme and $s$ is a geometric point of $S$, then $\pi_1(S, s)$ is defined to be the automorphism group of the fiber functor $F_s$ (see, e.g., \cite[V]{SGA1} or \cite{Mezard}).  Note that different choices of geometric points lead to isomorphic fiber functors as long as $S$ is connected.  If $h: S \to S'$ is a morphism such that $h(s) = s'$, then by functoriality there is a homomorphism $\pi_1(S, s) \to \pi_1(S', s')$.  Suppose $S$ is a connected scheme with smooth compactification $\widetilde{S}$, and $T$ is a smooth irreducible component of $\widetilde{S} \setminus S$, of codimension $1$ in $\widetilde{S}$. Then the \emph{inertia group} at $T$ is the kernel of the natural map $\pi_1(S, s) \to \pi_1(S \cup T, s)$.

\begin{prop}\label{Pinertia}
Let $R$ be a mixed characteristic complete discrete valuation ring with fraction field $K$ and residue field $k$. Let $S$ be a smooth, projective, geometrically connected relative $R$-curve.  Let $x_1, \ldots, x_n$ be distinct $R$-sections of $S$.
Let $h: X \to S$ be a finite flat $R$-morphism, such that the special fiber of $X$ is reduced, the special fiber $h_k$ of $h$ is generically separable, and the generic fiber $h_K$ of $h$ is \'{e}tale outside $\{x_{1,K}, \ldots, x_{n,K}\}$, where $x_{i, K}$ is the intersection of $x_i$ with the generic fiber $S_K$ of $S$. Pick a $K$-point $s_K \in S_K \setminus \{x_{1,K}, \ldots, x_{n, K}\}$. 

Let $m \leq n$ be such that $x_1, x_2, \ldots, x_m$ are pairwise non-overlapping. If the action of the subgroup of $\pi_1(S_K \setminus \{x_{1, K}, \ldots, x_{n, K}\}, s_K)$ generated by the inertia groups at $x_{1,K}, \ldots, x_{m, K}$ on a geometric fiber of $h_K$ above $s_K$ is transitive, then the special fiber of $X$ is geometrically irreducible.
\end{prop}

\begin{proof}
We may assume that $k$ is algebraically closed.  We use the notation $h_k$ (resp.\ $X_k$, $S_k$, $x_{i, k}$) to represent the special fiber of $h$ (resp.\ $X$, $S$, $x_i$), and similarly for $K$ and the generic fiber.    Write $T$ for $x_1 \cup \cdots \cup x_n$, and define $T_k$ and $T_K$ similarly.  Let $U \subseteq X$ be equal to $h^{-1} (S \setminus T)$, with $U_k$ and $U_K$ defined accordingly. Without loss of generality, $s_K$ specializes to a point $s_k \in S_k \setminus T_k$.  Since $X$ has reduced special fiber, purity of the branch locus (\cite[Theorem 5.2.13]{Szamuely})  shows that $h_k|_{U_k}$ is \'{e}tale, and thus $U_k$ is smooth.  So connectedness of $U_k$ implies irreducibility of $U_k$ and consequently irreducibility of $X_k$.  We are reduced to showing that $U_k$ is connected, or equivalently, that the action of $\pi_1(S_k \setminus T_k, s_k)$ on the fiber $h_k^{-1}(s_k)$ is transitive. 

Since $x_1, \ldots, x_m$ do not overlap on the special fiber, Corollary \ref{Cgoodredtest} 
implies that $h_k$ is tamely ramified above $x_{1,k}, \ldots, x_{m,k}$ 
with the same ramification indices as $h_K$, and all points lying above 
$x_{1, k}, \ldots, x_{m,k}$ are smooth.  Let $\eta_{i,k}$ be a 
geometric point of $S_k \setminus T_k$ above the generic point of 
$\Spec(\widehat{\mc{O}}_{S_k, x_{i,k}})$.  Then the action of the inertia subgroup 
of $\pi_1(\Spec \widehat{\mc{O}}_{S_k, x_{i,k}} \setminus \{x_{i,k}\}, \eta_{i,k})$ at $x_{i,k}$ and that of the inertia subgroup of $\pi_1(\Spec \widehat{\mc{O}}_{S, x_{i,k}} \setminus \{x_i\}, \eta_{i,k})$ are equal on $h_k^{-1}(\eta_{i,k})$.  By functoriality (using the inclusions $\Spec(\widehat{\mc{O}}_{S_k, x_{i,k}}) \hookrightarrow S_k \setminus T_k$ and $\Spec (\widehat{\mc{O}}_{S, x_{i,k}} \setminus \{x_i\}) \hookrightarrow S \setminus T$), the action of the inertia subgroup of $\pi_1(S_k \setminus T_k, \eta_{i,k})$ at $x_{i,k}$ and that of the inertia subgroup of $\pi_1(S \setminus T, \eta_{i_k})$ at $x_i$ 
are equal on $h_k^{-1}(\eta_{i,k})$.  Choosing paths from the $\eta_{i,k}$ to $s_k$, the actions of the various inertia subgroups of $\pi_1(S_k \setminus T_k, s_k)$ and $\pi_1(S \setminus T, s_k)$ are equal on $h_k^{-1}(s_k)$ as well.

Identifying the fiber functors $F_{s_k}$ and $F_{s_K}$ by specialization, it suffices to show that the combined action of the inertia groups of $\pi_1(S \setminus T, s_K)$ at $x_1, \ldots, x_m$ on $h^{-1}(s_K)$ is transitive.  Since the combined action of the inertia groups of $\pi_1(S_K \setminus T_K, s_K)$ at $x_{1,K}, \ldots, x_{m, K}$ on $h_K^{-1}(s_K)$ is transitive by assumption, the desired result follows from functoriality.
\end{proof}

By Proposition~\ref{prop:branchedCover}, all ramification points of the topological cover $\pi_0 : X_0(n)(\CC) \to \PP^1(\CC)$ have ramification index two. This means that a small loop around a branch point of $\pi_0$ induces a transposition of the sheets of the cover. Here we define the sheets to be the connected components of the inverse image under $\pi_0$ of $A := \PP^1(\CC) \setminus (\mc{M} \cup [0, \infty])$. Note that $A$ is simply connected and that $\pi_0$ is unramified above $A$.  For convenience in this and the following section, we make the following definition.

\begin{defn}\label{Dmonodromygraph}
The {\it monodromy graph} for $X_0(n)$ is the graph $\Gamma(n)$ defined as follows: The vertices of $\Gamma(n)$ are the sheets of the topological branched cover $\pi_0:X_0(n)(\CC) \to \PP^1(\CC)$, and the edges correspond to branch points of $\pi_0$.  The edge associated to a branch point $c$ connects the two vertices corresponding to the sheets that are transposed by a loop around $c$. An edge is {\it finite} if $c$ is finite and is {\it infinite} otherwise.
\end{defn}

For example,
Figures~\ref{fig:n5graph} and \ref{n7graph} from Appendix A are the graphs of $\Gamma(5)$ and $\Gamma(7)$, respectively (the labeling of the vertices with 0's and 1's is explained in \S\ref{Smonodromy}).

\begin{cor}\label{Cgeneralmonodromygraph}
Let $p$ be an odd prime.  Consider the cover $h: \mathfrak{X}_0(n)_R \to \PP^1_R$, where $R$ is a finite extension of $\ZZ_p^{ur}$ with fraction field $K$.  Let $T$ be the set of branch points of the generic fiber $h_K$ of $h$, let $T' \subseteq T$ be those branch points whose specializations do not collide with that of some other branch point on the special fiber, and assume that all points in $T'$ are $K$-rational.  Identify the geometric generic fiber $h_K \times_K \ol{K}$ with an isomorphic map $h_{\CC}: X_0(n)_{\CC} \to \PP^1_{\CC}$ via a fixed isomorphism $\iota: \ol{K} \cong \CC$, and correspondingly identify the set of branch points of $h_{\CC}$ with $T$.  If the monodromy graph $\Gamma(n)$ remains connected when all edges corresponding to points in $T \setminus T'$ are removed, then the special fiber of $\mathfrak{X}_0(n)_R$ is irreducible.    
\end{cor}

\begin{proof}
Pick a base point $s \in \PP^1_{\CC} \setminus T$. Since the monodromy graph is connected even when the edges corresponding to points of $T \setminus T'$ are removed, the monodromy action of the subgroup $G \subseteq \pi_1^{\text{top}}(\PP^1(\CC) \setminus T, s)$ generated by small loops around the points of $T'$ on $h_{\CC}^{-1}(s)$ is transitive.  The same is therefore true for the action of the subgroup $H \subseteq \pi_1(\PP^1_K \setminus T, s)$ topologically generated by the inertia groups at the points of $T'$, where we identify $s$ with a geometric point of $\PP^1_K$ via $\iota$.  The corollary now follows from Proposition \ref{Pinertia} (which applies to $\mathfrak{X}_0(n)_R \to \PP^1_R$ 
by Remark \ref{Rxmplusc} and Propositions \ref{Pnormal}, \ref{Pgensep}, and \ref{Psamemodel}).
\end{proof}

\begin{cor}\label{CX0irred}
If $p$ is an odd prime such that $v_p(D_{n,n}) = 1$, then $Y_0(n)$ is geometrically irreducible modulo $p$.  If $p \nmid n$, then $Y_1(n)$ is also geometrically irreducible modulo $p$.
\end{cor}
\begin{proof}
Let $R$ be a finite extension of $\ZZ_p^{ur}$ with fraction field $K$ such that all branch points of $\mathfrak{X}_0(n)_K = X_0(n)_K \to \PP^1_K$ are $K$-rational. Recall that $\Delta_{n,n}(c)$ is defined over $\ZZ$ and its leading coefficient is a unit in $\ZZ_p$ (Theorem \ref{discriminantOfPhi} and \cite[Proposition 3.4]{MortonVivaldi}).  By Lemma \ref{Lalgnum}, there are exactly $2$ roots of $\Delta_{n,n}(c)$ with the same residue in $\ol{\FF}_p$, and these roots lie in $R$.  Thus exactly $2$ branch points of the generic fiber of $\mathfrak{X}_0(n)_R \to \PP^1_R$ collide on the special fiber, and neither of these specializes to $\infty$.  By Theorem \ref{thm:connected} below (whose proof does not depend on anything from this section), the monodromy graph $\Gamma(n)$ is connected when the edges corresponding to the two colliding branch points are removed (under some identification of $\ol{K}$ with $\CC$).  By Corollary \ref{Cgeneralmonodromygraph}, the special fiber $\mathfrak{X}_0(n)_k$ of $\mathfrak{X}_0(n)_R$ is irreducible.  Since $Y_0(n)_k$ is a dense open subset of $\mathfrak{X}_0(n)_k$ (Proposition \ref{Psamemodel}(1)), the first statement of the corollary follows.  The second statement follows from Proposition \ref{PX1irredX0irred}.
\end{proof}

\begin{ex}
Corollary \ref{CX0irred} shows that $Y_0(5)$ is geometrically irreducible modulo every odd prime $p$ (including $p = 3701$, the only prime of bad reduction,
with $v_p(D_{5,5}) = 1$).
\end{ex}

\begin{rem}
Suppose $v_p(D_{n,n}) = d$.  Applying Corollary \ref{Cnoncollisioncount} to 
$\Spec R[c]/\Delta_{n,n}(c) \to \Spec R$, where $R$ is as in Corollary \ref{CX0irred}, shows that there are at least $\deg \Delta_{n,n} - 2d$ roots of $\Delta_{n,n}(c)$ that are unique in their residue class in $\ol{\FF}_p$.  So at most $2d$ branch points of $\mathfrak{X}_0(n)_R \to \PP^1_R$ can collide on the special fiber.  If we could show that the monodromy graph $\Gamma(n)$ is connected, even when any $2d$ finite edges are removed, then the proof of Corollary \ref{CX0irred} would carry through to show that $X_0(n)$ is geometrically irreducible in characteristic $p$. 

In other words, if we know that $\Gamma(n)$ remains connected after removing $E$ finite edges, then $X_0(n)$ is geometrically irreducible in characteristic $p$ whenever $v_p(D_{n,n}) \leq E/2$.  A preliminary investigation for $n \leq 18$ indicates that one can take $E = n$ when $n$ is odd and $E = n/2$ when $n$ is even.
\end{rem}

\section{Connectedness of the monodromy graph}\label{Smonodromy}
In this section, we consider only the family of quadratic polynomials $f_c(x) = x^2 + c$. The purpose of this section is to show that the monodromy associated to the cover
	\[
    	\pi_0 : X_0(n) \to \PP^1
    \]
is highly transitive for all $n \in \mathbb{N}$. 
More precisely, the main result of the section
is about the connectedness of
the monodromy graph $\Gamma(n)$ from Definition \ref{Dmonodromygraph}.

\begin{thm}\label{thm:connected}
For each positive integer $n$, the monodromy graph $\Gamma(n)$ remains connected after removing any two finite edges.
\end{thm}

We rely on the combinatorial description of the monodromy action for complex quadratic polynomials as developed by Lau and Schleicher \cite[\textsection 3]{lau/schleicher:1994}; we refer the reader to their paper for more details.  Our argument refines and expands the techniques of \cite{BuffLei14} by combining the combinatorial and the complex dynamical descriptions of quadratic polynomial dynamics.  The methods of this section are independent of those used in the rest of the paper; the reader interested only in the result is directed to Theorem \ref{thm:connected}, proved at the end of the section.

\subsection{Itineraries}\label{sec:itineraries}

Let $\Sigma := \{0,1\}^\NN$ be the space of binary sequences. This space comes equipped with a metric: if we write $v = v_1v_2v_3\dots$ and $w = w_1w_2w_3\dots$, then
	\[
    	d(v,w) = \sum_{i=1}^\infty \frac{|v_i - w_i|}{2^i}.
    \]
Note that $d(v,w) \le 1/2^n$ if and only if $v$ and $w$ agree for the first $n$ terms or agree for all but the $n$-th term. We define the {\it shift} map from $\Sigma$ to itself to be the continuous function
	\[
    	\sigma(v_1v_2v_3\dots) = v_2v_3v_4\dots.
    \]
Now let $c \in \CC$, and suppose that $c$ does not lie in the Mandelbrot set $\mcM$. We may assign to each point $\alpha$ in the Julia set $\mcJ_c$ of $f_c$ an \textit{itinerary}, which is an element $\iota(\alpha) \in \Sigma$ determined by the relative positions of the iterates of $\alpha$ in the Julia set $\mcJ_c$. We direct the reader to \cite[p.\ 8]{lau/schleicher:1994} for the definition of the itinerary in this setting. We do not use the definition explicitly in the current paper; however, we note that the itinerary is defined in such a way that the map $\alpha \mapsto \iota(\alpha)$ is a homeomorphism from $\mcJ_c$ to $\Sigma$, and the map $f_c$ restricted to $\mcJ_c$ is conjugate to the shift map on $\Sigma$; that is, $\iota \circ f_c = \sigma \circ \iota$. (See \cite[Theorem 3.1]{lau/schleicher:1994}.) Therefore, much of the dynamics of $f_c$ on $\mcJ_c$ may be understood by studying the dynamics of the shift map on $\Sigma$.

For $c \in \CC \setminus \mcM$, all periodic points for $f_c$ lie in the Julia set $\mcJ_c$. Therefore each periodic point $\alpha$ has a well-defined itinerary, and the period of $\alpha$ under $f_c$ is precisely the period of $\iota(\alpha)$ under the shift map. If $\alpha$ has period $n$, we write
	\[
		\iota(\alpha) = \overline{v_1\dots v_n}.
	\]
Let $W := \CC \setminus (\mcM \cup [0,\infty))$; note that the map $\pi_0$ is unramified over $W$, so $\pi_0^{-1}(W) \to W$ is an \'etale cover. As $c$ varies within $W$, the periodic points of $f_c$ move continuously in $\CC$. Buff and Lei use this fact to prove that the itinerary may be used to label the sheets of the cover $\pi_1^{-1}(W) \to W$, so there is a one-to-one correspondence between sequences in $\Sigma$ of exact period $n$ and sheets of the cover $X_1(n) \to \PP^1$.  Since $X_0(n)$ is the quotient of $X_1(n)$ by the cyclic group generated by the automorphism $f_c$, it follows that there is also a one-to-one correspondence between the sheets of $X_0(n)$ and the set of \emph{orbits} of length $n$ (under the shift map) in $\Sigma$. Two period-$n$ elements $v, w \in \Sigma$ correspond to the same sheet of $X_0(n) \to \PP^1$ if and only if $w = \sigma^k(v)$ for some integer $k$.

\subsection{Kneading sequences}
For an angle $\theta \in (0,1)$, we define the \textit{kneading sequence} of $\theta$ as follows: Partition $\mathbb{R}/\mathbb{Z}$ into three sets,
\[S_0 := \left(\frac{\theta + 1}{2}, \frac{\theta}{2}\right);\ 
S_1 := \left(\frac{\theta}{2}, \frac{\theta + 1}{2}\right);
\ S_\star := \left\{\frac{\theta}{2},\frac{\theta + 1}{2}\right\}.\]
We assume $\RR/\ZZ$ is oriented so that $0 \in S_0$. The \emph{kneading sequence} of $\theta$ is the sequence of symbols $K(\theta) = v_1v_2v_3 \dots \in \{0,1,\star\}^{\NN}$
defined so that $(2^{i-1}\theta \mod 1) \in S_{v_i}$ for all $i \ge 1$. The kneading sequence always begins with a 1, and $K(\theta)$ contains a $\star$ if and only if $\theta$ is periodic under doubling. If $\theta$ has exact period $n$ under doubling, then
	\[
		K(\theta) = \overline{v_1\dots v_{n-1}\star},
	\]
where each of $v_1,\ldots,v_{n-1}$ is either 0 or 1. In this case, we denote by $K_0(\theta)$ and $K_1(\theta)$ the sequences $\overline{v_1\dots v_{n-1}0}$ and $\overline{v_1\dots v_{n-1}1}$,
respectively.
For a periodic itinerary $v = \overline{v_1\dots v_n}$, we let $v^\star$ denote the periodic sequence $v = \overline{v_1\dots v_{n-1}\star}$. Thus $K(\theta) = K_0(\theta)^\star = K_1(\theta)^\star$.

\begin{defn} We say an $n$-periodic itinerary is {\it maximal} if it is maximal (with respect to the lexicographic order) within its orbit under the shift map.  Note that there is a unique shift (modulo $n$) which takes a given itinerary to a maximal itinerary.
\end{defn}

\begin{lemma}\label{lem:maximal_itinerary} Let $v = \overline{v_1 \dots v_n}$ be a maximal itinerary.  Then $v_n = 0$ and the angle $\theta$ defined by the binary number $.\overline{v_1 \dots v_n}$ has kneading sequence $K(\theta) = v^\star = \overline{v_1 \dots v_{n-1} \star}.$
\end{lemma}

\begin{proof} See \cite[Lemma 4.2]{BuffLei14}.
\end{proof}

\begin{ex}
Consider the maximal itinerary $v = \overline{110}$, and set $\theta = 6/7$, which has binary expansion $\theta = .\overline{110}$. Then
	\[
    	\theta = 6/7 \in S_1;\ \ \ 2\theta = 5/7 \in S_1;\ \ \ 2^2\theta = 3/7 \in S_{\star};
    \]
so the kneading sequence of $\theta$ is equal to $\overline{11\star}$.
\end{ex}

\subsection{External rays of the Mandelbrot set}

Douady and Hubbard showed in \cite{DouadyHubbard} that there is a conformal isomorphism $\Phi : \CC \setminus \mcM \to \CC \setminus \overline{\DD}$, where $\overline{\DD}$ denotes the closed unit disk. An {\it external (parameter) ray} is the preimage $\mcR(\theta)$ of a ray $\{re^{2\pi i\theta} : r > 1\}$ under the map $\Phi$. We say that the ray $\mcR(\theta)$ {\it lands} at the parameter $c \in \partial\mcM$ if
	\[
    	\lim_{r \to 1^+} \Phi^{-1}(re^{2\pi i\theta}) = c.
    \]
It is a well-known open conjecture that $\mcR(\theta)$ lands for all $\theta \in \RR/\ZZ$. It was shown by Douady and Hubbard \cite[Theorem 13.1]{DouadyHubbard} that  $\mcR(\theta)$ lands if $\theta \in \QQ/\ZZ$.

If $\theta \in \QQ/\ZZ$ has odd denominator, then $\theta$ is periodic under the doubling map, and the ray $\mcR(\theta)$ lands at the root of a hyperbolic component of the same period. Conversely, if $c$ is the root of a period-$n$ hyperbolic component, then there are precisely two angles $\theta,\theta' \in \QQ/\ZZ$ such that $\mcR(\theta)$ and $\mcR(\theta')$ land at $c$ --- unless $c = 1/4$, in which case only $\mcR(0) = \mcR(1)$ lands at $c$; see \cite[\textsection 14.6]{DouadyHubbard}. The angles $\theta$ and $\theta'$ both have period $n$ under the doubling map, and they have the same $n$-periodic kneading sequence by \cite[Lemma 3.9]{schleicher:2000}. We may therefore assign a well-defined kneading sequence to any root $c$ of a hyperbolic component of $\mcM$. As an example, we give in Appendix~\ref{Tmono} the kneading sequences of all roots of period-5 hyperbolic components of $\mcM$ --- or, equivalently, the kneading sequences of all angles $\theta \in \QQ/\ZZ$ of exact period 5 under the shift map.

We now give an elementary lemma about kneading sequences that is used in \textsection \ref{sec:graph}.

\begin{lemma} \label{complexlanding} 
Let $\theta \in (0,1)$. The kneading sequence $K(\theta)$ begins with 11 if and only if $\theta < 1/3$ or $\theta > 2/3$. In this case, if the ray $\mcR(\theta)$ lands, it does so at a non-real parameter.
\end{lemma}

\begin{proof}
By the definition of the kneading sequence, the kneading sequences $K(\theta)$ and $K(-\theta)$ agree for any angle $\theta$. To prove the first claim, therefore, it suffices to assume that $0 < \theta \le 1/2$ and to show that $K(\theta)$ begins with 11 if and only if $\theta < 1/3$. Indeed, to say that the kneading sequence $K(\theta)$ begins with 11 is equivalent to saying that $\theta/2 < 2\theta < (\theta + 1)/2$, which in turn is equivalent to saying that $0 < \theta < 1/3$. This proves the first claim.

To prove the second claim, we may once again assume that $0 < \theta < 1/3$ due to the symmetry of $\mcM$ across the real axis. Suppose $\mcR(\theta)$ lands at a parameter $c \in \partial \mcM$. For such $\theta$, the rays $\mcR(0)$ and $\mcR(1/3)$, together with the main cardioid, separate $c$ from the real line. Thus $c$ is non-real, as claimed.
\end{proof}

\subsection{The monodromy action of the cover $X_1(n) \to \PP^1$}

Though our main result involves the monodromy associated to the cover $\pi_0 : X_0(n) \to \PP^1$, we begin with the combinatorial description of the monodromy action for the cover $\pi_1 : X_1(n) \to \PP^1$ appearing in \cite[Lemmas 3.4 and 3.5]{lau/schleicher:1994}. First, we require a definition.

\begin{defn} \label{Dbinarycomplement}
For $\eps \in \{0,1\}$, let $\widehat{\eps} = 1 - \eps$. For an element $v = v_0v_1v_2\dots$, we define the \textit{binary complement} of $v$ to be the sequence
	\[
		\widehat{v} = \widehat{v_0}\widehat{v_1}\widehat{v_2}\dots.
	\]
\end{defn}

The branch points of the cover $X_1(n) \to \PP^1$ are the roots of the period-$n$ hyperbolic components of $\mcM$ as well as the point at infinity, and the monodromy action associated to a loop around such a branch point $c$ depends on whether $c = \infty$ or $c$ is a {\it primitive} or {\it satellite} parabolic parameter. (Refer to Definition~\ref{defn:satellite/primitive} for these terms.) Parts (1) and (2) of the following proposition are due to Lau and Schleicher, and part (3) is due to Blanchard, Devaney, and Keen; see Proposition~\ref{prop:branchedCover} for other information about the branched cover $\pi_1$.

\begin{prop}[{\cite[Lemmas 3.4, 3.5]{lau/schleicher:1994}; \cite[Theorem 1.3]{blanchard/devaney/keen:1991}}]\label{prop:loops}
Let $c$ be a branch point for the cover $\pi_1 : X_1(n) \to \PP^1$. Making a small turn around $c$ induces the following permutation of the sheets of the cover:
	\begin{enumerate}
	\item Suppose $c$ is the root of a primitive component of $\mcM$, with parameter ray $\theta$ landing at $c$. Then both $K_0(\theta)$ and $K_1(\theta)$ have exact period $n$, and a small loop around $c$ interchanges the sheets labeled $K_0(\theta)$ and $K_1(\theta)$.
	\item Suppose $c$ is a bifurcation point from a period-$k$ component, with parameter ray $\theta$ landing at $c$. Write $n = qk$ for some integer $q \ge 2$, and let $e^{2\pi ip/q}$ be the multiplier of the parabolic orbit of $f_c$. (Note that $p$ must be coprime to $q$.) Exactly one of $K_0(\theta)$ and $K_1(\theta)$ has period equal to $n$; let $K'(\theta)$ denote this period-$n$ sequence. Then a small loop around $c$ sends the sheet labeled $K'(\theta)$ to the sheet labeled with the $kp'$-fold shift of $K'(\theta)$, where $pp' \equiv 1 \pmod q$.
	\item Suppose $c = \infty$. Then a small turn around $c$ transposes sheets whose representative itineraries are binary complements.
	\end{enumerate}
\end{prop}

\begin{defn}\label{defn:primitiveKS}
An $n$-periodic kneading sequence $K$ is {\it primitive} if both $K_0$ and $K_1$ have period $n$. Otherwise, $K$ is {\it imprimitive}.
\end{defn}

Recall from \textsection\ref{sec:itineraries} that the sheets of $\pi_0 : X_0(n) \to \PP^1$ are labeled by \emph{orbits} of period-$n$ binary sequences under the shift map; that is, two period-$n$ binary sequences $v$ and $w$ represent the same sheet if and only if $w = \sigma^k(v)$ for some $k \in \NN$. The only branch points for $\pi_0$ are the roots of \textit{primitive} hyperbolic components of $\mcM$ and the point at infinity. Applying Proposition~\ref{prop:loops} immediately yields the following:

\begin{cor}\label{cor:x0monodromy}
Let $c$ be a branch point for $\pi_0 : X_0(n) \to \PP^1$.
	\begin{enumerate}
	\item Suppose $c$ is the root of a primitive component of $\mcM$, with parameter ray $\theta$ landing at $c$. Then both $K_0(\theta)$ and $K_1(\theta)$ have exact period $n$, and a small loop around $c$ interchanges the sheets labeled $K_0(\theta)$ and $K_1(\theta)$.
	\item Suppose $c = \infty$. Then a small turn around $c$ transposes sheets whose representative itineraries are binary complements.
	\end{enumerate}
In particular, a loop around a single branch point always induces a transposition. 
\end{cor}

\subsection{Structure of the monodromy graph}\label{sec:graph}

Recall that $\Gamma(n)$ is the monodromy graph from Definition \ref{Dmonodromygraph}, associated to the cover $X_0(n) \to \PP^1$. 
 
As the vertices of $\Gamma(n)$ correspond to orbits of $n$-periodic itineraries, the number of vertices grows rapidly with $n$, with leading term $2^n/n.$  The number of finite edges can also be computed, and it has leading term $2^{n-1}$. However, the graph $\Gamma(n)$ is generally far from regular. 

Recall that each sheet of $X_0(n) \to \PP^1$ corresponds to an orbit of a period-$n$ itinerary under the shift map. Thus it makes sense to label the vertices of $\Gamma(n)$ by period-$n$ itineraries, with the understanding that two itineraries $v$ and $v'$ are associated to the same vertex $\mfv$ if and only if $v' = \sigma^k(v)$ for some $k$. We call such an itinerary a {\it representative itinerary} for $\mfv$. Note that each vertex of $\Gamma(n)$ has a unique maximal representative.

\begin{rem}\label{rem:vertices}
To distinguish vertices from their representatives, we use the typeface $\mathfrak{v}$ to denote a vertex of $\Gamma(n)$, reserving the standard typeface $v$ for a representative itinerary. If $\mfv$ is the vertex associated to the itinerary $v$, we denote by $\widehat{\mfv}$ the binary complement of $\mfv$; that is, the vertex associated to $\widehat{v}$.
\end{rem}

As Figure \ref{n7graph} suggests, the monodromy graph has additional structure; we require some definitions to make that structure clear.

\begin{defn} Let $v$ be an $n$-periodic itinerary, and suppose that the first $n$ letters of $v$ contain $A$ zeros and $B$ ones.  We define the \emph{disparity} $D(v)$ of an $n$-periodic itinerary to be the difference $A-B$.  The disparity $D(\mathfrak{v})$ of a vertex $\mathfrak{v}$ of the monodromy graph is defined to be the disparity of any representative itinerary.
\end{defn}

Note the disparity of a vertex is well-defined independent of the choice of representative itinerary. As $\mfv$ ranges over the vertices of $\Gamma(n)$, $D(\mfv)$ takes all values from $2-n$ to $n-2$ with the same parity as $n$. We also note that $D(\widehat{\mfv}) = -D(\mfv)$.

\begin{lemma} The disparities of two vertices connected by a finite edge differ by 2.
\end{lemma}

\begin{proof}  By Corollary~\ref{cor:x0monodromy}, the edge corresponding to a given branch point $c$ connects two vertices only if the kneading sequence for $c$ has the form $K = \overline{v_1 \dots v_{n-1} \star}$ and the associated itineraries $K_0 = \overline{v_1 \dots v_{n-1} 0}$ and $K_1 = \overline{v_1 \dots v_{n-1} 1}$ are representatives for the two vertices. The result is immediate.
\end{proof}

We arrange the vertices into rows by decreasing disparity, with the unique vertex of maximal disparity $n-2$ (represented, for example, by $\overline{11 \dots 10}$) on the top row and the vertex of minimal disparity on the bottom row.  The preceding lemma shows that there is only one possibility for any finite edge: it connects two vertices in adjacent rows.

\begin{rem} \label{complexremark} There is one final structural feature of the graph worth describing. If $c \in \mathbb{C}$ is a branch point of $\pi_0$, then its complex conjugate $\bar{c}$ is also a branch point of $\pi_0$, and the angles of the rays landing at $\bar{c}$ are the negatives of those landing at $c$; in particular, $c$ and $\bar{c}$ have the same kneading sequence. Therefore,  edges corresponding to {\it non-real} branch points come in pairs.
\end{rem}

Note that it is also possible for branch points which are not conjugate to have the same kneading sequence, accounting for the high multiplicity of some of the connections.

\subsection{Successor edges} 
One may use Lemma~\ref{lem:maximal_itinerary} to show that, given any vertex $\mfv$ in $\Gamma(n)$, there is at least one edge connected to $\mfv$. This construction is highly useful for proving the existence of edges of the monodromy graph; for example, it is used by Buff and Lei \cite{BuffLei14} to prove that $Y_1(n)$ is connected for all $n$. Because of it's importance in this section, we give a name to this construction.

\begin{defn} \label{successor} Let $v = \overline{v_1 \dots v_n}$ be a maximal $n$-periodic itinerary, so that $v_n = 0$.  Define the \emph{successor} of $v$ to be 
$$s(v) := \overline{v_1 \dots v_{n-1} 1}.$$
If $v$ is non-maximal, define its successor $s(v)$ to be the successor of the shift of $v$ which is maximal. Note that the successor increases the disparity by 2.
\end{defn}

\begin{lemma} \label{upwardsarrows} Let $v$ be an itinerary of period $n \geq 3$ satisfying $D(v) < n - 2$. Then $s(v)$ also has period $n$, hence $v^\star = s(v)^\star$ is the kneading sequence of a \emph{primitive} branch point of $\pi_0$. Consequently, the vertices $\mfv$ and $s(\mfv)$ associated to $v$ and $s(v)$, respectively, are connected by an edge in $\Gamma(n)$.

Moreover, if $D(\mathfrak{v}) \ge 0$, then there are at least two edges of $\Gamma(n)$ which connect $\mathfrak{v}$ to $s(\mfv)$.
\end{lemma}

Because the edges constructed in Lemma~\ref{upwardsarrows} are so useful later in this section, we give them a name:

\begin{defn}
Let $\mfv$ be a vertex in $\Gamma(n)$ with disparity $D(\mfv) < n - 2$. If $v$ is a representative itinerary for $\mfv$, we define the {\it successor} of $\mfv$ to be the vertex $s(\mfv)$ associated to the itinerary $s(v)$, and we call an edge $\mfv \to s(\mfv)$ constructed as in Lemma~\ref{upwardsarrows} a {\it successor edge}.
\end{defn}

\begin{proof}[Proof of Lemma~\ref{upwardsarrows}] Without loss of generality, assume $v$ is maximal, with 
\begin{equation} \label{Evectorlong}
v = \overline{ \underbrace{1 \dots 1}_{a_1} \ \underbrace{0 \dots 0}_{b_1} \ \underbrace{1 \dots 1}_{a_2} \ \cdots \cdots \cdots \ \underbrace{1 \dots 1}_{a_r} \ \underbrace{0 \dots 0}_{b_r}}.
\end{equation}
By maximality, $a_1 \geq a_j$ for all $1 \leq j \leq r$.  Since 
\begin{equation} \label{Evectorslong}
s(v) = \overline{ \underbrace{1 \dots 1}_{a_1} \ \underbrace{0 \dots 0}_{b_1} \ \underbrace{1 \dots 1}_{a_2} \ \cdots \cdots \cdots \ \underbrace{1 \dots 1}_{a_r} \ \underbrace{0 \dots 0}_{b_r-1} \ 1},
\end{equation}
a right shift puts $s(v)$ into maximal form, with at least $a_1 + 1$ leading consecutive ones. Since $a_1 + 1 > a_j$ for all $j$, $s(v)$ has a unique block of 1's of length $a_1 + 1$. Combining this with the fact that $s(v)$ has at least one 0 (otherwise, $v$ would have only one 0, making the disparity of $v$ equal to $n - 2$), we conclude that $s(v)$ must also have period $n$, hence the kneading sequence $v^\star$ connecting $v$ and $s(v)$ is primitive.
  By Lemma \ref{lem:maximal_itinerary}, it is the kneading sequence of a primitive branch point $c$, and the vertices $\mfv$ and $s(\mfv)$ are connected by the edge corresponding to $c$.

If also $D(v) \geq 0$, then since $n > 2$ and $v$ is maximal, $v$ begins with at least two consecutive ones.  By Lemma \ref{complexlanding}, the branch point $c$ guaranteed by the previous paragraph is non-real, and its conjugate $\bar{c}$ provides a second edge connecting $\mfv$ to $s(\mfv)$ in $\Gamma(n)$ by Remark \ref{complexremark}.
\end{proof}

Together with the edges provided by the branch point at infinity, we see that every vertex in the graph is connected to the vertex of maximal disparity via a path on which each {\it finite} edge occurs with multiplicity at least 2. In particular, $\Gamma(n)$ is connected (cf. \cite[Lemma~4.2]{BuffLei14}).

\subsection{Additional edges of $\Gamma(n)$}

Lemma \ref{upwardsarrows} is insufficient to guarantee that $\Gamma(n)$ remains connected after two finite edges are removed.  In this section, we construct additional edges of $\Gamma(n)$, which requires a more delicate analysis. We begin by describing two types of vertices that will require special attention later.

\begin{defn}
A vertex $\mfv$ is {\it repetitive} if its maximal representative has the form
	\begin{equation} \label{Evectorrep}
    	v= \overline{ \underbrace{1 \dots 1}_{b+1} \ \underbrace{0 \dots 0}_{a} \ \underbrace{1 \dots 1}_{b} \ \cdots \cdots \cdots \ \underbrace{1 \dots 1}_{b} \ \underbrace{0 \dots 0}_{a}}
    \end{equation}
for some positive integers $a$ and $b$.
\end{defn}

\begin{defn} 
A vertex $\mfv$ of $\Gamma(n)$ is {\it self-complementary} if $\widehat{\mfv} = \mfv$. In other words, if $v$ is a representative itinerary of $\mfv$, then $\mfv$ is self-complementary if there is an integer $1 \le k < n$ such that $\sigma^k(v) = \widehat{v}$. Note that, in this case, $n$ must be even and $D(\mfv)$ must be zero.
\end{defn}

If a vertex $\mfv$ is neither repetitive nor self-complementary, then it is connected to another vertex of equal disparity via a path that does not include a successor edge from $\mfv$ to $s(\mfv)$.

\begin{lemma} \label{Lrepsc} \label{lemma:alternatepath} Suppose $\mathfrak{v}$ is a vertex of $\Gamma(n)$ with disparity $2- n < D(\mathfrak{v}) < n-2$. If $\mathfrak{v}$ is not repetitive or self-complementary, then there exists a path in $\Gamma(n)$ connecting $\mathfrak{v}$ to a vertex $\mathfrak{w}$ with $D(\mathfrak{w}) = D(\mathfrak{v})$ and $\mathfrak{w} \ne \mathfrak{v}.$  Furthermore, this path can be chosen disjoint from the edges connecting $\mathfrak{v}$ to its successor constructed in Lemma \ref{upwardsarrows}.
\end{lemma}

\begin{cor}\label{cor:higher_disparity}
If $n \geq 3$ and $\mfv$ satisfies $D(\mfv) < n -2$ and is not repetitive or self-complementary, then $\mfv$ is connected to a vertex of higher disparity via a path that does not include a successor edge $\mfv \to s(\mfv)$.
\end{cor}
\begin{proof}
If $\mfv$ has minimal disparity $2 - n < 0$, then $\mfv$ is connected by an infinite edge to $\widehat{\mfv}$, which has (maximal) disparity $n - 2 > 0$. Now suppose $2 - n < D(\mfv) < n - 2$. By Lemma~\ref{lemma:alternatepath}, $\mfv$ is connected to a different vertex $\mfw$ of the same disparity as $\mfv$ via a path that does not include the successor edge from $\mfv$. Following the successor edge $\mfw \to s(\mfw)$ completes the desired path.
\end{proof}

\begin{proof}[Proof of Lemma~\ref{lemma:alternatepath}] To obtain the path, note that any vertex is connected to its binary complement by the branch point at infinity.  Therefore, given a vertex $\mathfrak{v}$, the successor function together with the branch point at infinity provide a path 
	\begin{equation}\label{eqn:path}
    \mfv \overset{E_1}{\longrightarrow} \widehat{\mfv} \overset{E_2}{\longrightarrow} s(\widehat{\mfv}) \overset{E_3}{\longrightarrow} \widehat{s(\widehat{\mfv})} \overset{E_4}{\longrightarrow} s(\widehat{s(\widehat{\mfv})}),
    \end{equation}
which induces the following map on disparity:
$$D(\mfv) \longrightarrow -D(\mfv) \longrightarrow -D(\mfv) + 2 \longrightarrow D(\mfv) - 2 \longrightarrow D(\mfv).$$
In this way we produce a path to a vertex of equal disparity.

We now claim that, for each $i \in \{1,2,3,4\}$, the edge $E_i$ is not a successor edge from $\mfv$ (as constructed in Lemma~\ref{upwardsarrows}). The edges $E_1$ and $E_3$ are infinite edges, so they are not successor edges. If $E_2$ were a successor edge from $\mfv$ to $s(\mfv)$, we would have $\mfv = \widehat{\mfv}$, contradicting our assumption that $\mfv$ is not self-complementary. Finally, $E_4$ cannot be a successor edge $\mfv \to s(\mfv)$ by considering the disparity. Thus the path given in \eqref{eqn:path} does not pass through a successor edge from $\mfv$ to $s(\mfv)$.

To complete the proof, we must show that the final vertex 
$s(\widehat{s(\widehat{\mfv})})$ in the path \eqref{eqn:path}
is different from $\mathfrak{v}$. In other words, we must show that if $v$ is a representative itinerary of $\mfv$, then $s(\widehat{s(\widehat{v})})$ is not a shift of $v$. We choose a representative $v$ of $\mathfrak{v}$ so that $\widehat{v}$ is maximal. Writing
$$v = \overline{ \underbrace{0 \dots 0}_{a_1} \ \underbrace{1 \dots 1}_{b_1} \ \underbrace{0 \dots 0}_{a_2} \ \cdots \cdots \cdots \ \underbrace{0 \dots 0}_{a_r} \ \underbrace{1 \dots 1}_{b_r}},$$
we have 
$$\widehat{v} = \overline{ \underbrace{1 \dots 1}_{a_1} \ \underbrace{0 \dots 0}_{b_1} \ \underbrace{1 \dots 1}_{a_2} \ \cdots \cdots \cdots \ \underbrace{1 \dots 1}_{a_r} \ \underbrace{0 \dots 0}_{b_r}},$$
so that we have
$$\sigma(s(\widehat{v})) = \overline{\underbrace{1 \dots 1}_{a_1+1} \ \underbrace{0 \dots 0}_{b_1} \ \cdots \cdots \cdots \ \underbrace{1 \dots 1}_{a_r} \ \underbrace{0 \dots 0}_{b_r-1}}.$$
Note that it is possible that $b_r = 1$ and $\sigma(s(\widehat{v}))$ ends in a 1.  The complement of $\sigma(s(\widehat{v}))$ is 
$$\widehat{\sigma(s(\widehat{v}))} = \overline{ \underbrace{0 \dots 0}_{a_1+1} \ \underbrace{1 \dots 1}_{b_1} \ \underbrace{0 \dots 0}_{a_2} \ \cdots \cdots \cdots \ \underbrace{0 \dots 0}_{a_r} \ \underbrace{1 \dots 1}_{b_r-1}}.$$

Now let $u$ be the maximal shift of $\widehat{\sigma(s(\widehat{v}))}$ and consider its successor $s(u)$; we claim that $s(u)$ does not represent $\mathfrak{v}$.  As $\widehat{\sigma(s(\widehat{v}))}$ has at least $a_1+1$ consecutive zeros, so does $u$ and its successor $s(u)$, unless the maximal shift $u$ of $\widehat{\sigma(s(\widehat{v}))}$ begins with the length-$b_1$ block of ones.  Since $v$ has at most $a_1$ consecutive zeros by maximality of $\widehat{v}$, we need only deal with the case that the maximal shift of $\widehat{\sigma(s(\widehat{v}))}$ is 
$$u = \overline{ \underbrace{1 \dots 1}_{b_1} \ \underbrace{0 \dots 0}_{a_2} \ \cdots \cdots \cdots \ \underbrace{0 \dots 0}_{a_r} \ \underbrace{1 \dots 1}_{b_r-1} \ \underbrace{0 \dots 0}_{a_1+1} }.$$
In this case, $b_1 \geq b_j$ for $j \in \{ 1, \ldots, r-1 \}$ and $b_1 \geq b_r -1,$ so the maximal shift of $s(u)$ is
$$\sigma^{n-1}(s(u)) = \overline{ \underbrace{1 \dots 1}_{b_1+1} \ \underbrace{0 \dots 0}_{a_2} \ \cdots \cdots \cdots \ \underbrace{0 \dots 0}_{a_r} \ \underbrace{1 \dots 1}_{b_r-1} \ \underbrace{0 \dots 0}_{a_1} }.$$
As noted above, $b_1$ is maximal among $b_j$ except possibly $b_r$, so $v$ cannot contain $b_1+1$ consecutive ones unless $b_1 +1 = b_r$, and the maximal shift of $v$ is 
$$\sigma^{n-b_r}(v) = \overline{ \underbrace{1 \dots 1}_{b_r} \ \underbrace{0 \dots 0}_{a_1} \ \underbrace{1 \dots 1}_{b_1} \ \underbrace{0 \dots 0}_{a_2} \ \cdots \cdots \cdots \ \underbrace{0 \dots 0}_{a_r} },$$
If $s(u)$ represents $\mathfrak{v}$, then as both are maximal representatives, $\sigma^{n-1}(s(u)) = \sigma^{n - b_r}(v)$.  Identifying the lengths of equal blocks, we see that $a_1 = a_2 = \cdots = a_r$ and $b_1 = b_2 = \cdots = b_{r-1} = b_r -1$; since $D(v) > 2 - n$, the common value of the $b_i$ is nonzero, so $\mathfrak{v}$ is repetitive, contradicting the hypotheses.
\end{proof}

\subsection{Admissibility and repetitive vertices}

The path constructed in Lemma \ref{Lrepsc} may not succeed in producing a new endpoint for a repetitive vertex; for example, if $v = \overline{01011}$, then the path from \eqref{eqn:path} 
$$\overline{01011} \to \overline{10100} \to \overline{11010} \to \overline{00101} = \overline{10100} \to \overline{11010},$$
and the final vertex is indeed a shift of the initial. (Note that we are representing the vertices in the path by itineraries.)   Fortunately, the repetition in these vertices allows us to construct other edges.  To do this, however, we will require the full machinery of admissibility for the first time.  A more complete account should include the theory of Hubbard trees, but it suffices for us to use a combinatorial criterion for admissibility due to Bruin and Schleicher.

Denote by $\Sigma^{\star}$ the collection of sequences $v \in \{ 0 , 1\}^{\mathbb{N}}$, together with $\star$-periodic sequences, subject to the condition that $v_1 = 1.$ (A sequence $v$ is {\it $\star$-periodic} if it has the form $v = \overline{v_1\dots v_{n-1}\star}$ with $v_i \in \{0, 1\}$). 

\begin{defn}[\cite{BS08}]\label{defn:admissible}
A sequence $v \in \Sigma^\star$ is {\it admissible} (resp., {\it real-admissible}) if it is the (necessarily $\star$-periodic) kneading sequence of a parabolic parameter $c \in \CC$ (resp., $c \in \RR$). 
\end{defn} 

\begin{defn} Given $v \in \Sigma^{\star}$, the {\it $\rho$-function} of $v$ is the function $\rho_v: \mathbb{N} \rightarrow \mathbb{N} \cup \{ \infty \}$ defined by
$$\rho_v(n) = \inf \{ k > n: v_k \ne v_{k-n} \}.$$
The orbit of $1$ under $\rho_v$ is called the {\it internal address} of $v$. If the orbit of 1 contains $\infty$, it is standard convention to truncate the internal address just before the $\infty$ term.
\end{defn}

As suggested by the terminology, if $c$ is in the Mandelbrot set, the internal address gives a guide map to $c$ in terms of a path through hyperbolic components of minimal period.

We now give a necessary and sufficient condition, due to Bruin and Schleicher, for an element of $\Sigma^\star$ to be admissible.

\begin{thm} [{\cite[Theorem 4.2]{BS08}}]\label{thm:adm_conditions} Given a $\star$-periodic sequence $v \in \Sigma^{\star}$ of period $n$, $v$ is admissible if and only if $m$ is an admissible period for all $1 \leq m < n$; that is, at least one of the following holds:
\begin{itemize}
\item[(A1)] $m$ is in the internal address of $v$.
\item[(A2)] there exists $k<m$ with $k \mid m$ and $\rho_v(k) > m$.
\item[(A3)] if $r \in \{1,\ldots,m\}$ is congruent to $\rho_v(m)$ modulo $m$, then $m$ is not in the $\rho_v$-orbit of $r$.
\end{itemize}
\end{thm}

With this criterion for admissibility, we can now address the case of repetitive vertices.

\begin{prop}\label{prop:repetitive} 
Suppose $\mathfrak{v}$ is a repetitive vertex of $\Gamma(n)$.  Then $\mathfrak{v}$ is connected to a vertex of higher disparity by a path which does not include a successor edge $\mfv \to s(\mfv)$.
\end{prop}

\begin{proof}
The maximal representative of $\mathfrak{v}$ has the form of \eqref{Evectorrep}; that is,
\[
	v= \overline{ \underbrace{1 \dots 1}_{b+1} \ \underbrace{0 \dots 0}_{a} \ \underbrace{1 \dots 1}_{b} \ \cdots \cdots \cdots \ \underbrace{1 \dots 1}_{b} \ \underbrace{0 \dots 0}_{a}}
\]
for some positive integers $a$ and $b$.

First suppose that $v$ has a single block of 1's. Consider the shift
\[
	\sigma(v) = \overline{\underbrace{1 \dots 1}_{b} \ \underbrace{0 \dots 0}_a} \ 1.
\]
Replacing the final 1 of $\sigma(v)$ with a 0 yields the itinerary
\[
	u = \overline{\underbrace{1 \dots 1}_{b} \ \underbrace{0 \dots 0}_{a + 1}},
\]
which is maximal in its orbit. If $\mfu$ is the vertex corresponding to $u$, then this implies that $\mfv = s(\mfu)$, so there is a disparity-{\it decreasing} edge from $\mfv$ to $\mfu$. Continuing in this way, one constructs a path from $\mfv$ to the vertex of minimal disparity, represented by $\overline{10\ldots00}$, which is connected by an infinite edge to the vertex of maximal disparity.

We henceforth assume that $v$ has at least two blocks of 1's, and we consider the shift $w$ of $v$ given by
\[
	w = \overline{ \underbrace{1 \dots 1}_{b} \ \underbrace{0 \dots 0}_{a} \ \cdots \cdots \cdots \ \underbrace{1 \dots 1}_{b} \ \underbrace{0 \dots 0}_{a} \ \underbrace{1 \dots 1}_{b + 1} \ \underbrace{0 \dots 0}_a}.
\]
If $w^{\star}$ is admissible, then there is an edge in $\Gamma(n)$ connecting the vertex $\mfv$ associated to $w$ (hence also to $v$) to the vertex $\mfu$ associated to the itinerary
\[
	u = \overline{ \underbrace{1 \dots 1}_{b + 1} \ \underbrace{0 \dots 0}_{a} \ \cdots \cdots \cdots \ \underbrace{1 \dots 1}_{b} \ \underbrace{0 \dots 0}_{a} \ \underbrace{1 \dots 1}_{b + 1} \ \underbrace{0 \dots 0}_{a - 1}}.
\]
Note that $u$ is primitive since it contains a single block of zeroes of length $(a - 1)$, unless $a = 1$, in which case the maximal shift of $u$ has a single block of 1's of length $2(b + 1)$. Moreover, by comparing the sizes of the blocks of 0's and 1's, $u$ and $s(v)$ do not represent the same vertex; hence there is a non-successor edge connecting $\mfv$ to a vertex of higher disparity. It therefore remains to show that $w^\star$ is admissible, which we do by direct computation of the $\rho_{w^\star}$ function, abbreviated $\rho$ here for clarity.

Let $m$ be an integer with $1 < m < n$; we must show that $m$ satisfies the three conditions of Theorem~\ref{thm:adm_conditions}. First, because $w^\star$ is repetitive, it is relatively straightforward to show that the internal address of $w^\star$ is
	\[
		1 \to (1 + b) \to (2 + b) \to \cdots \to (a + b) \to (n - a) \to (n - a + 1) \to \cdots \to (n - 1) \to n.
	\]
	
Suppose $1 \le m \le a + b$. If $m = 1$ or $m \ge b + 1$, then $m$ is in the internal address of $w^\star$, hence $m$ satisfies (A1). On the other hand, if $2 \le m \le b$, then $\rho(1) = b + 1 > m$, so $m$ satisfies condition (A2).

Now suppose that $m \ge a + b + 1$. If $m \ge n - a$, then $m$ is in the internal address of $w^\star$, hence $m$ again satisfies (A1). We now assume that $a + b + 1 \le m < n - a$, in which case $m$ is not in the internal address. If $(a + b)$ divides $m$, then $(a + b)$ is a {\it proper} divisor of $m$ since $m > a + b$ by assumption. The fact that $\rho(a + b) = n - a > m$ then implies that $m$ satisfies (A2). Finally, if $(a + b)$ does not divide $m$, then $r := \rho(m) - m < b < m$ satisfies $\rho(r) = b + 1$; since $m$ is not in the internal address, $m$ cannot be in the orbit of $r$, so $m$ satisfies condition (A3).
\end{proof}

\subsection{Real admissibility and self-complementary vertices}

We use the Milnor-Thurston theory of real quadratic dynamics to construct additional edges for $\Gamma(n)$. Recalling the definition of real-admissibility from Definition~\ref{defn:admissible}, we have the following:

\begin{thm}[{\cite[Theorem 12.1]{MT88}}] \label{MTtheorem} Let $v$ be an $n$-periodic itinerary. Then there exists a shift $w$ of $v$ such that $w^\star$ is real-admissible.
\end{thm}

The reader should be aware that the notation and coordinates of~\cite{MT88} differ slightly from the current conventions (necessary for complex dynamics) that are used in this paper; in particular, their notion of a minimal shift, which corresponds to a real dynamical system, is \emph{not} the same as lexicographic minimality in our coordinates.  However, the existence of a shift with an associated real-admissible kneading sequence is independent of these choices, and a conscientious reader can easily translate between the two notions if desired.

Unfortunately, Theorem \ref{MTtheorem} does not guarantee that every vertex of $\Gamma(n)$ has an adjacent edge associated to a real branch point.  If $\mathfrak{v}$ is represented by an itinerary $v$ for which $v^\star$ is real-admissible, it may be that $v^\star$ is imprimitive, in which case the real parabolic parameter with kneading sequence $v^\star$ is a {\it satellite} parabolic parameter. If this is the case, then we say $v$ is \emph{nearly imprimitive}, and write
\begin{equation}\label{eq:nearly_imprimitive}
	v = \overline{\underbrace{\mu \dots \mu}_{k-1 \text{ times}} \mu'}
\end{equation}
for some divisor $k > 1$ of $n$ and some sequence $\mu$ of length $\ell = n/k$, where $\mu'$ changes the $\ell$th entry of $\mu$. We now show that if $v$ is self-complementary, then this problem cannot arise.

\begin{prop} \label{primselfcomp} If $n > 2$ and $v$ is self-complementary, then $v$ is not nearly imprimitive.
\end{prop}

In order to prove the proposition, we require the following lemma.
\begin{lemma} \label{selfcomplementary}
Let $v$ be a self-complementary itinerary of period $n$. Then $n$ is even, and $v = \overline{w\widehat{w}}$
for some binary sequence $w$ of length $n/2$.
\end{lemma}

\begin{proof} 
That $n$ must be even follows immediately from the fact that a self-complementary itinerary has disparity zero.

Since $v$ is self-complementary, there exists some integer $k$ such that $\sigma^k(v) = \widehat{v}$; we may assume $0 < k < n$ since $v$ has period $n$. Using the fact that the shift map $\sigma$ commutes with binary complements, we therefore also have
	\[
		v = \widehat{\widehat{v}} = \widehat{{\sigma^k(v)}} = \sigma^k(\widehat{v}) = \sigma^{2k}(v).
	\]
This implies that $n$ divides $2k$, and since we assumed that $0 < k < n$, it follows that $k = n/2$.

Now write $v = \overline{wx}$, where $w$ and $x$ are binary sequences of length $n/2$. By the previous paragraph, $\overline{\widehat{w}\widehat{x}} = \widehat{v} = \sigma^{n/2}(v) = \overline{xw}$, so $x = \widehat{w}$, completing the proof.
\end{proof}

\begin{proof} [Proof of Proposition \ref{primselfcomp}] Suppose $v$ is nearly imprimitive, written as in \eqref{eq:nearly_imprimitive}. If $v$ is also self-complementary,  then $D(v) = 0$, and as disparity is additive over concatenated strings, 
$$0 = D(v) = (k-1)D(\mu) + D(\mu').$$
Since $\mu'$ and $\mu$ differ only at the $\ell$-th digit, $D(\mu') = D(\mu) \pm 2,$ hence
$$0 = k D(\mu) \pm 2.$$
It follows that $k = 2$ and $D(\mu) = \pm 1$.  But if $k = 2$, then by Lemma \ref{selfcomplementary} $\mu' = \widehat{\mu}$, and so $n/2 = \ell = 1$, contradicting our assumption that $n > 2.$
\end{proof}

\begin{cor} \label{selfcompreal} If $\mathfrak{v}$ is self-complementary, then there is a path connecting $\mathfrak{v}$ to a vertex of higher disparity which does not contain a successor edge from $\mfv$.
\end{cor}

\begin{proof} Let $v$ be the representative itinerary for $\mfv$ such that $v^\star$ is real-admissible, guaranteed by Theorem~\ref{MTtheorem}. As $v$ is self-complementary, $D(v) = 0$ and Lemma \ref{primselfcomp} guarantees that $v$ is not nearly imprimitive. The real parabolic parameter $c$ with kneading sequence $v^\star$ therefore corresponds to an edge from $\mathfrak{v}$ to a vertex $\mathfrak{w}$. 

If $D(\mathfrak{w}) > D(\mathfrak{v})$, then we are done, since the successor edges from $\mfv$ correspond to complex branch points and this new edge corresponds to a real branch point. 
If $D(\mathfrak{w}) < D(\mathfrak{v})$, then $D(\mathfrak{w}) = D(\mathfrak{v}) - 2 = -2.$ Then $\mathfrak{w}$ is connected by the branch point at infinity to the vertex $\widehat{\mfw}$ of disparity $2 > D(\mfv)$.
\end{proof}

\subsection{Connectivity of the monodromy graph}

We may now prove the main result of this section.

\begin{proof}[Proof of Theorem~\ref{thm:connected}] If $n \leq 4$, the graph $\Gamma(n)$ is easily computed and checked directly to satisfy the statement of the theorem, so assume $n > 4.$ Let $\Gamma'$ be a graph obtained from $\Gamma(n)$ by removing any two finite edges.  We proceed by constructing a path from any vertex of $\Gamma'$ to the vertex $\mathfrak{m}$ of maximal disparity, which is represented by 
$$m = \overline{ \underbrace{1 \dots 1}_{n-1} 0 }.$$

If $\mfv$ has negative disparity, then its complement $\widehat{\mfv}$ has positive disparity, and there is an infinite edge connecting $\mfv$ to $\widehat{\mfv}$; we therefore assume $D(\mfv) \ge 0$. By Lemma \ref{upwardsarrows}, repeated application of the successor function yields a path (in $\Gamma(n)$) from $\mfv$ to $\mfm$ where each edge has multiplicity at least two. Therefore, we may immediately conclude that every vertex is connected to $\mathfrak{m}$ in $\Gamma'$ unless $\Gamma'$ was obtained from $\Gamma(n)$ by the removal of a pair of successor edges coming from a single vertex $\mathfrak{v}_0$ of nonnegative disparity.

To complete the proof, it suffices to show that there is a path in $\Gamma(n)$ from $\mfv_0$ to a vertex of higher disparity that avoids the successor edges from $\mfv_0$ to $s(\mfv_0)$. If $\mfv_0$ is not repetitive of self-complementary, this is precisely Corollary~\ref{cor:higher_disparity}; if $\mfv_0$ is repetitive, this follows from Proposition~\ref{prop:repetitive}; if $\mfv_0$ is self-complementary, this is the content of Corollary~\ref{selfcompreal}. In all cases, therefore, we find that $\Gamma'$ is connected.
\end{proof}

\section{Good reduction at primes dividing $2^n\pm 1$} \label{Sc0c2}

In this section, we restrict to the case $f_c(x)=x^2+c$.
We provide a geometric
explanation for good reduction of $Y_0(n)$ modulo a prime $p$ such that $p$ divides $2^n \pm 1$, even though $p$ divides $D_{n,n}$.
As an application, we
verify certain examples of good reduction when $n=5,6$
and prove new cases of good reduction when $n=7,8,11$.

\begin{thm}\label{corol:Examples}
The curve $Y_0(n)$ has good reduction modulo 
$p$ in the following cases even though $p \mid D_{n,n}$:\\
$n=4$: $Y_0(4)$ for $p=3,5,17$; \cite[page 92]{MortonX4}\\
$n=5$: $Y_0(5)$ for $p=3,11,31$; \cite[\S 7]{FPS} \\
$n=6$: $Y_0(6)$ for $p=3,5,7,13$; \cite[\S 2]{Stoll} \\
$n=7$: $Y_0(7)$ for $p=3,43,127$;\\
$n=8$: $Y_0(8)$ for $p=3,5,17,257$;\\ 
$n=11$: $Y_0(11)$ for $p=3,23,89,683$;\\
\end{thm}

When $n$ is prime and $p \nmid n$, then $Y_1(n)$ has good reduction 
modulo $p$ if and only if $Y_0(n)$ does by Proposition \ref{PX1goodX0good}. 
We computed that $Y_1(6)$ has bad reduction modulo $p=3,5$ and good reduction modulo $p=7,13$ and 
that $Y_1(8)$ has good reduction modulo 
$p=3,5,17,257$ (see Appendix A).

The strategy of the section is to
study the reduction $\ol{Y}_0(n):= Y_0(n) \times_\ZZ \FF_p$
above $c=0$ and $c=-2$ modulo primes such that
$p\mid (2^n \pm 1)$. 
(These values of $c$ are special because they yield maps that come from algebraic groups, namely power and Chebyshev maps, respectively.)
For particular $n$ and $p$, we prove that:
\begin{enumerate} \item $\overline{Y}_0(n)$
has no singularity above $\bar{c}=0$ and $\bar{c}=-2$,
and \item the power of $p$ dividing $D_{n,n}$ is fully explained by the points specializing to $\bar{c}=0$ and $\bar{c}=-2$.  \end{enumerate}

For part (1), we:
\begin{enumerate}[(i)]
\item determine the number of branch points which specialize to $c=0$ and $c=-2$; 
\item determine information about the degree of the discriminant above $\bar{c} =0$ and $\bar{c}=-2$ in characteristic $p$;
and
\item verify that the contribution to the degree of the discriminant above $c=0$ and $c=-2$ is the same in characteristics zero and $p$.
\end{enumerate}

For part (i), we need to compute 
$\Delta_{n,n}(c)$ modulo $p$.  
This step is the limiting factor for extending 
Theorem \ref{corol:Examples} to the case of larger $n$.

There are a few other cases that we expect can be solved using the techniques of this section, which 
we leave for an interested reader. 

\begin{qn} \label{Qtomorrow}\
\begin{enumerate}
\item
Does $Y_0(9)$ have good reduction modulo $p=3,7,19,73$?

\item Does $Y_0(10)$ have good reduction modulo $p=3,5,11,33,41$?

\item Does $Y_0(13)$ have good reduction modulo $p=3, 2731, 8191$?
\end{enumerate}
\end{qn}


\subsection{Points of formal period $n$ above $c=0$ and $c=-2$}

We describe (the orbits of) the points of formal period $n$ when $c=0$ and $c=-2$ and
their collision behavior modulo $p$.

\begin{lemma}\ \label{Lpointformalperiodn} \begin{enumerate} 
\item The points $x$ of formal period $n$ above $c=0$ occur at \[R_0:=\mu_{2^n-1}\setminus\bigcup_{\substack{d\mid n\\ d\neq n}} \mu_{2^d-1}.\] 
\item The points $x$ of formal period $n$ above $c=-2$ occur at \[R_{-2}:=\{\zeta+\zeta^{-1}:\zeta\in \mu_{2^n-1}\cup\mu_{2^n+1}\setminus\bigcup_{\substack{d\mid n\\ d\neq n}} \mu_{2^d-1}\setminus\bigcup_{\substack{d\mid n \\ d\neq n}} \mu_{2^d+1}\}.\]
\end{enumerate} \end{lemma}

Note that these $(2^n\pm 1)$-st roots of unity are not necessarily primitive. 
Note that $2^d+1$ divides $2^n-1$ if $n/d$ even 
and divides $2^n+1$ if $n/d$ is odd.

\begin{proof}\ \begin{enumerate} 
\item If $c=0$, then
 \[\Phi_n(x,0)= \prod_{d\mid n} (x^{2^d}-1)^{\mu(n/d)}.\]
 So $\Phi_n(x,0)$ is the product of those cyclotomic polynomials $C_i(x)$ where $i$ is a divisor of $2^n-1$ but not of $2^d-1$ for $d$ a proper divisor of $n$.  

\item If $z\in\CC^\times$, then $(z+z^{-1})^2-2=z^2+z^{-2}$. If $\zeta\in\mu_{2^n-1}\cup\mu_{2^n+1}$, then $\Psi_n(\zeta+\zeta^{-1},-2)=0$, i.e., $\zeta+\zeta^{-1}$ is an $n$-periodic point. 
If $\zeta+\zeta^{-1}=\eta+\eta^{-1}$, then $\eta=\zeta^{\pm 1}$. Thus there are $2^{n-1}$ numbers $\zeta+\zeta^{-1}$ with $\zeta\in\mu_{2^n-1}$ and $2^{n-1}+1$ such numbers with $\zeta\in\mu_{2^n+1}$. 
Since $\mu_{2^n-1}\cap\mu_{2^n+1}=\{1\}$, $|\mu_{2^n-1}\cup\mu_{2^n+1}|=2^n$, which is the degree of $\Psi_n(x,-2)$. Thus the roots of $\Psi_n(x, -2)$ are distinct 
and all arise in this way.

If $\zeta+\zeta^{-1}$ is also a $d$-periodic point, then $f_{-2}^d(\zeta+\zeta^{-1})=\zeta+\zeta^{-1}$. 
The roots of $\Phi_n(x,-2)$ are the points of exact period $n$, which by induction are the values $\zeta+\zeta^{-1}$ 
such that $\zeta$ is not in any $\mu_{2^d-1}\cup\mu_{2^d+1}$ for $d\mid n$
and $d \not = n$. 
\end{enumerate} \end{proof}

The points of $Y_1(n)$ above $c=0$ and $c=-2$ are the points of $R_0$ and $R_{-2}$, respectively.
The points of $Y_0(n)$ are the orbits of these under the $C_n$-action defined as follows:
for $c=0$, the action is $(1\in C_n)\mapsto(\zeta\mapsto\zeta^2)$;
for $c=-2$, the action is $(1\in C_n)\mapsto(\zeta+\zeta^{-1}\mapsto\zeta^2+\zeta^{-2})$. 
In other words, the action squares the root of unity in both cases.

If the reduction of $Y_1(n)$ modulo $p$ has a singularity above $\bar{c}=0$ (resp.\ $\bar{c}=-2$),
then points of formal period $n$ above $c=0$ (resp.\ $c=-2$) collide modulo $p$.
Let $\overline{R}_0$ and $\overline{R}_{-2}$ denote the reduction of $R_0$ and $R_{-2}$, respectively, modulo $p$.

\begin{lemma}\
$|\overline{R}_0|<|R_0|$ (resp.\ $|\overline{R}_{-2}|<|R_{-2}|$) if and only if
$p \mid (2^n-1)$  (resp.\ $p\mid (2^n-1)$ or $p\mid (2^n+1)$). \end{lemma}

\begin{proof}
The polynomial $x^k-1$ is separable modulo $p$ if and only if $p\nmid k$. Thus the $k$-th roots of unity are distinct in $\ol{\FF}_p$ if and only if $p\nmid k$. It follows that if $p\nmid (2^n-1$), then $|\ol{R}_0|=|R_0|$. 

If $p\mid (2^n-1)$, then it suffices to show that there are two $(2^n-1)$-st roots of unity \emph{that are not $(2^d-1)$-st roots for $d\mid n$ and $d<n$} that coalesce modulo $p$; that is, two elements of $R_0$ coalesce modulo $p$. To do this, let $\zeta$ be a primitive $(2^n-1)$-st root of unity. Suppose that $2^n-1=p^km$, where $p\nmid m$. Then $\zeta^i\equiv\zeta^j$ modulo some prime $\mfp$ of $\ZZ[\zeta]$ above $p$ if and only if $i\equiv j\pmod{m}$. Since $\nu(n)=|R_0|>m$, there must be two elements of $R_0$ coalescing in $\ol{\FF}_p$ if $p\mid (2^n-1)$.

For $c=-2$, then a similar argument applies, with the main difference being that $\zeta^i+\zeta^{-i}\equiv\zeta^j+\zeta^{-j}\pmod{\mfp}$ if and only if $i\equiv \pm j\pmod{m}$, since $\zeta^i+\zeta^{-i}\equiv\zeta^j+\zeta^{-j}$ is equivalent to $\zeta^{-i}(\zeta^{i+j}-1)(\zeta^{i-j}-1)\equiv 0$.
\end{proof}

\subsection{Degree of discriminant}

We now define four numbers which measure various contributions to the degree of the discriminant
of the cover $\pi_0: Y_0(n) \to {\mathbb A}^1$ and its reduction $\bar{\pi}_0$ modulo $p$.
Let $\bar{\rho}_0$ (resp.\ $\bar{\rho}_{-2}$) be the contribution to the degree of the discriminant of 
the normalization of $\bar{\pi}_0$ over $\bar{c}=0$ (resp.\ $\bar{c}=-2$). 
The number of ramification points of $\bar{\pi}_0$ above $\bar{c}=0$ and 
$\bar{c}=-2$
is easy to compute combinatorially, but wild ramification often makes it 
harder to compute $\bar{\rho}_0$ and $\bar{\rho}_{-2}$.

Define $\rho_0$ (resp.\ $\rho_{-2}$) to be the degree of the discriminant of 
$\pi_0$ above the branch points which specialize to $\bar{c}=0$ (resp.\ $\bar{c} = 2$) modulo $p$.

The values $\rho_0$ and $\bar{\rho}_0$ 
(resp.\ $\rho_{-2}$ and $\bar{\rho}_{-2}$) determine whether or not $\ol{Y}_0(n)$ 
has a singularity above $\bar{c}=0$ (resp.\ $\bar{c}=-2$).

\begin{lemma} \label{prop:badat0-2}
With notation as above, $\bar{\rho}_0 \leq \rho_0$ and $\bar{\rho}_{-2} \leq \rho_{-2}$.
The curve $\ol{Y}_0(n)$ is smooth above $\bar{c}=0$
(resp.\ $\bar{c}=-2$) if and only if $\rho_0=\bar{\rho}_0$ (resp.\ $\rho_{-2}=\bar{\rho}_{-2}$).
\end{lemma}

\begin{proof}
Consider $\bar{c} = 0$; (the case $\bar{c} = -2$ is exactly the same).  In the notation used in Proposition~\ref{Pfultontest}, $\rho_0 = d_\eta$ and $\bar{\rho}_0 \leq d_s$, with equality holding if and only if $\ol{Y}_0(n)$ 
is normal (that is, smooth) above $\bar{c} = 0$.  Since $d_{\eta} = d_s$ (Proposition \ref{Pfultontest}), it follows that $\rho_0 = d_{\eta} = d_s \geq \bar{\rho}_0$, with equality holding if and only if 
$\ol{Y}_0(n)$ 
is smooth above $\bar{c} = 0$.  This proves the lemma.
\end{proof}

We do not have a good method to compute
$\rho_0$ and $\rho_{-2}$ in general but the following lemma is 
useful to compute them in specific cases.

\begin{lemma} \label{Lhowtochar0}
The values $\rho_0$ and $\rho_{-2}$ are the integers such that
there exists a polynomial $g(c)\in\ZZ[c]$ with $g(0) g(-2) \not \equiv 0 \pmod p$ and 
\[\Delta_{n,n}(c) \equiv c^{\rho_0}(c+2)^{\rho_{-2}}g(c) \bmod p.\]
\end{lemma}

\begin{proof}
By Proposition \ref{prop:branchedCover}(4a), 
each branch point of $\pi_0$ has one ramification point $\eta$ in its fiber, and $\eta$ has ramification index $2$.
So $\rho_0$ equals the number of branch points which specialize to $\bar{c}=0$.
Thus $\rho_0$ is the integer such that $\Delta_{n,n}(c) \equiv c^{\rho_0} g_0(c) \bmod p$
with $g_0(0) \not \equiv 0 \bmod p$.
Similarly, $\rho_{-2}$ is the integer such that 
$\Delta_{n,n}(c) \equiv (c+2)^{\rho_{-2}}g_{-2}(c) \bmod p$ with $g_{-2}(-2) \not \equiv 0 \bmod p$.
\end{proof}

\subsection{Other contributions to the discriminant}
\label{Sother}

We also use the polynomial $\Delta_{n,n}(c)$ to check that $\ol{Y}_0(n)$ 
is smooth everywhere, not just over the points $\bar{c}=0,-2$.
In order to guarantee that $Y_0(n)$ has good reduction at a prime $p\mid (2^n\pm 1)$, it is necessary to show that no points $\alpha \in \ol{\FF}_p$ other than $0, -2$ contribute to the power of $p$ in the discriminant $D_{n,n}$.

The points $\alpha\in\ol{\FF}_p$ that contribute to 
the power of $p$ in $D_{n,n}$ are exactly those $\alpha$ for which $(c-\alpha)^2\mid(\Delta_{n,n}(c)\bmod p)$. 
These $\alpha$ can be found by factoring $\Delta_{n,n}(c)$ modulo $p$ directly, but it is easier computationally to compute 
\begin{equation} \label{EG02}
\Gamma_{n,p}(c):=\gcd(\Delta_{n,n}\bmod p,\Delta'_{n,n}\bmod p)\in\FF_p[c].
\end{equation}
Then $(c-\alpha)^2\mid(\Delta_{n,n}(c)\bmod p)$ if and only if $(c-\alpha)\mid\Gamma_{n,p}(c)$.

If $\bar{c}=0,-2$ are the only roots of $\Gamma_{n,p}(c)$,
then the curve $\ol{Y}_0$ is smooth except possibly above $\bar{c}=0,-2$.
This is the outcome in all but one of the cases we computed.
The one exception is $n=6$, $p=3$, when $\Gamma_{6,3}=c^5(c+2)^2(c^2+2c+2)$. Hence, in this case, we must check that $\ol{Y}_0$ is smooth above the roots of $c^2+2c+2$ as well, which we do in Example~\ref{ex:n=6p=3}.
Coincidentally, this is the only case we computed for which $p^2\mid(2^n-1)$.
The next such example is when $n=9$, $p=3$. 

\subsection{Degree of discriminant in characteristic $p$}

This section contains an exact value (resp.\ lower bound) for $\bar{\rho}_0$ and $\bar{\rho}_{-2}$
when the ramification is tame
(resp.\ wild).

Consider the reduction $\overline{R}_0$ of $R_0$ modulo $p$. 
Let $t_0$ be the number of orbits of $\overline{R}_0$ 
under the action of $C_n$.  For each orbit $O_i$, its ramification index $e_i$ is the number of orbits in $R_0$ specializing to it.  
Then $\nu(n)/n = \sum_{i=1}^{t_0} e_i$.  
Similarly, let $t_{-2}$ be the number of orbits of $\overline{R}_{-2}$ under the action of $C_n$; for each orbit
$O_i'$, let $e_i'$ denote the ramification index;
then $\nu(n)/n = \sum_{i=1}^{t_{-2}} e_i'$. 

The cover $\bar{\pi}_0$ is tamely ramified above $\bar{c}=0$ (resp.\ $\bar{c}=-2$) when $p \nmid e_i$ for any orbit $O_i$
(resp.\ $p \nmid e_i'$ for any orbit $O_i'$).

\begin{lemma} \label{Ltamecont}
In the tame case, the degree of the discriminant above $\bar{c}=0$ (resp.\ $\bar{c}=-2$) is $\bar{\rho}_0 = \sum (e_i-1) = \nu(n)/n -t_0$ (resp.\
$\bar{\rho}_{-2} = \sum (e_i-1) = \nu(n)/n -t_{-2}$).
\end{lemma}

In the wild case, let $s_0$ (resp.\ $s_{-2}$)
be the number of 
orbits $O_i$ of $\overline{R}_0$ 
(resp.\ orbits $O_i'$ of $\overline{R}_{-2}$) 
for which $p$ divides the ramification index $e_i$ (resp.\ $e_i'$).  

\begin{lemma} \label{Lwildcont}
In the wild case, 
the degree of the discriminant above $\bar{c}=0$ (resp.\ $\bar{c}=-2$) satisfies
\[\bar{\rho}_0 \geq \sum (e_i-1) + s_0 = \nu(n)/n -t_0 + s_0,\] (resp.\ 
\[\bar{\rho}_{-2} \geq \sum (e_i-1) + s_{-2} = \nu(n)/n -t_{-2} + s_{-2}).\]
\end{lemma}

\begin{proof}
At each ramification point above $\bar{c}=0$, the contribution to the degree of the discriminant of $\bar{\rho}_0$ is an integer.  At each wildly ramified point, the contribution is strictly greater than $e_i-1$.
\end{proof}

\subsection{Ramification indices modulo $p$}

In general, the combinatorics of the collisions depends on the divisors of $n$ and $2^n-1$ and their $p$-adic valuations. When $n$ is prime, the ramification indices at $\bar{c}=0$ and $\bar{c}=-2$ for $\bar{\varphi}: \overline{Y}_1(n)\to \PP^1$ and $\bar{\pi}_0: \overline{Y}_0(n)\to \PP^1$
can be computed precisely.

\begin{lemma} \label{Lramindex}
Suppose $n$ is an odd prime and $2^n \pm 1 = p^k m$ with $p \nmid m$.
Then $\bar{\pi}_1:\overline{Y}_1(n)\to \PP^1$ has the following ramification indices at $\bar{c}=0$ and $\bar{c}=-2$: 
\begin{enumerate} 
\item If $p\mid (2^n-1)$, then the fiber of $\bar{\pi}_1$ above $0$
has $m$ points,
one with ramification index $p^k-1$ and the other $m-1$ with ramification index $p^k$. 

\item If $p\mid (2^n - 1)$, 
then the fiber of $\bar{\pi}_1$ above $-2$ has one point with ramification index $e=\frac{p^k-1}{2}$ and $\frac{m-1}{2}$ points with $e=p^k$, and the remaining points are unramified. 

\item If $p \mid (2^n + 1)$ and $p \not = 3$, then the fiber of $\bar{\pi}_1$ above $-2$ has one point with ramification index $e=\frac{p^k-1}{2}$ and one
with $e=p^k-1$ and $\frac{m-3}{2}$ points with $e=p^k$, and the remaining points are unramified.

\item Let $n > 3$.  If $p \mid (2^n + 1)$ and $p = 3$, then $k=1$ (since $n$ is prime), and the fiber of $\bar{\pi}_1$ above $-2$ has 
$\frac{m-1}{2}$ points with $e=3$, and the remaining points are unramified.
\end{enumerate} \end{lemma}

\begin{proof} \
\begin{enumerate} 
\item  
By Lemma~\ref{Lpointformalperiodn}(1), the points of $Y_1(n)$ above $c=0$ occur when $x \in R_0$.
Since $n$ is prime, the elements of
$R_0 = \mu_{2^n-1} \setminus\{1\}$ can be identified with pairs $(a,b) \in \ZZ/p^k\ZZ \times \ZZ/m\ZZ \setminus \{(0,0)\}$.
The pairs $(a_1,b_1)$ and $(a_2,b_2)$ coalesce modulo $p$ if and only if $b_1 = b_2$. Thus the ramification index is $e=p^k-1$ when $b=0$ and is $e=p^k$ otherwise.

\item By Lemma~\ref{Lpointformalperiodn}(2), the points of $Y_1(n)$ above $c=-2$
occur when $x \in R_{-2}$. 
Since $n$ is an odd prime, 
the elements of 
$R_{-2}$ are of the form $x=\zeta+\zeta^{-1}$, 
where $\zeta \in (\mu_{2^n-1} \setminus \{1\}) \cup (\mu_{2^n + 1} \setminus \mu_3)$.

If $p \mid (2^n-1)$, then the 
$2^{n-1}-1$ values $x=\zeta+\zeta^{-1}$ for $\zeta \in \mu_{2^n+1} \setminus \mu_3$ remain distinct modulo $p$.
Similarly to part (1), the remaining elements can be identified with pairs $(a,b) \in 
\ZZ/p^k\ZZ \times \ZZ/m\ZZ \setminus \{(0,0)\}$ 
up to equivalence $(a,b) \equiv (-a,-b)$.
Modulo $p$, two pairs $(a_1,b_1)$ and $(a_2,b_2)$ coalesce if and only if $b_1 = \pm b_2$.
Thus the ramification index is $e=(p^k-1)/2$ when $b=0$ and is $e=p^k$ for the $(m-1)/2$ choices of $\pm b \not = 0$.

\item If $p \mid (2^n+1)$, then the $2^{n-1} -1$ values $x=\zeta+\zeta^{-1}$ with $\zeta \in \mu_{2^n-1}\setminus\{1\}$ remain distinct modulo $p$.  
If $p \not = 3$, 
the other points of period $n$ can be identified with pairs $(a,b) \in 
\ZZ/p^k\ZZ \times \ZZ/m\ZZ$ such that $a \not = 0$ if $b$ has order $1$ or $3$,
up to equivalence $(a,b) \equiv (-a,-b)$.
Modulo $p$, two pairs $(a_1,b_1)$ and $(a_2,b_2)$ coalesce if and only if $b_1 = \pm b_2$.
Thus the ramification index is $e=(p^k-1)/2$ when $b=0$ and is $e=p^k-1$ when $b$ has order $3$ and is $p^k$ otherwise.

\item If $p=3 \mid (2^n+1)$, 
then the points of period $n$ of the form $\zeta+\zeta^{-1}$ where $\zeta\in\mu_{2^n+1}\cong\ZZ/3\ZZ\times\ZZ/m\ZZ$ can be identified with pairs 
$(a,b) \in 
\ZZ/3\ZZ \times \ZZ/m\ZZ$ such that $b \not = 0$, up to equivalence $(a,b) \equiv (-a,-b)$.
Thus the ramification index is 3. The other points of period $n$ remain distinct modulo $p$.
\end{enumerate} 
\end{proof}

\begin{ex}
When $n=5$ and $c=-2$, then $\zeta \in (\mu_{31} \setminus\{1\}) \cup (\mu_{33}\setminus \mu_3)$.  
When $p=11$, there is no collision of the roots of unity in $\mu_{31}$ but there is collision of the roots of unity in $\mu_{33}$.
Above $\bar{c}=-2$ in $\ol{Y}_1(5)$, there is one point of ramification index 5, one point of ramification index 10, and 15 unramified points.
The next result shows that above $
\bar{c}=-2$ in 
$\ol{Y}_0(5)$, there is one point with ramification index $2$ and $4$ unramified points.
\end{ex}

\begin{lemma}\label{LX0ramindex}
Suppose $n$ is prime and $2^n \pm 1 = p^k m$ with $p \nmid m$.
Then $\bar{\pi}_0:\overline{Y}_0(n)\to {\mathbb A}^1$ has the following ramification indices at $\bar{c}=0$ and $\bar{c}=-2$: \begin{enumerate} \item If $p\mid (2^n-1)$, then there are $1+\frac{m-1}{n}$ points of $\overline{Y}_0(n)$ above $\bar{c}=0$, one with ramification index $\frac{p^k-1}{n}$ and the other $\frac{m-1}{n}$ points with ramification index $p^k$.
\item If $p\mid (2^n-1)$, then there is one point in $\ol{Y}_0(n)$ above $\bar{c}=-2$ with ramification index $\frac{p^k-1}{2n}$, there are $\frac{m-1}{2n}$ points of ramification index $p^k$, and the remaining points are unramified.
\item If $p\mid (2^n+1)$ and $p\neq 3$, then there is one point in $\ol{Y}_0(n)$ above $\bar{c}=-2$ with ramification index $\frac{p^k-1}{2n}$, there is one point with ramification index $\frac{p^k-1}{n}$, there are $\frac{m-3}{2n}$ points with ramification index $p^k$, and the remaining points are unramified.

\item If $p\mid (2^n+1)$ and $p=3$, then there are $\frac{m-1}{2n}$ points in $\ol{Y}_0(n)$ above $\bar{c}=-2$ with ramification index 3, and the remaining points are unramified.
\end{enumerate} \end{lemma}

\begin{proof} Let $\eta$ be a ramification point of $\pi_1$ with ramification index $\eps$.
Let $\eta'$ be the image of $\eta$ in $\ol{Y}_0(n)$ and let $\eps'$ be its ramification index.
Since $n$ is prime, either $\eps'=\eps$ or $\eps'=\eps/n$. 
If $\eta$ is the unique point with ramification index $\eps$ in a fiber of $\pi_1$ then $\eps'=\eps/n$.
If $n \nmid \eps$, then $\eps'=\eps$.
Parts (1)--(4) then follow from Lemma~\ref{Lramindex}). 
\end{proof}

\begin{ex}
Let $n=11$, then $R_0=\mu_{2047} \setminus\{1\}$.  
Note that $R_0 \simeq {\mathbb Z}/23 \times {\mathbb Z}/89 \setminus \{(0,0)\}$, which has order $186 \cdot 11$, so $\nu(11)/11 = 186$.
The doubling map has 2 orbits of length 11 on $({\mathbb Z}/23)^\times$ and 
$8$ orbits of length 11 on  $({\mathbb Z}/89)^\times$.
So, $R_0$ has 8 orbits in 
$0 \times ({\mathbb Z}/89)^\times$ and $2$ orbits in $({\mathbb Z}/23)^\times$ 
and $176=8 \cdot 22$ other orbits containing pairs $(a,b)$ where $ab \not = 0$.

For $p=23$, we index the $176 + 8$ orbits by $(a,b_0)$ where $a \in {\mathbb Z}/23$ and $b_0$
is the smallest value in one of the $8$ orbits.
Modulo 23, $\ol{Y}_0(11)$ has 8 points with wild ramification 
of order $23$ (from the collapsing of $a$) and $1$ point with ramification of order $2$.  

For $p=89$, we index the $176+2$ orbits by $(a_0, b)$ where $b \in {\mathbb Z}/89$ 
and $a_0$ is the smallest value in one of the two orbits.
Modulo 89, $\ol{Y}_0(11)$ has $2$ points with wild ramification of order $89$ (from the collapsing of $b$) and $1$ point with ramification 
of order $8$.  
\end{ex}

\subsection{The tame case}

If $n$ is prime and $p \mid (2^n \pm 1)$, then the ramification is tame in Lemma \ref{LX0ramindex} items 1,2,4 exactly when $m=1$ and in item 3 
when $m=3$.
By Catalan's Conjecture, the condition $m=1$ occurs if and only if $n$ is prime and $p=2^n-1$ is a Mersenne prime
or when $n=p=3$. 
There are currently 49 known Mersenne primes.

When $m=1$ and $p=2^n-1$, then there is a unique $p^k\text{th}$ root of unity in $\ol{\FF}_p$.
So all $p^k-1=2^n-2$ points of $Y_1(n)$ above $c=0$ collide modulo $p$. Thus 
$t_0=1$ and $e_0=p^k-1$. Since $p\nmid e_0$, ramification at $\bar{c}=0$ of $\bar{\pi}_1$ is tame and so ramification at $\bar{c}=0$ of $\bar{\pi}_0$ is also tame.  The contribution to the degree of the 
discriminant of $\bar{\pi}_0$ is 
$\bar{\rho}_0=\frac{p^k-1}{n}-1=\frac{\nu(n)}{n}-1$.

\subsection{The wild case} \label{Swild}

In the wild case, we expect that the lower bound 
in Lemma \ref{Lwildcont} for $\bar{\rho}_{0}$ (resp.\ $\bar{\rho}_{-2}$) typically equals the upper bound $\rho_0$ (resp.\ $\rho_{-2}$).
In the data we computed, this was true except when
$n=6$ and $p=7$.
When the lower bound and upper bound are equal, 
then $\ol{Y}_0(n)$ is smooth
above $\bar{c}=0$ (resp.\ $\bar{c}=-2$).

We proved that this equality holds under 
certain algebraic and numerical conditions.  
First, under certain numerical conditions on $n$ and $p$ and when $k=1$, we proved a generalization of \cite[Proposition 1.3, Lemma 3.2]{Bouw11} to show that the contribution to the degree of the discriminant at a wildly ramified point $\eta$ for the {\it non-Galois} cover $\bar{\pi}_0$
is determined by the upper jump $\sigma$ of the ramification filtration of a pre-image of $\eta$ in the Galois closure of $\bar{\pi}_0$.
Next, under a certain algebraic condition on $\zeta \in R_0$ modulo a prime of $\ZZ[\zeta]$ above $p$, 
we determined the upper jump $\sigma$ for a point $x$ of formal period $n$ specializing to a wildly ramified point 
of $\bar{\pi}_0$. 
Specifically, we proved that the upper jump is as small as possible, namely $\sigma = 1/(p^k-1)$.
This proof relied on an algebraic analysis of valuations of differences of roots of $\Psi_n(x,c)$.

This material about the wild case is not included here since it is lengthy and not used in the proof of Theorem~\ref{corol:Examples}.
The limiting factor in extending Theorem~\ref{corol:Examples} to larger $n$ is the difficulty in computing $\Delta_{n,n}$ modulo $p$.
In Example \ref{Eweird}, we briefly describe the unusual case $n=6$ and $p=7$.

\subsection{Proof of Theorem \ref{corol:Examples}}

The proof of Theorem \ref{corol:Examples} is now a long calculation.
We calculate $\Delta_{n,n}(c)$ modulo $p$. (See the ancillary file \verb|delta_nn.txt| for the explicit polynomials $\Delta_{n,n}(c)$ for $n=5,6,7,8$, as well as $\Delta_{11,11}(c)$ modulo $p$ for $p=3, 23, 89, 683$.)
Using Lemma \ref{Lhowtochar0}, this determines $\rho_0$ and $\rho_{-2}$.
Using \eqref{EG02}, we check there is no singularity away from $\bar{c}=0,-2$.

If $n \in \{5, 7, 11\}$, we use Lemma \ref{LX0ramindex} to compute the ramification indices
for $\ol{Y}_0(n)$ above $\bar{c}=0,-2$ and determine if the ramification is tame or wild (if $n \in \{6, 8\}$ we make a similar one-off calculation).  
When tame, Lemma \ref{Ltamecont} determines $\bar{\rho}_0$
and $\bar{\rho}_{-2}$.
When wild, we use Lemma \ref{Lwildcont} to determine 
a lower bound for $\bar{\rho}_{0}$ and $\bar{\rho}_{-2}$.
We summarize these calculations in Table \ref{tab:primen}.
By Lemma~\ref{prop:badat0-2}, there is good reduction for all pairs $(n,p)$ in Table \ref{tab:primen}.

\begin{table} \label{Tred1}
\caption{Reduction data for $n$ prime. Here $e$ denotes the ramification indices, either above $\bar{c}=0$ or $\bar{c}=-2$; points with ramification index 1 are omitted, and exponents denote multiplicity. So, for example, $e=23^8,2$ means that there are 8 points with ramification index 23 and 1 point with ramification index 2, and all other points have ramification index 1.}
\label{tab:primen}
\begin{tabularx}{0.95\textwidth}{ccccccccccl}
\toprule
$n$ & $p$ & $\rho_0$ & $\bar{\rho}_0$ & $e$ & $0$ tame? & $\rho_{-2}$ & $\bar{\rho}_{-2}$ & $e$ & $-2$ tame? & \shortstack[l]{Other sing?} \\ \hline\hline 
5 & 3 & 0 & 0 & --- & yes & 3 & 3 & 3 & no & no \\ 
5 & 11 & 0 & 0 & --- & yes & 1 & 1 & 2 & yes & no \\ 
5 & 31 & 5 & 5 & 6 & yes & 2 & 2 & 3 & yes & no \\ \hline 
7 & 3 & 0 & 0 & --- & yes & 9 & 9 & $3^3$ & no & no \\ 
7 & 43 & 0 & 0 & --- & yes & 7 & 7 & $6,3$ & yes & no \\ 
7 & 127 & 17 & 17 & 18 & yes & 8 & 8 & 9 & yes & no \\ \hline 
11 & 3 & 0 & 0 & --- & yes & 93 & 93 & $3^{31}$ & no & no \\ 
11 & 23 & 185 & 185 & $23^8,2$ & no & 92 & 92 & $23^4$ & no & no \\ 
11 & 89 & 185 & 185 & $89^2,8$ & no & 92 & 92 & $89,4$ & no & no \\ 
11 & 683 & 0 & 0 & --- & yes & 91 & 91 & $62,31$ & yes & no \\
\bottomrule
\end{tabularx}
\end{table}

We give some details about the proof in a few sample cases.

\begin{ex} Let $n=5$ and $p=31$.
Then $\Delta_{5,5}(c)\equiv c^5 (c+2)^2 g(c)\bmod{31}$, where $g(0), g(-2) \not\equiv 0\bmod{31}$.
Thus $\rho_0=5$ and $\rho_{-2}=2$.
Furthermore, $\Gamma_{5,31}(c)=c^4(c+2)$, so the branch points remain distinct modulo $31$ unless they specialize to $\bar{c}=0,-2$.

By Lemma~\ref{LX0ramindex} (1), $\bar{\pi}_0:\overline{Y}_0(5)\to \PP^1$ is tamely ramified at $\bar{c}=0$, with ramification index $\frac{p-1}{n}=6$. Thus $\bar{\rho}_0=5$. 
If $\bar{c}=-2$, then $\bar{\rho}_{-2}=\frac{30}{2\times 5}-1=2$.
By Lemma~\ref{prop:badat0-2}, $\overline{Y}_0(5)$ has no singularity above $\bar{c}=0$
or $\bar{c}=-2$ modulo $31$.
Thus $Y_0(5)$ has good reduction modulo 31. 
\end{ex}

\begin{ex} When $n=7$ and $p=127$,
then $\Delta_{7,7}(c)\equiv c^{17}(c+2)^8 g(c)\bmod{127}$, where
$g(c)$ has simple roots not congruent to $0$ or $-2$ modulo $127$.  Thus $\rho_0=17$ and 
$\rho_{-2} = 8$.
Also $\bar{\rho}_0=\frac{127-1}{7}-1=17$ and $\bar{\rho}_{-2}=\frac{127-1}{2\times 7}-1=8$. So $Y_0(7)$ has good reduction modulo $127$. 
\end{ex}

\begin{ex}
Let $n=7$ and $p=43$.  Then $p \mid (2^n+1)$ with $m=3$.  Note that $43^{47} \mid D_{7,7}$.
Above $\bar{c}=-2$, there is one point with ramification index $3$ and one with ramification index $6$ so $\bar{\rho}_{-2}=7$.
Also $\rho_{-2}=7$, so $\ol{Y}_0$ is smooth above $c=-2$.  
\end{ex}

\subsection{Composite $n$} \label{Scomposite}

A similar analysis is possible when $n$ is composite.
When $n=6,8$,
we compute the data in Table~\ref{tab:compn}. 
When $n=6$, this gives a new proof of a result in 
\cite[\S 2]{Stoll}.

\begin{table} \label{Tred2}
\caption{Reduction data for $n$ composite}
\label{tab:compn}
\begin{tabularx}{0.95\textwidth}{ccccccccccl}
\toprule
$n$ & $p$ & $\rho_0$ & $\bar{\rho}_0$ & $e$ & $0$ tame? & $\rho_{-2}$ & $\bar{\rho}_{-2}$ & $e$ & $-2$ tame? & \shortstack[l]{Other sing?} \\ \hline\hline 
6 & 3 & 6 & 6 & $4^2$ & yes & 3 & 3 & 4 & yes & $c^2+2c+2$, $\rho=2$ \\ 
6 & 5 & 0 & 0 & --- & yes & 5 & 5 & 5 & no & no \\ 
6 & 7 & 9 & 9 & $7,2$ & no & 2 & 2 & 3 & yes & no \\ 
6 & 13 & 0 & 0 & --- & yes & 3 & 3 & 4 & yes & no \\ \hline 
8 & 3 & 30 & 30 & $3^{10}$ & no & 12 & 12 & $3^4$ & no & no \\ 
8 & 5 & 30 & 30 & $5^6$ & no & 12 & 12 & $5^2,2^2$ & no & no \\ 
8 & 17 & 25 & 25 & $8^3,4,2$ & yes & 11 & 11 & $8,4,2$ & yes & no \\ 
8 & 257 & 0 & 0 & --- & yes & 15 & 15 & 16 & yes & no \\
\bottomrule
\end{tabularx}
\end{table}

\subsubsection{The case $n=6$}  

The degree of $\pi_0:Y_0(6)\to {\mathbb A}^1$ is 9.
By Proposition~\ref{prop:Pbadimpliesdisc} and Tables~\ref{table:data1-7}--\ref{table:data8}, there are 9 odd primes dividing $D_{n,n}$, 
which are possibly primes of bad reduction for $Y_0(6)$.
The methods of this section are useful for analyzing four of these, namely $p=3,5,7,13$.

By Lemma \ref{Lpointformalperiodn},
the points of $Y_1(6)$ above $c=0$ are $R_0=\mu_{63}-(\mu_7\cup\mu_3)$, so that $|R_0|=54$. 
We identify the elements of $R_0$ with pairs $(a,b) \in \ZZ/9\ZZ \times \ZZ/7\ZZ$, excluding $(0,b)$ for $b \in \ZZ/7\ZZ$ and $(3,0)$ and $(6,0)$. The squaring map has 9 orbits over $c=0$, represented by:
\[A = (1,0), \ B_i = (1, i) \ {\rm for} \ 1 \le i \le 6, \ C_1=(3,1), \ C_2 = (3,5).\]
For $C_1$ and $C_2$, note that 1 is a quadratic residue and 5 is a quadratic nonresidue modulo 7. 

\begin{ex} \label{ex:n=6p=3} Let $p=3$. Then \eqref{EG02} is $\Gamma_{6,3}(c)=c^5(c+2)^2(c^2+2c+2)$.
Modulo 3, at $\bar{c}=0$, orbit $A$ stays disjoint from the other orbits, orbit $C_1$ collides with the $B_i$ orbits with $\left(\frac{i}{7}\right)=+1$, and orbit $C_2$ collides with the $B_i$ orbits with $\left(\frac{i}{7}\right)=-1$. Thus the ramification type is $(4,4,1)$. A similar analysis shows that the ramification type at $\bar{c}=-2$ is $(4,1^5)$.

There is an additional complication here, namely that there are points other than $c=0$ and $c=-2$ which contribute to $v_3(D_{6,6})$, coming from the additional factor of $\Gamma_{6,3}$. These are the roots of $c^2+2c+2$, where $\rho=2$ in characteristic 0. To compute $\bar{\rho}$ at these points, we factor the dynatomic polynomial $\Phi_6(x)$ over $\FF_3[c]$ when $c$ is a root of $c^2+2c+2$: 
\begin{multline*}(x^6+2cx^5+x^4+(c+1)x^3+(2c+1)x^2+(2c+2)x+c+1)^2 \cdot \\ (x^6+(2c+1)x^5+(2c+1)x^4+(2c+2)x^3+2cx^2+2cx+2c+2)^2\cdot g(x),\end{multline*} where $g$ has simple roots which do not collide with those of the other two factors listed. If $x$ is a root of the first square factor, then $x^2+c$ is a root of the second square factor, and vice versa. It follows that $\bar{\pi}_0$ has two points of ramification index $2$ above each of these values of $c$. This matches the degree in characteristic 0, so $Y_0(6)$ has good reduction modulo $p=3$.
\end{ex}

\begin{ex}
When $p=5$, then $\bar{c}=0$ is not a branch point.  There is one ramified point above $\bar{c}=-2$, coming from the orbits of $\zeta+\zeta^{-1}$ with $\zeta\in\mu_{65}$, with $e'=5$. Thus $\bar{\rho}_{-2}\ge 5$, and since $\rho_{-2}=5$ and $\bar{\rho}_{-2}\le\rho_{-2}$, it follows that $\bar{\rho}_{-2}=5$. Thus $Y_0(6)$ has good reduction at $p=5$. (On the other hand, $Y_1(6)$ has bad reduction modulo $p=5$.)
\end{ex}

\begin{ex} \label{Eweird}
When $p=7$, then $\Delta_{6,6}(c)\equiv c^9(c+2)^2h(c)\pmod{7}$.  One can check that $h(c)$
is squarefree modulo 7 by computing 
$\Gamma_{6,7}(c)$ from \eqref{EG02}.
Thus it suffices to check that there is no singularity above $\bar{c}=0$ and $\bar{c}=-2$.

When $p=7$ and $c=0$, then 
orbit $A$ and all the $B$ orbits collapse to a point, and the $C$ orbits collapse to a different point. The ramification type is $(7,2)$, so ramification is wild.

At the wildly ramified point, the minimal possible contribution to $\bar{\rho}_0$ is 7. However, the contribution is actually 8 because the algebraic condition mentioned in \S~\ref{Swild} is not satisfied for some $\zeta\in R_0$, namely $\zeta$ a primitive $21$-st root of unity. 
We compute that $\sigma=\frac{1}{3}$ 
so the wild contribution to the degree of the discriminant is $(p-1)(1 + \sigma) = 8$. The tame contribution is $1$.  Thus $\bar{\rho}_0=9$, so $Y_0(6)$ has no singularity over $\bar{c}=0$ when $p=7$.

If $p=7$ and $c=-2$ then $\rho_{-2} = 2$. 
The orbits from the (real parts of the) $65$-th roots of unity do not collide modulo 7. The orbits above $c=-2$ coming from the $63$-rd roots of unity are closely related to the orbits above $c=0$, but the orbit of A above $c=0$ is not included in $R_{-2}$.
The 6 B orbits turn into 3 orbits in pairs, which collide. The 2 C orbits turn into one orbit. 
Thus there is one point with ramification index $3$, yielding $\bar{\rho}_{-2}=2$, as desired. 
Thus $Y_0(6)$ has good reduction modulo $p=7$. 
\end{ex}

\begin{ex}
Let $n=6$ and $p=13$. Then $\bar{c}=0$ is unbranched. There is one ramified point above $\bar{c}=-2$, coming from $\zeta+\zeta^{-1}$ with $\zeta\in\mu_{65}$, with $e=4$. Thus $\bar{\rho}_{-2}=3$
and $Y_0(6)$ has good reduction modulo $p=13$. 
\end{ex}

\subsubsection{The case $n=8$}

The degree of $\pi_0:Y_0(8)\to {\mathbb A}^1$ is $\frac{2^8-2^4}{8}=30$.
By Proposition~\ref{prop:Pbadimpliesdisc} and Tables~\ref{table:data1-7}--\ref{table:data8}, there are many possible primes of bad reduction for $Y_0(8)$.  We use the methods of this section to show that $Y_0(8)$ has good reduction modulo $p=3,5,17,257$ below.

When $n=8$ and $p=3$, there are 10 points of $\ol{Y}_0(8)$ above $\bar{c}=0$, each of which has $e=3$. Thus all ramification points are wild, so $\bar{\rho}_0\ge 3\times 10=30$. Equality occurs because $\rho_0=30$, and $\bar{\rho}_0\le\rho_0$ by
Lemma~\ref{prop:badat0-2}. Above $\bar{c}=-2$, there are 6 points coming from orbits of points of the form $\zeta+\zeta^{-1}$ with $\zeta\in\mu_{255}$. (There are others coming from $\zeta\in\mu_{257}$, but these are unramified and thus do not contribute to $\bar{\rho}_{-2}$.) Four of these points have $e=3$, and the other two have $e=1$. Thus $\bar{\rho}_{-2}\ge 4\times 3=12$, and since $\rho_{-2}=12$, equality occurs.

The case $p=5$ is similar. Above $\bar{c}=0$, there are 6 points, each of which has ramification index 5, so $\bar{\rho}_0\ge 30$, and again equality occurs by comparison with $\rho_0$. Above $\bar{c}=-2$, there are 4 points coming from orbits of points of the form $\zeta+\zeta^{-1}$ with $\zeta\in\mu_{255}$, two with $e=5$ and two with $e=2$. Thus $\bar{\rho}_{-2}\ge 5\times 2+1\times 2=12$, and again there is equality.

Let $n=8$ and $p=17$.
Then $\Delta_{8,8} \equiv c^{25}(c+2)^{11}g(c) \bmod 17$, where $g(c)$ has simple roots and $g(0),g(-2)\not\equiv 0\pmod{17}$.
Note that $R_0 = \mu_{2^8-1} \setminus \mu_{2^4-1} = (\mu_{17} \times \mu_5 \times \mu_3) 
\setminus (\{1\} \times \mu_5 \times \mu_3)$.  
The orbits of the squaring map on $R_0$
have lengths $8$, $8$, $8$, $4$, $2$.  Thus 
$\bar{\rho}_0 = 25$ and there is no singularity above $\bar{c}=0$. When $\bar{c}=-2$, then all ramification again comes from those $\zeta+\zeta^{-1}$ with $\zeta\in\mu_{255}$. There are 3 points of this form above $\bar{c}=-2$, with ramification index 8, 4, and 2, for a contribution to $\bar{\rho}_{-2}$ of $7+3+1=11$.

When $n=8$ and $p=257$, then $\bar{c}=0$ is unramified, but $\bar{c}=-2$ is ramified, with all contribution coming from $\zeta+\zeta^{-1}$ with $\zeta\in\mu_{257}$. There is only one ramification point above $\bar{c}=-2$, with ramification index 16, so $e=16$ and $\bar{\rho}_{-2}=15$.

\bibliographystyle{alpha}
\bibliography{dyna.bib}

\appendix
\section{Data} \label{Sdata}

The following is a complete list of odd primes of bad reduction of $Y_1(n)$ with $f(x,c) = x^2 + c$ for $6\leq n \leq8$: \begin{itemize}
\item $Y_1(6):3, 5, 67, 8029187$. (Note that $Y_0(6)$ has good reduction modulo $3$, $5$, and $67$).
\item $Y_1(7): 7, 84562621221359775358188841672549561$.
\item $Y_1(8): 1567, 18967471, 120664513, 268015967, 1751050452798629934784579,$ \\
$48452315131278500437597584581$.
\end{itemize}

 \subsection{Computation}
 
We would like to thank Michael Stoll for suggesting the following method to find candidates for primes of bad reduction.
Let $F(x,c)=\Phi_n(x,c)$ and write $F_x$ and $F_c$ for its two partial derivatives. The primes of bad reduction are those primes $p$ such that $F, 
F_x, F_c$ vanish at a common point mod $p$. A necessary condition for this is that the three resultants
  \begin{equation*}
R_1 = \Res_x(F, F_x), \quad R_2 = \Res_x(F, F_c), \quad R_3 = \Res_x(F_x, F_c)
\end{equation*}
have a common root mod $p$. 

By Theorem \ref{discriminantOfPhi}, $R_1$ factors as $\pm  \Delta_{n,n}(c)^n \prod_{\substack{d\mid n \\ d \neq n}} \Delta_{n,d}(c)^{n-d}$. In the cases that we computed, $R_2$ was irreducible. To find a list of candidate primes of bad reduction, we computed 
\begin{equation*}
\gcd(\Res(f, g), \Res(f, R_2))
\end{equation*}
for each irreducible factor $f$ of $R_3$ and each irreducible factor $g$ of $R_1$. Any bad prime $p > n$ must divide one of these greatest common divisors. For $n\leq 8$, the computations terminate in a reasonable time, giving a set $L_n$ of candidate primes of bad reduction for $Y_1(n)$. One can then check for each prime $p\in L_n$ whether $Y_1(n)$ is singular mod $p$. For $5\leq n\leq 8$, the set $L_n$ consists of precisely the bad primes greater than $n$; on the other hand, the set $L_4$ contains $107$, which is not a bad prime for $Y_1(4)$. 

\subsection{Tables}
In Tables~\ref{table:data1-7}--\ref{table:data8}, we describe the factorizations of $D_{n,d}=\disc(\Delta_{n,d})$ and of $\Res(\Delta_{n,d},\Delta_{n,e})$. The notation $(e,d)$ stands for $\disc(\Delta_{n,d})$ when $d=e$, and for $\Res(\Delta_{n,d},\Delta_{n,e})^2$ when $d\neq e$. The discriminants and resultants for $n\ge 7$ are large and could not be completely factored within a reasonable amount of time. All prime factors less than $10^{10}$ are listed, as well as some larger factors we were able to find. The odd bad primes for $X_0(n)$ are underlined, and the odd bad primes  for $X_1(n)$ are in boldface.

\begin{table}[h]
\begin{minipage}{.4\linewidth}
\caption{$n=4$}
\label{table:data1-7}
\centering
\begin{tabular}{l l}
\toprule

$(e,d)$ & Factorization \\ \midrule
$(1,1)$  &  $  2^{8} $ \\
$  (1,2)$  &  $  2^{14}\cdot 5^{2} $ \\
$  (1,4)$  &  $  2^{32}\cdot 5^{2}\cdot 17^{4} $ \\
$  (2,2)$  &  $1 $ \\
$  (2,4)$  &  $  2^{16}\cdot 5^{4} $ \\
$  (4,4)$  &  $  2^{16}\cdot 3^{9}$ \\
\bottomrule
\end{tabular}
\end{minipage}
\begin{minipage}{.5\linewidth}
\caption{$n=5$}
\label{table:data1-7.5}
\begin{tabular}{l l}
\toprule
$(e,d)$ & Factorization \\ \midrule
$  (1,1)$  &  $  2^{24}\cdot \mathbf{5}^{7}\cdot 11^{2} $ \\
$  (1,5)$  &  $  2^{232}\cdot 11^{6}\cdot 31^{18}\cdot 86131^{2} $ \\
$  (5,5)$  &  $  2^{274}\cdot 3^{12}\cdot 31^{27}\cdot \underline{\mathbf{3701}}^{1}\cdot 4217^{3} $  \\
\bottomrule
\end{tabular}
\end{minipage}
\end{table}

\begin{table}[h]
\caption{$n=6$}
\label{table:data1-7.6}
\begin{tabularx}{\textwidth}{l l}
\toprule
$(e,d)$ & Factorization \\ \midrule
$  (1,1)$  &  $  2^{4}\cdot \mathbf{3}^{1} $ \\
$  (1,2)$  &  $  2^{20}\cdot \mathbf{3}^{6}\cdot 13^{2} $ \\
$  (1,3)$  &  $  2^{40}\cdot \mathbf{3}^{4}\cdot 7^{2} $ \\
$  (1,6)$  &  $  2^{192}\cdot \mathbf{3}^{12}\cdot 7^{4}\cdot 13^{6}\cdot 211^{4}\cdot 68700493^{2} $ \\
$  (2,2)$  &  $  2^{4}\cdot \mathbf{3}^{1} $ \\
$  (2,3)$  &  $  2^{28}\cdot \mathbf{3}^{4}\cdot 157^{2} $ \\
$  (2,6)$  &  $  2^{204}\cdot \mathbf{3}^{12}\cdot 7^{18}\cdot 13^{12}\cdot 79^{2} $ \\
$  (3,3)$  &  $  2^{12}\cdot \mathbf{3}^{4}\cdot \mathbf{5}^{2}\cdot \mathbf{67}^{1} $ \\
$  (3,6)$  &  $  2^{296}\cdot \mathbf{3}^{66}\cdot 7^{6}\cdot 239^{4}\cdot 409^{2}\cdot 3331^{2} $ \\
$  (6,6)$  &  $  2^{956}\cdot \mathbf{3}^{91}\cdot \mathbf{5}^{25}\cdot 7^{66}\cdot 13^{8}\cdot 29^{3}\cdot 61^{2}\cdot \underline{\mathbf{8029187}}^{1}\cdot 55218797^{3}\cdot 47548578843011867^{2} $ \\
\bottomrule
\end{tabularx}
\end{table}
\begin{table}[h]
\caption{$n=7$}
\label{table:data1-7.7}
\begin{tabularx}{\textwidth}{l l}
\toprule
$(e,d)$ & Factorization \\ \midrule
$ (1,1)$ & $ 2^{60}\cdot \mathbf{7}^{11}\cdot 29^{2}\cdot 43^{2}$ \\
$ (1,7)$ & $ 2^{1668}\cdot 43^{32}\cdot 127^{54}\cdot 987211^{4} $ \\  & $\quad\cdot \text{\scriptsize{ 617019606746943232762656703658773289881979839340269069795489872959090140927}}^{2}$\\ 
$ (7,7)$ & $ 2^{7712}\cdot 3^{81}\cdot 43^{47}\cdot 127^{351}\cdot 10273^{2}\cdot 194003^{2}$\\
 & $\quad\cdot \underline{\mathbf{84562621221359775358188841672549561}}^{1}$\\  & $\quad\cdots\text{other factors (good primes)}$ \\
\bottomrule
\end{tabularx}
\end{table}

\begin{table}[h]
\caption{$n=8$}
\label{table:data8}
\centering
\begin{tabularx}{\textwidth}{l l}

\toprule

$(e,d)$ & Factorization \\ \midrule
 $(1,1)$ & $  2^{40} \cdot 3^{2} $ \\
$(1,2)$ & $  2^{100} \cdot 17^{8} $ \\
$(1,4)$ & $  2^{264} \cdot 7^{8} \cdot 17^{16} \cdot 1321^{4} $ \\
$(1,8)$ & $  2^{4096} \cdot 17^{156} \cdot 257^{34} \cdot 593^{4} \cdot 12073^{8}\cdot 158091133929843713^{4}  $\\
 &$\quad\cdot (\text{102-digit prime})^{4} $ \\
$(2,2)$ & $  2^{6} $ \\
$(2,4)$ & $  2^{132} \cdot 17^{16} $ \\
$(2,8)$ & $  2^{1028} \cdot 17^{152} \cdot 53^{8} \cdot 248117^{4}$\\
	 & $\quad\cdot 18205929282889572644368241959974766492471262272780397^{4} $ \\
$(4,4)$ & $  2^{80} \cdot 3^{2} \cdot 5^{2} \cdot 17^{6} \cdot \mathbf{18967471}^{1} $ \\
$(4,8)$ & $  2^{3072} \cdot 3^{48} \cdot 5^{52} \cdot 17^{180} \cdot 22129^{2}  $\\
 & $\quad \cdot 6200609^{2} \cdot 7408189^{2}\cdot 27756089^{4} \dots\text{other factors}$ \\
$(8,8)$ & $  2^{27528} \cdot 3^{1044} \cdot 5^{1052} \cdot 17^{738} \cdot 257^{224} \cdot  \underline{\mathbf{1567}}^{1} \cdot 17863^{
2} \cdot 1733999^{2} \cdot 8711621^{3}   $\\
	& $\quad\cdot 141715459^{2}\cdot \underline{\mathbf{120664513}}^{1} \cdot \underline{\mathbf{268015967}}^{1} \cdot \underline{\mathbf{1751050452798629934784579}}^{1}$\\
 & $\quad\cdot 
\underline{\mathbf{48452315131278500437597584581}}^{1}\cdots\text{other factors (good primes)}$ \\

\bottomrule
\end{tabularx}
\end{table}

\clearpage

\section{Monodromy data for $n=5$} \label{Tmono}

In Table~\ref{table:BP5}, we describe the monodromy data at the branch points of $\pi_1$ for $Y_1(5)$. As mentioned in \S \ref{Smonodromy}, sheets of the cover $\pi_1$ may be identified with sequences in $\{0,1\}^{\NN}$ of period 5; for convenience, we order the sheets 1 through 30 according to Table~\ref{table:label5}. The monodromy action at each branch point is then obtained from its kneading sequences via Proposition~\ref{prop:loops}.

There are fifteen branch points for $\pi_1$, and there are exactly two parameter rays that land at each branch point. We therefore use the landing rays to differentiate the branch points in Table~\ref{table:BP5}. The ``bifurcation type'' of a branch point $c$ is $(5,1)$ or $(5,5)$ according to whether $c$ is a satellite parabolic parameter (in which case a $5$-cycle collapses to a $1$-cycle) or a primitive parabolic parameter (in which case two $5$-cycles collide). The $(5,5)$-branch points are precisely the finite branch points of $\pi_0 : Y_1(5) \to \AA^1$, and the data in Table~\ref{table:BP5} associated to such branch points is used to construct the monodromy graph $\Gamma(5)$ in Figure~\ref{fig:n5graph}.

Though we do not give the analogous data for $n = 7$, we illustrate $\Gamma(7)$ in Figure~\ref{n7graph}.






\begin{figure}[h] \begin{minipage}{.45\linewidth}
\captionof{table}{Ordering of sheets for $n = 5$}
\label{table:label5}
\centering
	\begin{tabular}{ccccc}
	1 & 2 & 3 & 4 & 5\\
	$\overline{00001}$ & $\overline{00010}$ & $\overline{00100}$ & $\overline{01000}$ & $\overline{10000}$ \\
	6 & 7 & 8 & 9 & 10\\
	$\overline{00011}$ & $\overline{00110}$ & $\overline{01100}$ & $\overline{11000}$ & $\overline{10001}$ \\
	11 & 12 & 13 & 14 & 15\\
	$\overline{00101}$ & $\overline{01010}$ & $\overline{10100}$ & $\overline{01001}$ & $\overline{10010}$ \\
	16 & 17 & 18 & 19 & 20\\
	$\overline{00111}$ & $\overline{01110}$ & $\overline{11100}$ & $\overline{11001}$ & $\overline{10011}$ \\
	21 & 22 & 23 & 24 & 25\\
	$\overline{01011}$ & $\overline{10110}$ & $\overline{01101}$& $\overline{11010}$& $\overline{10101}$ \\
	26 & 27 & 28 & 29 & 30\\
	$\overline{01111}$ & $\overline{11110}$ & $\overline{11101}$ & $\overline{11011}$ & $\overline{10111}$
	\end{tabular}
    \end{minipage}
    \begin{minipage}{.54\linewidth}
    \begin{tikzpicture}
\tikzset{vertex/.style = {shape=rectangle,minimum size=1em}}
\tikzset{edge/.style = {-, = latex'}}
\tikzset{dash/.style = {style=thick, style=dashed, -, = latex'}}
\tikzset{around/.style = {distance = 7cm, out = 0, in = 0, style=dashed, style=thick, -, = latex'}}
\tikzset{rightaround/.style = {distance = 3.5cm, out = 0, in = 0, style=dashed, style=thick, -, = latex'}}
\tikzset{mirror/.style = {distance = 3.5cm, out = 180, in = 180, style=dashed, style=thick, color=white, -, = latex'}}
\node[vertex] (1) at  (2, 0) {$\overline{10000}$};
\node[vertex] (2a) at  (0, 1.5) {$\overline{10100}$};
\node[vertex] (2b) at  (4, 1.5) {$\overline{11000}$};
\node[vertex] (3a) at (0, 3) {$ \overline{11010} $};
\node[vertex] (3b) at (4, 3) {$ \overline{11100} $};
\node[vertex] (4) at (2, 4.5) {$ \overline{11110} $};
\draw[edge] (1) to (2b);
\draw[rightaround] (1) to (4);
\draw[mirror] (1) to (4);
\draw[edge] (2a) to[bend left=15] (3a);
\draw[edge] (2a) to[bend right=15] (3a);
\draw[dash] (2a) to (3a);
\draw[edge] (2a) to (3b);
\draw[edge] (2b) to[bend left=15] (3b);
\draw[edge] (2b) to[bend right=15] (3b);
\draw[dash] (2b) to (3b);
\draw[edge] (3a) to[bend left=8] (4);
\draw[edge] (3a) to[bend right=8] (4);
\draw[edge] (3a) to (4);
\draw[edge] (3b) to[bend left=5] (4);
\draw[edge] (3b) to[bend right=5] (4);
\end{tikzpicture}
\caption{The monodromy graph $\Gamma(5)$.
Finite edges are solid, and infinite edges are dashed.}
\label{fig:n5graph}
\end{minipage}
\end{figure}

\begin{table}[h]
\caption{Branch points for $n = 5$}
\label{table:BP5}
\centering
\begin{tabular}{ccccc}

\toprule
Landing rays & $K(\theta)$ & Bifurcation & Multiplier & Cycle structure\\
($\theta = \square/31$) & & type & ($\lambda = e^{2\pi i \square}$) \\
\midrule
1, 2 & $1111*$ & (5,1) & 1/5 & (26 27 28 29 30)\\
3, 4 & $1110*$ & (5,5) & 1 & (16 26)(17 27)(18 28)(19 29)(20 30)\\
5, 6 & $1101*$ & (5,5) & 1 & (21 26)(22 27)(23 28)(24 29)(25 30)\\
7, 8 & $1100*$ & (5,5) & 1 & (6 16)(7 17)(8 18)(9 19)(10 20)\\
9, 10 & $1111*$ & (5,1) & 2/5 & (26 29 27 30 28)\\
11, 12 & $1010*$ & (5,5) & 1 & (11 23)(12 24)(13 25)(14 21)(15 22)\\
\hline
13, 18 & $1011*$ & (5,5) & 1 & (21 29)(22 30)(23 26)(24 27)(25 28)\\
14, 17 & $1001*$ & (5,5) & 1 & (11 16)(12 17)(13 18)(14 19)(15 20)\\
15, 16 & $1000*$ & (5,5) & 1 & (1 6)(2 7)(3 8)(4 9)(5 10)\\
\hline
19, 20 & $1010*$ & (5,5) & 1 & (11 23)(12 24)(13 25)(14 21)(15 22)\\
21, 22 & $1111*$ & (5,1) & 3/5 & (26 28 30 27 29)\\
23, 24 & $1100*$ & (5,5) & 1 & (6 16)(7 17)(8 18)(9 19)(10 20)\\
25, 26 & $1101*$ & (5,5) & 1 & (21 26)(22 27)(23 28)(24 29)(25 30)\\
27, 28 & $1110*$ & (5,5) & 1 & (16 26)(17 27)(18 28)(19 29)(20 30)\\
29, 30 & $1111*$ & (5,1) & 4/5 & (26 30 29 28 27)\\
\bottomrule
\end{tabular}
\end{table}

\begin{figure}
\begin{tikzpicture}
\tikzset{vertex/.style = {shape=rectangle,minimum size=1em}}
\tikzset{edge/.style = {-, = latex'}}
\tikzset{dash/.style = {style=thick, style=dashed, -, = latex'}}
\tikzset{around/.style = {distance = 7cm, out = 0, in = 0, style=dashed, style=thick, -, = latex'}}
\tikzset{leftaround/.style = {distance = 3.5cm, out = 180, in = 180, style=dashed, style=thick, -, = latex'}}
\tikzset{rightaround/.style = {distance = 3.5cm, out = 0, in = 0, style=dashed, style=thick, -, = latex'}}
\tikzset{littlearound/.style =  {distance = 1.3cm, out = 45, in = 315, style=dashed, style=thick, -, = latex'}}
\tikzset{mirror/.style = {distance = 7cm, out = 180, in = 180, style=dashed, style=thick, color=white, -, = latex'}}
\node[vertex] (1) at  (4, 0) {$\overline{1000000}$};
\node[vertex] (2a) at  (2, 2) {$\overline{1001000}$};
\node[vertex] (2b) at  (4, 2) {$\overline{1010000}$};
\node[vertex] (2c) at (6, 2) {$ \overline{1100000} $};
\node[vertex] (3a) at (0, 4) {$ \overline{1010100} $};
\node[vertex] (3b) at (2, 4) {$ \overline{1100010} $};
\node[vertex] (3c) at (4, 4) {$ \overline{1100100} $};
\node[vertex] (3d) at (6, 4) {$ \overline{1101000} $};
\node[vertex] (3e) at (8, 4) {$ \overline{1110000} $};
\node[vertex] (4a) at (0, 6) {$ \overline{1101010} $};
\node[vertex] (4b) at (4, 6) {$ \overline{1101100} $};
\node[vertex] (4c) at (6, 6) {$ \overline{1110010} $};
\node[vertex] (4d) at (2, 6) {$ \overline{1110100} $};
\node[vertex] (4e) at (8, 6) {$ \overline{1111000} $};
\node[vertex] (5a) at (2, 8) {$ \overline{1110110} $};
\node[vertex] (5b) at (4, 8) {$ \overline{1111010} $};
\node[vertex] (5c) at (6, 8) {$ \overline{1111100} $};
\node[vertex] (6) at (4, 10) {$ \overline{1111110} $};
\draw[edge] (1) to (2c);
\draw[around] (1) to (6);
\draw[mirror] (1) to (6);
\draw[edge] (2a) to (3b);
\draw[edge] (2a) to[bend left=5] (3c);
\draw[edge] (2a) to[bend right=5] (3c);
\draw[leftaround] (2a) to (5a);
\draw[edge] (2b) to[bend left=5] (3d);
\draw[edge] (2b) to[bend right=5] (3d);
\draw[edge] (2b) to (3e);
\draw[littlearound] (2b) to (5b);
\draw[edge] (2c) to[bend left=5] (3e);
\draw[edge] (2c) to[bend right=5] (3e);
\draw[rightaround] (2c) to (5c);
\draw[edge] (3a) to[bend left=12] (4a);
\draw[edge] (3a) to[bend right=12] (4a);
\draw[edge] (3a) to (4c);
\draw[edge] (3a) to[bend left=5] (4d);
\draw[edge] (3a) to[bend right=5] (4d);
\draw[dash] (3a) to (4a);
\draw[edge] (3b) to[bend left=2] (4e);
\draw[edge] (3b) to[bend right=2] (4e);
\draw[dash] (3b) to (4d);
\draw[edge] (3c) to[bend left=8] (4b);
\draw[edge] (3c) to[bend left=5] (4c);
\draw[edge] (3c) to[bend right=5] (4c);
\draw[dash] (3c) to[bend right=8] (4b);
\draw[edge] (3d) to (4e);
\draw[edge] (3d) to[bend left=2] (4d);
\draw[edge] (3d) to[bend right=2] (4d);
\draw[dash] (3d) to (4c);
\draw[edge] (3e) to[bend left=12] (4e);
\draw[edge] (3e) to[bend right=12] (4e);
\draw[dash] (3e) to (4e);
\draw[edge] (4a) to (5a);
\draw[edge] (4a) to[bend left=2] (5b);
\draw[edge] (4a) to[bend left=6] (5b);
\draw[edge] (4a) to[bend right=2] (5b);
\draw[edge] (4a) to[bend right=6] (5b);
\draw[edge] (4b) to[bend left=10] (5a);
\draw[edge] (4b) to[bend left=3] (5a);
\draw[edge] (4b) to[bend right=3] (5a);
\draw[edge] (4b) to[bend right=10] (5a);
\draw[edge] (4b) to[bend left=5] (5c);
\draw[edge] (4b) to[bend right=5] (5c);
\draw[edge] (4c) to[bend left=2] (5a);
\draw[edge] (4c) to[bend right=2] (5a);
\draw[edge] (4c) to[bend left=8] (5c);
\draw[edge] (4c) to[bend right=8] (5c);
\draw[edge] (4d) to[bend left=5] (5b);
\draw[edge] (4d) to[bend right=5] (5b);
\draw[edge] (4d) to (5c);
\draw[edge] (4e) to[bend left=3] (5c);
\draw[edge] (4e) to[bend left=10] (5c);
\draw[edge] (4e) to[bend right=10] (5c);
\draw[edge] (4e) to[bend right=3] (5c);
\draw[edge] (5a) to[bend left=3] (6);
\draw[edge] (5a) to[bend left=10] (6);
\draw[edge] (5a) to[bend right=10] (6);
\draw[edge] (5a) to[bend right=3] (6);
\draw[edge] (5b) to[bend left=20] (6);
\draw[edge] (5b) to[bend left=10] (6);
\draw[edge] (5b) to[bend right=10] (6);
\draw[edge] (5b) to[bend right=20] (6);
\draw[edge] (5b) to (6);
\draw[edge] (5c) to[bend left=5] (6);
\draw[edge] (5c) to[bend right=5] (6);
\end{tikzpicture}
\caption{The monodromy graph $\Gamma(7)$. Finite edges are solid, and infinite edges are dashed.}
\label{n7graph}
\end{figure}

\newpage

\phantom{Some phantom text to keep addresses on last page}

\newpage

\end{document}